\documentclass[12pt]{article}
\usepackage{mystyle} 

\renewcommand{\epsilon}{\varepsilon}

\newcommand{\dd}{\, \mathrm{d}}

\renewcommand{\P}{\mathbb{P}}

\newcommand{\PP}{\mathbb{P}}

\newcommand{\cA}{\mathcal{A}}
\newcommand{\cB}{\mathcal{B}}
\newcommand{\cC}{\mathcal{C}}
\newcommand{\cD}{\mathcal{D}}

\newcommand{\cF}{\mathcal{F}}
\newcommand{\cG}{\mathcal{G}}
\newcommand{\cH}{\mathcal{H}}

\newcommand{\cM}{\mathcal{M}}

\newcommand{\cS}{\mathcal{S}}

\newcommand{\fc}{\mathfrak{c}}

\newcommand{\sA}{\mathscr{A}}

\usepackage{our_comment}
\usepackage[a4paper,bindingoffset=0.2in,
            left=1in,right=1in,
            top=1in,bottom=1in,
            footskip=0.4in]{geometry}
\usepackage[style=numeric, backend=biber, doi=false, url=false, isbn=false, maxbibnames=99,
uniquename=false, giveninits]{biblatex}
\AtEveryBibitem{\clearlist{language}}
\renewbibmacro{in:}{%
  \ifentrytype{article}{}{\printtext{\bibstring{in}\intitlepunct}}}
\usepackage[nottoc]{tocbibind}
\newcommand{\tymchange}[1]{{#1}}

\begin{document}
\newcommand*{\TitleFont}{%
      \usefont{\encodingdefault}{\rmdefault}{b}{n}%
      \fontsize{16}{20}%
      \selectfont}

\title{\vspace{-0.3in} An extension of the stochastic sewing lemma and applications to fractional stochastic calculus}

\author{Toyomu Matsuda\thanks{EPFL, Bâtiment MA, 1015 Lausanne, Switzerland,
\href{mailto:toyomu.matsuda@epfl.ch}{toyomu.matsuda@epfl.ch}}
\and Nicolas Perkowski\thanks{FU Berlin, Arnimallee 7, 14195 Berlin, Germany,
\href{mailto:perkowski@math.fu-berlin.de}{perkowski@math.fu-berlin.de}}
}

\date{}
\maketitle

\begin{abstract}
We give an extension of Lê's stochastic sewing lemma [\emph{Electron. J. Probab.} 25: 1 - 55, 2020]. The stochastic sewing lemma proves convergence in $L_m$ of Riemann type sums $\sum_{[s,t] \in \pi} A_{s,t}$ for an adapted two-parameter stochastic process $A$, under certain conditions on the moments of $A_{s,t}$ and of conditional expectations of $A_{s,t}$ given $\mathcal F_s$.
Our extension replaces the conditional expectation given $\mathcal F_s$ by that given $\mathcal F_v$ for $v<s$, and it allows to make use of asymptotic decorrelation properties between $A_{s,t}$ and $\mathcal F_v$ by including a singularity in $(s-v)$.
We provide three applications for which Lê's stochastic sewing lemma seems to be insufficient.
The first is to prove the convergence of It\^o or Stratonovich approximations of
stochastic integrals along fractional Brownian motions under low regularity assumptions. The second is to obtain new representations of
local times of fractional Brownian motions via discretization.
The third is to improve a regularity assumption on the diffusion coefficient of a stochastic differential equation driven by a fractional Brownian motion
for pathwise uniqueness and strong existence.
\bigskip



\end{abstract}
\tableofcontents

\section{Introduction and the main theorem}
In analysis and probability theory, we often consider the convergence of sums
\begin{equation}\label{eq:riemann_sum}
  \sum_{[s, t] \in \pi} A_{s, t}.
\end{equation}
Here $\pi$ is a partition of an interval $[0, T]$ and we consider the limit of
\begin{equation*}
  \abs{\pi} \defby \max_{[s, t] \in \pi} \abs{t - s} \to 0.
\end{equation*}
For instance, if $A_{s, t} \defby f(s) (t - s)$, then we consider a Riemann sum approximation of $\int_0^T f(s) \dd s$,
and if $A_{s, t} \defby X_s (W_t - W_s)$, where $W$ is a Brownian motion and $X$ is an adapted process,
then we consider the Itô approximation of the stochastic integral $\int_0^T X_r \dd W_r$.

Gubinelli~\cite{gubinelli04}, {inspired by Lyons' results on almost multiplicative functionals in the theory of rough paths~\cite{lyons1998},} showed that if
\begin{equation}\label{eq:delta_A_s_u_t}
  \delta A_{s, u, t} \defby A_{s, t} - A_{s, u} - A_{u, t}, \quad 0\leq s < u < t \leq T,
\end{equation}
satisfies $\abs{\delta A_{s, u, t}} \lesssim \abs{t - s}^{1 + \epsilon}$ for some $\epsilon > 0$, then
the sums \eqref{eq:riemann_sum} converge. This result is now called the \emph{sewing lemma}, named so in the work
of Feyel and de La Pradelle~\cite{feyel06}. This lemma is so powerful that many applications and many extensions are known.
For instance, it can be used to define rough integrals, see \cite{gubinelli04} and the monograph \cite{friz2020} of Friz and Hairer.

When $(A_{s, t})_{s \leq t}$ is random and when we want to prove the convergence of the sums \eqref{eq:riemann_sum},
the above sewing lemma is often not sufficient. For instance, if $A_{s, t} \defby (W_t - W_s)^2$, the sums converge to the
quadratic variation of the Brownian motion. However, we only have
\begin{equation*}
  \abs{\delta A_{s, u, t}(\omega)} \lesssim_{\epsilon, \omega}
\abs{t-s}^{1 - \epsilon}
\end{equation*}
 almost surely
for every $\epsilon > 0$, and hence we cannot apply the sewing lemma.

Lê~\cite{khoa20} proved a stochastic version of the sewing lemma (\emph{stochastic sewing lemma}):
if a filtration $(\cF_t)_{t \in [0, T]}$ is given
such that
\begin{itemize}
  \item $A_{s, t}$ is $\cF_t$-measurable and
  \item for some $\epsilon_1, \epsilon_2 > 0$ and $m \in [2, \infty)$, we have for every $s < u < t$,
  \begin{equation}\label{eq:le_conditional}
    \norm{\expect[\delta A_{s, u, t} \vert \cF_s]}_{L_m(\P)} \lesssim \abs{t - s}^{1 + \epsilon_2},
  \end{equation}
  \begin{equation}\label{eq:le_one_half}
    \norm{\delta A_{s, u, t}}_{L_m(\P)} \lesssim \abs{t - s}^{\frac{1}{2} + \epsilon_1},
  \end{equation}
\end{itemize}
then the sums \eqref{eq:riemann_sum} converge in $L_m(\P)$. 
\tymchange{As usual, the Banach space $L_m(\mathbb{P})$ is equipped with the norm 
\begin{align*}
  \norm{X}_{L_m(\mathbb{P})} \defby \Big( \int_{\Omega} \abs{X}^m \dd \mathbb{P} \Big)^{\frac{1}{m}}.
\end{align*}
}
If $A_{s, t} \defby (W_t - W_s)^2$, then we have $\expect[\delta A_{s, u, t} \vert \cF_s] = 0$ and
\eqref{eq:le_one_half} is satisfied with $\epsilon_1 = \frac{1}{2}$.
Therefore, we can prove the convergence of \eqref{eq:riemann_sum} in $L_m(\P)$.
The stochastic sewing lemma has been already shown to be very powerful in the original work \cite{khoa20} of Lê,
and an increasing number of papers are appearing that take advantage of the lemma.

However, there are situations where Lê's stochastic sewing lemma seems insufficient.
For instance, consider
\begin{equation}\label{eq:germ_fbm}
  A_{s, t} \defby \abs{B_t - B_s}^{\frac{1}{H}},
\end{equation}
where $B$ is a fractional Brownian motion
with Hurst parameter $H \in (0, 1)$. It is well known that the sums \eqref{eq:riemann_sum} converge to
$c_H T$ in $L_m(\P)$. Although we have the estimate \eqref{eq:le_one_half}, we fail to obtain the estimate
\eqref{eq:le_conditional} unless $H = \frac{1}{2}$.

To get an idea on how Lê's stochastic sewing lemma should be modified for this problem,
observe the following trivial fact:
\begin{equation*}
  \expect[\delta A_{s, u, t}] = 0.
\end{equation*}
This suggests that we consider estimates that interpolate $\expect[\delta A_{s, u, t}]$ and
$\expect[\delta A_{s, u, t} \vert \cF_s]$. In fact, we can obtain the following estimates:
\begin{equation}\label{eq:fbm_p_variation_conditional}
  \norm{\expect[\delta A_{s, u, t} \vert \cF_v]}_{L_m(\P)}
  \lesssim_{H} \Big( \frac{t-s}{s-v} \Big)^{1-H} (t - s),
  \quad 0 \leq v < s < u < t \leq T.
\end{equation}
We can prove \eqref{eq:fbm_p_variation_conditional} for instance by applying Picard's result \cite[Lemma~A.1]{picard08} on the asymptotic independence of fractional Brownian increments,
or more directly by doing a similar calculation as in Section~\ref{sec:local_time}.
This discussion motivates the following main theorem of our paper.
\begin{theorem}\label{thm:generalized_stochastic_sewing}
  Suppose that we have a filtration $(\cF_t)_{t \in [0, T]}$ and a family of $\R^d$-valued random variables
  $(\sewing_{s, t})_{0 \leq s \leq t \leq T}$ such that
  $\sewing_{s, s} = 0$ for every $s \in [0, T]$ and such that $\sewing_{s, t}$ is $\cF_t$-measurable.
  We define $\delta A_{s, u, t}$ by \eqref{eq:delta_A_s_u_t}.
  Furthermore, suppose that there exist constants
  \begin{equation*}
    m \in [2, \infty), \quad \Gamma_1, \Gamma_2, M \in [0, \infty),
       \quad \alpha, \beta_1, \beta_2 \in [0, \infty)
  \end{equation*}
  such that the following conditions are satisfied.
  \begin{itemize}
    \item For every  $0 \leq t_0 < t_1 < t_2 < t_3 \leq T$, we have
    \begin{align}
      \label{eq:sewing_regularization}
      \norm{\expect[\delta \sewing_{t_1, t_2, t_3} \vert \cF_{t_0}]}_{L_m(\P)} &\leq \Gamma_1
      (t_1 - t_0)^{-\alpha} (t_3 - t_1)^{\beta_1}, \quad \text{if } M(t_3 - t_1) \leq t_1 - t_0, \\
      \label{eq:sewing_one_half}
      \norm{\delta \sewing_{t_0, t_1, t_2}}_{L_m(\P)}
      &\leq \Gamma_2  (t_2 - t_0)^{\beta_2}.
    \end{align}
    \item We have
    \begin{equation}\label{eq:alpha_and_beta}
      \beta_1  > 1, \quad \beta_2  > \frac{1}{2},
      \quad \beta_1 - \alpha > \frac{1}{2}.
    \end{equation}
  \end{itemize}
  Then, there exists a unique, up to modifications, $\R^d$-valued stochastic process $(\cA_t)_{t \in [0, T]}$ with
  the following properties.
  \begin{itemize}
    \item $\cA_0 = 0$, $\cA_t$ is $\cF_t$-measurable and $\cA_t$ belongs to $L_m(\P)$.
    \item There exist non-negative constants $C_1$, $C_2$ and $C_3$ such that
    \begin{equation}\label{eq:curl_A_conditional}
      \norm{\expect[\cA_{t_2} - \cA_{t_1} - A_{t_1, t_2} \vert \cF_{t_0}]}_{L_m(\P)}
      \leq C_1 \abs{t_1 - t_0}^{-\alpha} \abs{t_2 - t_1}^{\beta_1} ,
    \end{equation}
    \begin{equation}\label{eq:curl_A_one_half}
      \norm{\cA_{t_2} - \cA_{t_1} - A_{t_1, t_2} }_{L_m(\P)} \leq  C_2 \abs{t_2 - t_1}^{\beta_1 - \alpha}
      + C_3 \abs{t_2 - t_1}^{\beta_2},
    \end{equation}
    where $t_2 - t_1 \leq M^{-1}(t_1 - t_0)$ is assumed for the inequality \eqref{eq:curl_A_conditional}.
  \end{itemize}
  In fact, we can choose $C_1$, $C_2$ and $C_3$ so that
  \begin{equation*}
    C_1 \lesssim_{\beta_1} \Gamma_1, \quad C_2 \lesssim_{\alpha, \beta_1, \beta_2, M}  \kappa_{m, d} \Gamma_1, \quad
        C_3 \lesssim_{\alpha, \beta_1, \beta_2, M} \kappa_{m, d} \Gamma_2,
  \end{equation*}
  where $\kappa_{m,d}$ is the constant of the Burkholder-Davis-Gundy inequality, see \eqref{eq:BDG}.
  Furthermore, for $\tau \in [0, T]$, if we set
  \begin{equation*}
    \sewing^{\pi}_{\tau} \defby \sum_{[s, t] \in \pi} \sewing_{s, t}, \quad \text{where $\pi$ is a partition of $[0, \tau]$},
  \end{equation*}
  then the family $(A^{\pi}_{\tau})_{\pi}$ converges to $\cA_{\tau}$ in $L_m(\P)$ as
  $\abs{\pi}$ tends to $0$.
\end{theorem}
\tymchange{
\begin{remark}\label{rem:par_optimal}
  We discuss the optimality of the condition \eqref{eq:alpha_and_beta}. By considering the deterministic $(A_{s,t})$, we see that 
  the condition $\beta_1 > 1$ is necessary. To see that the conditions $\beta_2 > \frac{1}{2}$ and $\beta_1 - \alpha > \frac{1}{2}$ are necessary, 
  let $B^1$ and $B^2$ be two independent one-dimensional fractional Brownian motions with Hurst parameter $\frac{1}{4}$ (see Definition~\ref{def:fbm}), and we set 
  $A_{s, t} \defby B^1_s B^2_{s, t}$. It is well-known since the work \cite{coutin02} of Coutin and Qian that the iterated integral $\int B^1 \mathrm{d} B^2$ does not exist, 
  and therefore the Riemann sum with respect to $(A_{s, t})$ should not converge. 
  In fact, the family $(A_{s, t})$, with filtration $(\mathcal{F}_t)$ generated by $(B^1, B^2)$, satisfies \eqref{eq:sewing_regularization} and \eqref{eq:sewing_one_half} with 
  \begin{align*}
    \alpha = \frac{3}{2}, \quad \beta_1 = 2, \quad \beta_2 = \frac{1}{2}.
  \end{align*}
  To see this, we observe $\delta A_{t_1, t_2, t_3} = - B^1_{t_1, t_2} B^2_{t_2, t_3}$, and 
  \begin{align*}
    \norm{\delta A_{t_1, t_2, t_3}}_{L_m(\mathbb{P})} \lesssim_{m} (t_3-t_1)^{\frac{1}{2}}.
  \end{align*}
  To compute the conditional expectation, we observe 
  \begin{align*}
    \mathbb{E}[\delta A_{t_1, t_2, t_3} \vert \mathcal{F}_{t_0}] = - \mathbb{E}[B^1_{t_1, t_2} \vert \mathcal{F}_{t_0}] 
    \mathbb{E}[B^2_{t_2, t_3} \vert \mathcal{F}_{t_0}],
  \end{align*}
  and by the estimate \eqref{eq:estimate_of_Y_s_u} we have 
  \begin{align*}
    \norm{\mathbb{E}[\delta A_{t_1, t_2, t_3} \vert \mathcal{F}_{t_0}]}_{L_m(\mathbb{P})} 
    \lesssim_m (t_1-t_0)^{-\frac{3}{2}} (t_3 - t_1)^2.
  \end{align*}
\end{remark}
}
\begin{remark}
  The proof shows that if
  \begin{equation}\label{eq:alpha_and_beta_technical}
    1 + \alpha - \beta_1 < 2 \alpha  \beta_2 - \alpha,
  \end{equation}
  then we have $C_2 \lesssim_{\alpha, \beta_1, \beta_2, M} \Gamma_1$ and we can omit the factor $\kappa_{m,d}$.
  This is similar to \cite{khoa20}, where $C_2$ also does not depend on $\kappa_{m,d}$.
  If $\alpha = 0$ and $M = 0$,
  Theorem~\ref{thm:generalized_stochastic_sewing} recovers Lê's stochastic sewing lemma \cite[Theorem~2.1]{khoa20}.
  If $\alpha = 0$ and $M >0 $, it recovers a lemma \cite[Lemma~2.2]{Gerencser2022} by Gerencsér.

  \tymchange{Recently, Gerecsér's stochastic sewing lemma is called \emph{shifted} stochastic sewing lemma. 
  In the follow-up works, we continue to refer to Theorem~\ref{thm:generalized_stochastic_sewing} by the same name.
  }
\end{remark}
\begin{remark}\label{rem:convergence_rate}
  The proof shows that there exists $\epsilon = \epsilon(\alpha, \beta_1, \beta_2) > 0$ such that
  \begin{equation*}
    \norm{\cA_{\tau} - A^{\pi}_{\tau}}_{L_m(\P)} \lesssim_{\alpha, \beta_1, \beta_2, M, m, d, T} (\Gamma_1 + \Gamma_2) \abs{\pi}^{\epsilon}
  \end{equation*}
  for every $\tau \in [0, T]$ and every partition $\pi$ of $[0, \tau]$.
  A similar remark holds in the setting of Corollary~\ref{cor:singular_generalized_stochastic_sewing_lemma}.
\end{remark}
\begin{remark}
  As in another work \cite{le21} of Lê, it should be possible to extend
  Theorem~\ref{thm:generalized_stochastic_sewing}
  so that the stochastic process $(A_{s,t})_{s,t \in [0, T]}$ takes values in a certain Banach space.
\end{remark}
\begin{remark}
  A multidimensional version of the sewing lemma is the \emph{reconstruction theorem} \cite[Theorem~3.10]{hairer_theory_2014} of
  Hairer. A stochastic version of the reconstruction theorem was obtained by Kern~\cite{kern2021stochastic}.
  It could be possible to extend Theorem~\ref{thm:generalized_stochastic_sewing} in the multidimensional setting, but we will not pursue it
  in this paper.
\end{remark}
The proof of Theorem~\ref{thm:generalized_stochastic_sewing} is given in Section~\ref{sec:proof}.
If $A_{s, t}$ is given by \eqref{eq:germ_fbm}, then we can apply Theorem~\ref{thm:generalized_stochastic_sewing} with
\begin{equation*}
  \alpha = 1 - H, \quad \beta_1 = 2 - H, \quad \beta_2 = 1.
\end{equation*}
However,
the application of Theorem~\ref{thm:generalized_stochastic_sewing} goes beyond this simple problem of
$\frac{1}{H}$-variation of the fractional Brownian motion.
Indeed, in Section~\ref{sec:stochastic_integral} we prove the convergence of It\^o and Stratonovich approximations to
the stochastic integrals
\begin{equation*}
  \int_0^T f(B_s) \dd B_s \quad \text{and } \quad \int_0^T f(B_s) \circ \dd B_s
\end{equation*}
with $H > \frac{1}{2}$ in It\^o's case and with $H > \frac{1}{6}$ in Stratonovich's case,
under rather general assumptions on the regularity of $f$, in fact $f \in C^2_b(\R^d, \R^d)$ works for all $H > \frac16$.
In Section~\ref{sec:local_time}, we obtain new representations of local times of fractional Brownian motions via discretization.

Finally, we remark that one of the most interesting applications of Lê's stochastic sewing lemma lies in
the phenomenon of \emph{regularization by noise}, see e.g. \cite{khoa20}, Athreya, Butkovsky, Lê and Mytnik~\cite{athreya2021wellposedness},
\cite{Gerencser2022} and Anzeletti, Richard and Tanré~\cite{anzeletti21}.
In these works, they consider the stochastic differential equation (SDE)
\begin{equation}\label{eq:additive_noise}
  \dd X_t = b(X_t) \dd t + \dd Y_t
\end{equation}
with an \emph{additive} noise $Y$, which is often a fractional Brownian motion.
It is interesting that, although in absence of noise the coefficient $b$ needs to belong to $C^1$ for
well-posedness, the presence of noise enables us to prove certain well-posedness of \eqref{eq:additive_noise} under much weaker assumption,
in fact $b$ can be even a distribution; hence the name \emph{regularization by noise}.

In Section~\ref{sec:regularization_by_noise} of our paper, we are interested in a related but different problem.
Indeed, we are interested in improving the regularity of the diffusion coefficient rather than the drift coefficient.
We consider the Young SDE
\begin{equation*}
  \dd X_t = b(X_t) \dd t + \sigma(X_t) \dd B_t
\end{equation*}
driven by a fractional Brownian motion $B$ with Hurst parameter $H \in (\frac{1}{2}, 1)$.
The pathwise theory of Young's differential equation requires that the regularity of $\sigma$ is better than
$1/H$ for uniqueness, and this condition is sharp for general drivers $B$ of the same regularity as the fractional Brownian motion. We will improve this regularity assumption for pathwise uniqueness and strong existence.
Again, a stochastic sewing lemma (Lemma~\ref{lem:stochastic-sewing-tradeoff}), which is a variant of Theorem~\ref{thm:generalized_stochastic_sewing}, will play a key role.
\subsubsection*{Acknowledgment}
\tymchange{The main part of the work was done while TM was a PhD student at Freie Universität Berlin 
under the financial support of the German Science Foundation (DFG) via the IRTG 2544.}
NP gratefully acknowledges funding by DFG through the Heinz Maier-Leibnitz Prize.
\tymchange{We thank Khoa Lê for suggesting a simpler proof of Lemma~\ref{lem:refine_A_t_0_t_1} in Appendix~\ref{app:technical} than the original one.}
TM thanks Henri Altman, Hannes Kern and Helena Kremp for the discussions related to this paper.

\subsubsection*{Notation}
We write $\N_0 \defby \{0, 1, 2, \ldots\}$ and $\N \defby \{1, 2, \ldots\}$.
Given a function $f:[S, T] \to \R^d$, we write $f_{s, t} \defby f_t - f_s$.
We denote by $\kappa_{m, d}$ the best constant
 of the discrete Burkholder-Davis-Gundy (BDG) inequality for
$\R^d$-valued martingale differences \cite{BDG72}.
Namely, if we are given a filtration $(\cF_n)_{n=1}^{\infty}$ and a sequence $(X_n)_{n=1}^{\infty}$
of $\R^d$-valued random variables such that $X_n$ is $\cF_{n}$-measurable for every $n \geq 1$ and
$\expect[X_n \vert \cF_{n-1}] = 0$ for every $n \geq 2$, then
\begin{equation}\label{eq:BDG}
  \norm{\sum_{n=1}^{\infty} X_n}_{L_m(\P)} \leq \kappa_{m, d}
  \norm{\sum_{n=1}^{\infty} X_n^2}_{L_{\frac{m}{2}}(\P)}^{\frac{1}{2}}.
\end{equation}
Rather than \eqref{eq:BDG}, we mostly use the inequality
\begin{equation}\label{eq:BDG-Minkowski}
  \norm{\sum_{n=1}^{\infty} X_n}_{L_m(\P)} \leq \kappa_{m, d}
  \Big(\sum_{n=1}^{\infty} \norm{X_n}_{L_m(\P)}^2 \Big)^{\frac{1}{2}}.
\end{equation}
for $m \geq 2$, which follows from \eqref{eq:BDG} by Minkowski's inequality.
We write $A \lesssim B$ or $A = O(B)$ if there exists a non-negative constant $C$ such that
$A \leq C B$. To emphasize the dependence of $C$ on some parameters $a, b, \ldots$, we write
$A \lesssim_{a, b, \ldots} B$.

\section{Proof of the main theorem}\label{sec:proof}
The overall strategy of the proof is the same as that of the original work \cite{khoa20} of Lê.
Namely, we combine the argument of the deterministic sewing lemma (\cite{gubinelli04}, \cite{feyel06} and
Yaskov~\cite{yaskov18}) with
the discrete BDG inequality \cite{BDG72}.
However, the proof of Theorem~\ref{thm:generalized_stochastic_sewing} requires more labor at technical level.
Some proofs will be postponed to Appendix~\ref{app:technical}.

As in \cite{khoa20}, the following lemma, which originates from \cite{yaskov18},  will be needed. It allows us to replace general partitions by dyadic partitions.
\begin{lemma}[{\cite[Lemma~2.14]{khoa20}}]\label{lem:dyadic_allocation}
  Under the setting of Theorem~\ref{thm:generalized_stochastic_sewing},
  let
  \begin{equation*}
    0 \leq  t_0 < t_1 < \cdots < t_{N-1} < t_N \leq T.
  \end{equation*}
  Then, we have
  \begin{equation}\label{eq:A_and_R_n_i}
    A_{t_0, t_N} - \sum_{i=1}^N A_{t_{i-1}, t_i} = \sum_{n \in \N_0} \sum_{i=0}^{2^n - 1}
    R^n_i,
  \end{equation}
  where
  \begin{equation}\label{eq:expression_of_R_n_i}
    R^n_i \defby \delta A_{s^{n, i}_1, s^{n, i}_2, s^{n, i}_3} + \delta A_{s^{n, i}_1, s^{n, i}_3, s^{n, i}_4},
  \end{equation}
  and
  \begin{equation*}
    n \in \N_0, \quad i \in \{0, 1, \ldots, 2^n - 1\},
    \quad s^{n, i}_j \in [t_0 + \frac{i(t_N - t_0)}{2^n}, t_0 + \frac{(i+1)(t_N - t_0)}{2^n}],
  \end{equation*}
  and where $R^n_i = 0$ for all sufficiently large $n$.
\end{lemma}
The next two lemmas (Lemma~\ref{lem:refine_A_t_0_t_1} and Lemma~\ref{lem:refine_A_t_0_t_1_conditioning}) correspond to the estimates
\cite[(2.50) and (2.51)]{khoa20} respectively.

\begin{lemma}\label{lem:refine_A_t_0_t_1_conditioning}
  Under the setting of Theorem~\ref{thm:generalized_stochastic_sewing},
  let
  \begin{equation*}
    0 \leq s < t_0 < t_1 < \cdots < t_{N-1} < t_N \leq T,  \quad t_N - t_1 \leq \frac{t_0 - s}{M}.
  \end{equation*}
  Then,
  \begin{equation*}
    \norm{\expect[A_{t_0, t_N} - \sum_{i=1}^N A_{t_{i-1}, t_i} \vert \cF_{s}]}_{L_m(\P)}
    \lesssim_{\beta_1}
      \Gamma_1  \abs{t_0 - s}^{-\alpha} \abs{t_N - t_0}^{\beta_1}.
  \end{equation*}
\end{lemma}
\begin{proof}
  In view of the decomposition \eqref{eq:A_and_R_n_i}, the triangle inequality gives
  \begin{equation*}
    \norm{\expect[A_{t_0, t_N} - \sum_{i=1}^N A_{t_{i-1}, t_i} \vert \cF_{s}]}_{L_m(\P)}
    \leq \sum_{n \in \N_0} \sum_{i=0}^{2^n - 1} \norm{\expect[R_i^n \vert \cF_{s}]}_{L_m(\P)}.
  \end{equation*}
  By \eqref{eq:sewing_regularization} and \eqref{eq:expression_of_R_n_i},
  \begin{align*}
    \norm{\expect[R_i^n \vert \cF_{s}]}_{L_m(\P)}
    \leq 2 \Gamma_1 (t_0 - s)^{-\alpha} (2^{-n} \abs{t_N - t_0})^{\beta_1}
    = 2 \Gamma_1 2^{-n \beta_1 } \abs{t_0 - s}^{-\alpha} \abs{t_N - t_0}^{\beta_1}.
  \end{align*}
  Therefore, recalling $\beta_1 > 1$ from \eqref{eq:alpha_and_beta}, the claim follows.
\end{proof}

The following lemma is the most important technical ingredient for the proof of Theorem~\ref{thm:generalized_stochastic_sewing}.
\begin{lemma}\label{lem:refine_A_t_0_t_1}
  Under the setting of Theorem~\ref{thm:generalized_stochastic_sewing},
  let
  \begin{equation*}
    0 \leq t_0 < t_1 < \cdots < t_{N-1} < t_N \leq T.
  \end{equation*}
  Then,
  \begin{equation*}
    \norm{A_{t_0, t_N} - \sum_{i=1}^N A_{t_{i-1}, t_i} }_{L_m(\P)}
    \lesssim_{\alpha, \beta_1, \beta_2, M}  \kappa_{m, d} \Gamma_1  \abs{t_N - t_0}^{\beta_1 - \alpha}
    +  \kappa_{m, d} \Gamma_2 \abs{t_N - t_0}^{\beta_2}.
  \end{equation*}
  Under \eqref{eq:alpha_and_beta_technical}, we can replace $\kappa_{m, d} \Gamma_1$ by $\Gamma_1$.
\end{lemma}
\begin{proof}[Proof under \eqref{eq:alpha_and_beta_technical}]
  To simplify the proof, here we assume \eqref{eq:alpha_and_beta_technical}, i.e. that the additional technical condition $1+\alpha-\beta_1 < 2\alpha\beta_2 - \alpha$ holds.
  The proof in the general setting will be given in Appendix~\ref{app:technical}.

  We again use the representation \eqref{eq:A_and_R_n_i}.
  We fix a large $n \in \N$ and set $\cF^n_k \defby \cF_{t_0 + \frac{k}{2^n}(t_N - t_0)}$.
  Fix an  integer $L = L_n \in [M + 1, 2^n]$, which will be chosen later.
  We have
  \begin{multline}\label{eq:R_n_i_decomposition}
    \sum_{i=0}^{2^n - 1} R^n_i
    = \sum_{l = 0}^{L - 1} \sum_{\substack{j \ge 0:\\ L j + l < 2^n}} \Big( R^n_{Lj + l} -
    \expect[R^n_{L j + l} \vert \cF_{L (j-1) + l + 1}^n] \indic_{\{j \geq 1\}} \Big)  \\
    + \sum_{l = 0}^{L - 1} \sum_{\substack{j \ge 0:\\ L j + l < 2^n}}
    \expect[R^n_{L j + l} \vert \cF_{L (j-1) + l + 1}^n].
  \end{multline}
  We estimate the first term of \eqref{eq:R_n_i_decomposition}.
  By the BDG inequality together with Minkowski's inequality (see~\eqref{eq:BDG-Minkowski}), we have
  \begin{align*}
    \MoveEqLeft[5]
    \norm{\sum_{\substack{j \ge 0:\\ L j + l < 2^n}} \Big( R^n_{Lj + l} -
    \expect[R^n_{L j + l} \vert \cF_{L (j-1) + l + 1}^n] \indic_{\{j \geq 1\}} \Big)}_{L_m(\P)}^2 \\
    &\leq \kappa_{m, d}^2 \sum_{\substack{j \ge 0:\\ L j + l < 2^n}} \norm{ R^n_{Lj + l} -
    \expect[R^n_{L j + l} \vert \cF_{L (j-1) + l + 1}^n] \indic_{\{j \geq 1\}} }_{L_m(\P)}^2 \\
    &\leq 4 \kappa_{m, d}^2 \sum_{\substack{j \ge 0:\\ L j + l < 2^n}} \norm{ R^n_{Lj + l} }_{L_m(\P)}^2.
  \end{align*}
  Using \eqref{eq:sewing_one_half} and \eqref{eq:expression_of_R_n_i} and noting that we include more terms in the sum by requiring $j \le 2^n/L$ only instead of $Lj+l \le 2^n - 1$, we get
  \begin{equation*}
    \sum_{\substack{j \ge 0:\\ L j + l < 2^n}}\norm{ R^n_{Lj + l} }_{L_m(\P)}^2
    \leq 4 \Gamma_2^2 2^{-n (2\beta_2 - 1)} L^{-1} \abs{t_N - t_0}^{2 \beta_2}.
  \end{equation*}
  Therefore,
  \begin{multline*}
    \norm{\sum_{l = 0}^{L - 1}  \sum_{\substack{j \ge 0:\\ L j + l < 2^n}} \Big( R^n_{Lj + l} -
    \expect[R^n_{L j + l} \vert \cF_{L (j-1) + l + 1}^n] \indic_{\{j \geq 1\}} \Big)}_{L_m(\P)}
    \lesssim \kappa_{m, d} \Gamma_2 L^{\frac{1}{2}} 2^{-n(\beta_2 - \frac{1}{2})} \abs{t_N - t_0}^{\beta_2}.
  \end{multline*}

  We next estimate the second term of \eqref{eq:R_n_i_decomposition}.
  The triangle inequality yields
  \begin{equation*}
    \norm{\sum_{l = 0}^{L - 1}  \sum_{\substack{j \ge 0:\\ L j + l < 2^n}}
    \expect[R^n_{L j + l} \vert \cF_{L (j-1) + l + 1}^n]}_{L_m(\P)}
    \leq
    \sum_{l = 0}^{L - 1}  \sum_{\substack{j \ge 0:\\ L j + l < 2^n}}
    \norm{\expect[R^n_{L j + l} \vert \cF_{L (j-1) + l + 1}^n]}_{L_m(\P)}.
  \end{equation*}
  By \eqref{eq:sewing_regularization},
  \begin{equation*}
    \norm{\expect[R^n_{L j + l} \vert \cF_{L (j-1) + l + 1}^n]}_{L_m(\P)}
    \leq \Gamma_1 (L-1)^{-\alpha} 2^{-(\beta_1 - \alpha) n} \abs{t_N - t_0}^{\beta_1 - \alpha}.
  \end{equation*}
  Therefore,
  \begin{equation*}
    \norm{\sum_{l = 0}^{L - 1}  \sum_{\substack{j \ge 0:\\ L j + l < 2^n}}
    \expect[R^n_{L j + l} \vert \cF_{L (j-1) + l + 1}^n]}_{L_m(\P)}
    \lesssim_{\alpha} \Gamma_1 L^{-\alpha} 2^{-(\beta_1 - \alpha - 1) n} \abs{t_N - t_0}^{\beta_1 - \alpha}.
  \end{equation*}

  In conclusion,
  \begin{equation}\label{eq:estimate_of_R_n_i_one_step}
    \norm{\sum_{i=0}^{2^n - 1} R^n_i}_{L_m(\P)} \lesssim_{\alpha}
    \Gamma_1 L^{-\alpha} 2^{-(\beta_1 - \alpha - 1) n} \abs{t_N - t_0}^{\beta_1 - \alpha}
    + \kappa_{m, d} \Gamma_2 L^{\frac{1}{2}} 2^{-n(\beta_2 - \frac{1}{2})} \abs{t_N - t_0}^{\beta_2}.
  \end{equation}
  We wish to choose $L = L_n$ so that \eqref{eq:estimate_of_R_n_i_one_step} is summable with respect to $n$.
  We therefore set $L_n \defby \lfloor 2^{\delta n} \rfloor$, where
  \begin{equation}\label{eq:delta_under_technical}
    \alpha \delta + \beta_1 - \alpha - 1 > 0, \quad
    0 < \delta < \min\{2 \beta_2 - 1, 1 \}.
  \end{equation}
  Such a $\delta$ exists exactly under the additional technical assumption~\eqref{eq:alpha_and_beta_technical}, namely if $1 + \alpha - \beta_1 < 2 \alpha  \beta_2 - \alpha$. Then, \eqref{eq:estimate_of_R_n_i_one_step} yields
  \begin{equation*}
    \norm{ \sum_{n: 2^{n \delta} \geq M + 2}\sum_{i=0}^{2^n - 1} R^n_i}_{L_m(\P)}
    \lesssim_{\alpha, \beta_1, \beta_2} \Gamma_1 \abs{t_N - t_0}^{\beta_1 - \alpha} + \kappa_{m, d} \Gamma_2 \abs{t_N - t_0}^{\beta_2}.
  \end{equation*}
  To estimate the contribution coming from the small $n$ with $2^{n\delta} < M+2$, we apply \eqref{eq:sewing_one_half} which yields
  \begin{equation*}
    \norm{\sum_{i=0}^{2^n - 1} R^n_i}_{L_m(\P)}
    \leq 2 \Gamma_2 \sum_{i=0}^{2^n - 1} 2^{-n \beta_2} \abs{t_N - t_0}^{\beta_2} = \Gamma_2 2^{1 + n (1 - \beta_2)} \abs{t_N - t_0}^{\beta_2}.
  \end{equation*}
  Thus, we conclude
  \begin{equation*}
    \norm{ \sum_{n \in \N_0} \sum_{i=0}^{2^n - 1} R^n_i}_{L_m(\P)}
    \lesssim_{\alpha, \beta_1, \beta_2, M} \Gamma_1 \abs{t_N - t_0}^{\beta_1 - \alpha} + \kappa_{m, d} \Gamma_2 \abs{t_N - t_0}^{\beta_2},
  \end{equation*}
  where the fact $\kappa_{m, d} \geq 1$ is used.
  \end{proof}

\begin{lemma}\label{lem:A_and_A_dash}
  Under the setting of Theorem~\ref{thm:generalized_stochastic_sewing},
  let $\pi, \pi'$ be partitions of $[0, T]$ such that $\pi$ refines  $\pi'$.
  Suppose that we have
  \begin{equation}\label{eq:pi_almost_uniform}
    \min_{[s, t] \in \pi'} \abs{s - t} \geq \frac{\abs{\pi'}}{3}.
  \end{equation}
  Then, there exists  $\epsilon \in (0, 1)$ such that
  \begin{equation*}
    \norm{A^{\pi'}_T - A^{\pi}_T}_{L_m(\P)}
    \lesssim_{\alpha, \beta_1, \beta_2, M, m, d, T} (\Gamma_1 + \Gamma_2) \abs{\pi'}^{\epsilon}.
  \end{equation*}
\end{lemma}
\begin{proof}[Sketch of the proof]
  Here we give a sketch of the proof under \eqref{eq:alpha_and_beta_technical}.
  The complete proof is given in Appendix~\ref{app:technical}.
  The argument is similar to Lemma~\ref{lem:refine_A_t_0_t_1}.

  Write
  \begin{equation*}
    \pi' =: \{0 = t_0 < t_1 < \cdots < t_{N-1} < t_N = T\}
  \end{equation*}
   and
  \begin{equation*}
    \set{[s, t] \in \pi \given t_j \leq s < t \leq t_{j+1}}
    =: \{t_j = t^j_0 < t^j_1 < \cdots < t^j_{N_j - 1} < t^j_{N_j} = t_{j+1}\}.
  \end{equation*}
  We set $L \defby \lfloor \abs{\pi'}^{-\delta} \rfloor$, where $\delta$ satisfies \eqref{eq:delta_under_technical}.
  We set
  \begin{equation*}
    Z_{j}^{l} \defby A_{t_{jL + l}, t_{jL + l +1}} - \sum_{k=1}^{N_{jL + l}} A_{t^{jL + l}_{k-1}, t^{jL + l}_{k}}.
  \end{equation*}
  As in Lemma~\ref{lem:refine_A_t_0_t_1}, we consider the decomposition $A^{\pi'}_T - A^{\pi}_T = A + B$, where
  \begin{equation*}
    A \defby \sum_{l < L} \sum_{j: Lj < N - l}
    \Big\{ Z^{l}_j - \expect[Z^{l}_j \vert \cF_{t_{(j-1)L + l + 1}}]   \Big\}, \quad
    B \defby \sum_{l < L} \sum_{j: Lj < N - l} \expect[Z^{l}_j \vert \cF_{t_{(j-1)L + l + 1}}].
  \end{equation*}
  We estimate $A$ by using the BDG inequality, Lemma~\ref{lem:refine_A_t_0_t_1} and \eqref{eq:pi_almost_uniform}, \tymchange{to obtain
  \begin{align*}
    \norm{A}_{L_m(\mathbb{P})} \lesssim_{\alpha, \beta_1, \beta_2, M, m, d, T} L^{\frac{1}{2}} (\Gamma_1 \abs{\pi'}^{\beta_1 - \alpha - \frac{1}{2}} + 
    \Gamma_2 \abs{\pi'}^{\beta_2 - \frac{1}{2}}).
  \end{align*}
  }
  We estimate $B$ by using the triangle inequality, Lemma~\ref{lem:refine_A_t_0_t_1_conditioning} and \eqref{eq:pi_almost_uniform}, 
  \tymchange{to obtain 
  \begin{align*}
    \norm{B}_{L_m(\mathbb{P})} \lesssim_{\alpha, \beta_1, M, m, d, T} \Gamma_1 L^{-\alpha k} \abs{\pi'}^{\beta_1 - \alpha - 1}.
  \end{align*}
  As in Lemma~\ref{lem:refine_A_t_0_t_1}, we choose $L \defby \lfloor \abs{\pi'}^{-\delta}  \rfloor$ with $\delta$ satisfying \eqref{eq:delta_under_technical}. 
  We then obtain the claimed estimate.
  } 
\end{proof}

{
\begin{remark}\label{rem:A_dash_and_curly_A}
	In the setting of Lemma~\ref{lem:A_and_A_dash}, assume that the adapted process $(\mathcal A_t)_{t \in [0,T]}$ satisfies~\eqref{eq:curl_A_conditional} and~\eqref{eq:curl_A_one_half}. Then we obtain for some $\varepsilon > 0$:
	\begin{equation*}
    \norm{A^{\pi'}_T - \mathcal A_T}_{L_m(\P)}
    \lesssim_{\alpha, \beta_1, \beta_2, M, m, d, T} (\Gamma_1 + \Gamma_2) \abs{\pi'}^{\epsilon}.
  \end{equation*}
  Indeed, it suffices to replace $A_{t_{jL + l}, t_{jL + l +1}}$ by $\mathcal A_{t_{jL + l}, t_{jL + l +1}}$ in the previous proof.
\end{remark}
}

\begin{lemma}\label{lem:partition}
  Let $\pi$ be a partition of $[0, T]$. Then, there exists a partition $\pi'$ of $[0, T]$ such that
  $\pi$ refines $\pi'$, $\abs{\pi'} \leq 3 \abs{\pi}$ and
  \begin{equation*}
    \min_{[s, t] \in \pi'} \abs{t - s} \geq  \frac{\abs{\pi'}}{3}.
  \end{equation*}
\end{lemma}
\begin{proof}
  We write $\pi = \{0=t_0 < t_1 < \cdots <t_{N-1} < t_N = T\}$.
  We set $k_0 \defby -1$ and for $l\in \N$ we inductively set
  \begin{equation*}
    k_l \defby \inf \set{j > k_{l-1} \given t_{j+1} - t_{k_{l-1} + 1} \ge \abs{\pi}}, \quad \mbox{where } \inf \emptyset \defby N.
  \end{equation*}
  Set $L \defby \sup \set{l \given k_l < N}$. Then, we define
  \begin{equation*}
    s_j \defby
    \begin{cases}
      t_{k_j + 1} & \text{if } j < L, \\
      t_N         & \text{if } j= L.
    \end{cases}
  \end{equation*}
  By construction, $\pi'= \{s_j\}_{j=1}^L$ satisfies the claimed properties: $s_{j+1} - s_j \le 2|\pi|$ if $j<L-2$, and $s_L - s_{L-1} \le 3 |\pi|$, so $|\pi'| \le 3 |\pi|$; moreover, $\min_{[s,t] \in \pi'}|t-s| \ge |\pi| \ge 3^{-1} |\pi'|$.
\end{proof}

\begin{proof}[Proof of Theorem~\ref{thm:generalized_stochastic_sewing}]
  We will not write down dependence on $\alpha, \beta_1, \beta_2, M, m, d, T$.
  We first prove the convergence of $(A^{\pi}_{\tau})_{\pi}$. Without loss of generality, we assume $\tau = T$.
  Let $\pi_1, \pi_2$ be partitions of $[0, T]$. By Lemma \ref{lem:partition},
  there exist partitions $\pi_1'$, $\pi_2'$ such that for $j \in \{1, 2\}$
  the partition $\pi_j$ refines $\pi_j'$, $\abs{\pi'_j} \leq 3 \abs{\pi_j}$ and
  \begin{equation*}
    \min_{[s, t] \in \pi'_j} \abs{t- s} \geq 3^{-1} \abs{\pi_j'}.
  \end{equation*}
  Lemma~\ref{lem:A_and_A_dash} shows that
  for some $\epsilon > 0$ we have
  \begin{equation*}
    \norm{A^{\pi_j}_T - A^{\pi_j'}_T}_{L_m(\P)} \lesssim (\Gamma_1 + \Gamma_2) \abs{\pi_j}^{\epsilon}.
  \end{equation*}
  Therefore, by the triangle inequality,
  \begin{equation}\label{eq:A_pi}
    \norm{A^{\pi_1}_T - A^{\pi_2}_T}_{L_m(\P)}
    \lesssim \norm{A^{\pi_1'}_T - A^{\pi_2'}_T}_{L_m(\P)} + (\Gamma_1 + \Gamma_2)(\abs{\pi_1}^{\epsilon} + \abs{\pi_2}^{\epsilon}).
  \end{equation}
  Let $\pi$ refine both $\pi_1'$ and $\pi_2'$. Lemma~\ref{lem:A_and_A_dash} implies that
  \begin{equation}\label{eq:A_dash_pi}
    \norm{A^{\pi_1'}_T - A^{\pi_2'}_T}_{L_m(\P)}
    \leq \norm{A^{\pi_1'}_T - A^{\pi}_T}_{L_m(\P)}
    + \norm{A^{\pi_{\tymchange{2}}'}_T - A^{\pi}_T}_{L_m(\P)}
    \lesssim (\Gamma_1 + \Gamma_2) (\abs{\pi_1}^{\epsilon} + \abs{\pi_2}^{\epsilon}).
  \end{equation}
  The estimates \eqref{eq:A_pi} and \eqref{eq:A_dash_pi} show
  \begin{equation*}
    \norm{A^{\pi_1}_T - A^{\pi_2}_T}_{L_m(\P)} \lesssim (\Gamma_1 + \Gamma_2)(\abs{\pi_1}^{\epsilon} + \abs{\pi_2}^{\epsilon}).
  \end{equation*}
  Thus, $\{A^{\pi}_T\}_{\pi}$ forms a Cauchy {net} in $L_m(\P)$. We denote the limit by $\sA_T$.

  We next prove that $(\sA_t)_{t \in [0, T]}$ satisfies \eqref{eq:curl_A_conditional} and \eqref{eq:curl_A_one_half}.
  Let $t_0 < t_1 < t_2$ be such that $M(t_2 - t_1) \leq t_1 - t_0$. Let $\pi_n = \{ t_1 + k 2^{-n}(t_2-t_1): k=0,\dots, 2^n\}$ be
  the $n$th dyadic partition of $[t_1, t_2]$ and we write
  \begin{equation*}
    A^n_{t_1, t_2} \defby \sum_{[s, t] \in \pi_n} A_{s, t}.
  \end{equation*}
  We have
  \begin{equation}\label{eq:sA_conditional_limit}
    \expect[\sA_{t_1, t_2} - A_{t_1, t_2} \vert \cF_{t_0}] =
    \lim_{n \to \infty} \expect[A^n_{t_1, t_2} - A_{t_1, t_2} \vert \cF_{t_0}] \quad \text{in } L_m(\P).
  \end{equation}
  By Lemma~\ref{lem:refine_A_t_0_t_1_conditioning},
  \begin{equation*}
    \norm{\expect[A_{t_1, t_2} - A^n_{t_1, t_2} \vert \cF_{t_0}]}_{L_m(\P)}
    \lesssim_{\beta_1} \Gamma_1 \abs{t_1 - t_0}^{-\alpha} \abs{t_2 - t_1}^{\beta_1}.
  \end{equation*}
  In this estimate, we can replace $A^n_{t_1, t_2}$ by $\sA_{t_1, t_2}$ in view of \eqref{eq:sA_conditional_limit}.
  Similarly, by Lemma~\ref{lem:refine_A_t_0_t_1}, we obtain
  \begin{equation*}
    \norm{\sA_{t_1, t_2} - A_{t_1, t_2}}_{L_m(\P)}
    \lesssim_{\alpha, \beta_1, \beta_2, M}  \kappa_{m, d} \Gamma_1 \abs{t_2 - t_1}^{\beta_1 - \alpha}
    + \kappa_{m, d} \Gamma_2 \abs{t_2 - t_1}^{\beta_2}.
  \end{equation*}
  Under \eqref{eq:alpha_and_beta_technical}, we can replace $ \kappa_{m, d} \Gamma_1$ by $\Gamma_1$.

  Finally, let us prove the uniqueness of $\cA$. Let $(\tilde{\cA}_t)_{t \in [0, T]}$ be another
  adapted process satisfying $\tilde{\cA}_0 = 0$, \eqref{eq:curl_A_conditional} and \eqref{eq:curl_A_one_half}.
  It suffices to show $\cA_T = \tilde{\cA}_T$ \tymchange{almost surely}.
  Let $\pi_n$ be the $n$th dyadic partition of $[0, T]$. {By Remark~\ref{rem:A_dash_and_curly_A} we have
  \[
  	\norm{\cA_T - \tilde \cA_T}_{L_m(\PP)} \le \norm{\cA_T - A^{\pi^n}_T}_{L_m(\PP)} + \norm{A^{\pi^n}_T - \tilde \cA_T}_{L_m(\PP)} \lesssim 2^{-n\varepsilon}T^\varepsilon.
  \]
  Since $n \in \N$ is arbitrary we must have $\cA_T = \tilde \cA_T$ \tymchange{almost surely}.
  }
\end{proof}
As in \cite[Theorem~4.1]{athreya2021wellposedness} of Athreya, Butkovsky, Lê and Mytnik, we will give an extension of
Theorem~\ref{thm:generalized_stochastic_sewing} that allows singularity at $t = 0$, which will be needed in Section~\ref{sec:local_time}.
\begin{corollary}\label{cor:singular_generalized_stochastic_sewing_lemma}
  Suppose that we have a filtration $(\cF_t)_{t \in [0, T]}$ and a family of $\R^d$-valued random variables
  $(\sewing_{s, t})_{0 \leq s \leq t \leq T}$ such that
  $\sewing_{s, s} = 0$ for every $s \in [0, T]$ and such that $\sewing_{s, t}$ is $\cF_t$-measurable.
  Furthermore, suppose that there exist constants
  \begin{equation*}
    m \in [2, \infty), \quad \Gamma_1, \Gamma_2, \Gamma_3, M \in [0, \infty),
       \quad \alpha, \beta_1, \beta_2, \beta_3, \gamma_1, \gamma_2 \in [0, \infty)
  \end{equation*}
  such that the following conditions are satisfied.
  \begin{itemize}
    \item For every  $0 \leq t_0 < t_1 < t_2 < t_3 \leq T$, we have
    \begin{align}
      \label{eq:sewing_regularization_singular}
      \norm{\expect[\delta \sewing_{t_1, t_2, t_3} \vert \cF_{t_0}]}_{L_m(\P)} &\leq \Gamma_1 t_1^{-\gamma_1}
      (t_1 - t_0)^{-\alpha} (t_3 - t_1)^{\beta_1}, \\
      \label{eq:sewing_one_half_singular}
      \norm{\delta \sewing_{t_0, t_1, t_2}}_{L_m(\P)}
      &\leq \Gamma_2  t_0^{-\gamma_2} (t_2 - t_0)^{\beta_2}, \\
      \label{eq:sewing_continuity_at_zero}
      \norm{\delta \sewing_{t_0, t_1, t_2}}_{L_m(\P)}
      &\leq \Gamma_3  (t_2 - t_0)^{\beta_3},
    \end{align}
    where $M(t_3 - t_1) \leq t_1 - t_0$ is assumed for \eqref{eq:sewing_regularization_singular}
    and $t_0 > 0$ is assumed for \eqref{eq:sewing_one_half_singular}.
    \item We have
    \begin{equation}\label{eq:alpha_and_beta_gamma}
      \beta_1  > 1, \quad \beta_2  > \frac{1}{2},
      \quad \beta_1 - \alpha > \frac{1}{2},
      \quad \gamma_1, \gamma_2 < \frac{1}{2}, \quad \beta_3 > 0.
    \end{equation}
  \end{itemize}
  Then, there exists a unique, up to modifications, $\R^d$-valued stochastic process $(\cA_t)_{t \in [0, T]}$ with
  the following properties.
  \begin{itemize}
    \item $\cA_0 = 0$, $\cA_t$ is $\cF_t$-measurable and $\cA_t$ belongs to $L_m(\P)$.
    \item There exist non-negative constants $C_1, \ldots, C_6$ such that
    \begin{align}
      \label{eq:curl_A_conditional_singular}
      &\norm{\expect[\cA_{t_2} - \cA_{t_1} - A_{t_1, t_2} \vert \cF_{t_0}]}_{L_m(\P)}
      \leq C_1 t_1^{-\gamma_1} \abs{t_1 - t_0}^{-\alpha} \abs{t_2 - t_1}^{\beta_1}, \\
      \label{eq:curl_A_one_half_singular}
      &\norm{\cA_{t_2} - \cA_{t_1} - A_{t_1, t_2} }_{L_m(\P)} \leq  C_2 t_1^{-\gamma_1} \abs{t_2 - t_1}^{\beta_1 - \alpha}
      + C_3 t_1^{-\gamma_2} \abs{t_2 - t_1}^{\beta_2}, \\
      \label{eq:curl_A_continuity_at_zero}
      &\norm{\cA_{t_2}  - \cA_{t_1} - A_{t_1, t_2} }_{L_m(\P)} \leq  C_4 \abs{t_2 - t_1}^{\beta_1 - \alpha_1 -\gamma_1}
      + C_5 \abs{t_2 - t_1}^{\beta_2 - \gamma_2}
      + C_6 \abs{t_2 - t_1}^{\beta_3},
    \end{align}
    where $t_2 - t_1 \leq M^{-1}(t_1 - t_0)$ is assumed for the inequality \eqref{eq:curl_A_conditional_singular}
    and $t_1 > 0$ is assumed for the inequality \eqref{eq:curl_A_one_half_singular}.
  \end{itemize}
  In fact, we can choose $C_1, \ldots, C_6$ so that
  \begin{align*}
    &C_1 \lesssim_{\beta_1} \Gamma_1, \quad C_2 \lesssim_{\alpha, \beta_1, \beta_2, M}  \kappa_{m, d} \Gamma_1, \quad
        C_3 \lesssim_{\alpha, \beta_1, \beta_2, M} \kappa_{m, d} \Gamma_2, \\
    &C_4 \lesssim_{\alpha, \beta_1, \gamma_1, M} \kappa_{m, d} \Gamma_1,
    \quad
    C_5 \lesssim_{\beta_2, \gamma_2, M} \kappa_{m, d} \Gamma_2, \quad
        C_6 \lesssim_{\beta_3, M} \kappa_{m, d} \Gamma_3.
  \end{align*}
  Furthermore, for $\tau \in [0, T]$, if we set
  \begin{equation*}
    \sewing^{\pi}_{\tau} \defby \sum_{[s, t] \in \pi} \sewing_{s, t}, \quad \text{where $\pi$ is a partition of $[0, \tau]$},
  \end{equation*}
  then the family $(A^{\pi}_{\tau})_{\pi}$ converges to $\cA_{\tau}$ in $L_m(\P)$ as
  $\abs{\pi} \to 0$.
\end{corollary}
The proof is given in Appendix~\ref{app:technical}.
\section{Integration along fractional Brownian motions}\label{sec:stochastic_integral}
The goal of this section is to prove the convergence of It\^o and Stratonovich approximations of
\begin{equation*}
  \int_0^t f(B_s) \dd B_s \quad \text{and} \quad \int_0^t f(B_s) \circ \dd B_s
\end{equation*}
along a multidimensional fractional Brownian motion $B$ with Hurst parameter $H$, using Theorem~\ref{thm:generalized_stochastic_sewing}.
For It\^o's case, we let $H \in (\frac{1}{2}, 1)$ and for Stratonovich's case, we let $H \in (\frac{1}{6}, \frac{1}{2})$.

\begin{definition}\label{def:fbm}
  Let $(\cF)_{t \in \R}$ be a filtration.
  We say that a process $B$ is an \emph{$(\cF_t)$-fractional Brownian motion with Hurst parameter $H \in (0, 1)$}
  if
  \begin{itemize}
    \item a two-sided $d$-dimensional $(\cF_t)$-Brownian motion $(W_t)_{t \in \R}$ is given;
    \item a random variable $B(0)$ is a (not necessarily centered) $\cF_0$-measurable
    $\R^d$-valued Gaussian random variable and is independent of $(W_t)_{t \in \R}$;
    \item if we set
    \begin{equation*}
      K(t, s) \defby K_H(t, s) \defby [(t-s)_+^{H-\frac{1}{2}} - (-s)_+^{H- \frac{1}{2}}],
    \end{equation*}
    then we have the \emph{Mandelbrot--van Ness representation} (\cite{mandelbrot68})
    \begin{equation}\label{eq:fbm_def}
      B_t = B(0) + \int_{\R} K_H(t, s) \dd W_s.
    \end{equation}
  \end{itemize}
\end{definition}
If $B$ has the representation \eqref{eq:fbm_def}, then
\begin{equation*}
  \expect[(B_t^i - B_s^i)(B_t^j - B_s^j)] = \delta_{ij} c_H \abs{t - s}^{2H},
  \quad c_H \defby \frac{3/2-H}{2H} \cB(2-2H, H+1/2)
\end{equation*}
where we write $B = (B^i)_{i=1}^d$ in components and $\cB$ is the Beta function
\begin{equation*}
  \cB(\alpha, \beta) \defby \int_0^1 t^{\alpha-1} (1-t)^{\beta-1} \dd t.
\end{equation*}
Regarding the expression of the constant $c_H$, see \cite[Appendix~B]{picard11}.
In particular, we have
\begin{equation}\label{eq:correlation_fbm}
  \expect[(B_s^i - B(0)^i) (B_t^i - B(0)^i) ] = \frac{c_H}{2} (t^{2H} + s^{2H} - \abs{t-s}^{2H}).
\end{equation}

In this section, we always write $B$ for an $(\cF_t)$-fractional Brownian motion.
An advantage of the representation \eqref{eq:fbm_def} is that given $v < s$, we have the decomposition
\begin{equation*}
  B_s - B(0) = \int_{-\infty}^v K(s, r) \dd W_r + \int_v^s K(s, r) \dd W_r,
\end{equation*}
where the second term $\int_v^s K(s, r) \dd W_r$ is independent of $\cF_v$.
Later we will need to estimate the correlation of
\begin{equation*}
  \int_v^s K(s, r) \dd W_r, \quad s > v.
\end{equation*}
We note that for $s \leq t$
\begin{equation*}
  \expect[\int_v^s K(s, r) \dd W_r^i \int_v^t K(t, r) \dd W_r^j ] = \delta_{ij} \int_v^s K(s, r) K(t, r) \dd r.
\end{equation*}
\begin{lemma}\label{lem:kernel_correlation}
  Let $H \neq \frac{1}{2}$.
  Let $0 \leq v < s \leq t$ be such that $t-s \leq s - v$. Then,
  \begin{multline*}
    \int_v^s K(s, r) K(t, r) \dd r  \\
      = \frac{1}{2H} (s-v)^{2H} + \frac{1}{2} (s-v)^{2H-1} (t-s) - \frac{c_H}{2} (t-s)^{2H} + g_H(v, s, t)
  \end{multline*}
  where we have
  \begin{equation*}
    \abs{g_H(v, s, t)} \lesssim_H (s - v)^{2H - 2} (t - s)^2.
  \end{equation*}
  uniformly over such $v, s, t$.
\end{lemma}
\begin{proof}
  See Appendix~\ref{app:technical}.
\end{proof}

We apply Theorem~\ref{thm:generalized_stochastic_sewing} to construct a stochastic integral
\begin{equation*}
  \int_0^T f(B_s) \dd B_s, \quad H \in (\frac{1}{2}, 1)
\end{equation*}
as the limit of Riemann type approximations.
 An advantage of the stochastic sewing lemma is that
we do not need any regularity of $f$.
We denote by $L_{\infty}(\R^d,\R^d)$ the space of bounded measurable maps from $\R^d$ to $\R^d$.
We write
  \begin{equation*}
    x \cdot y \defby \sum_{i=1}^d x^i y^i,\quad x=(x^i)_{i=1}^d, \,\, y=(y^i)_{i=1}^d
  \end{equation*}
  for the inner product of $\R^d$.
\begin{proposition}\label{prop:integral_H_greater_than_one_half}
  Let $H \in (1/2, 1)$ and $f \in L_{\infty}(\R^d, \R^d)$. Then, for any $\tau \in [0, T]$ and $m \in [2, \infty)$, the sequence
  \begin{equation*}
    \sum_{[s, t] \in \pi} f(B_s) \cdot (B_t - B_s), \quad \text{where $\pi$ is a partition of $[0, \tau]$},
  \end{equation*}
  converges in $L_m(\P)$ for every $m < \infty$ as $\abs{\pi} \to 0$.
  Furthermore, if we denote the limit by $\int_0^{\tau} f(B_r)  \dd B_r$ and if we write
  \begin{equation*}
    \int_s^t f(B_r)  \dd B_r \defby \int_0^t f(B_r)  \dd B_r - \int_0^s f(B_r)  \dd B_r,
  \end{equation*}
  then for every $0 \leq s < t \leq T$,
  \begin{equation*}
    \norm{\int_s^t f(B_r)  \dd B_r}_{L_m(\P)}
    \lesssim_{d, H, m} \norm{f}_{L_{\infty}(\R^d)} \abs{t-s}^H
  \end{equation*}
\end{proposition}
\begin{remark}
  We can replace $f(B_s)$ by $f(B_u)$ for any $u \in [s, t]$. It is well known that the sums converge to the Young integral
  if $f \in C^{\gamma} (\R)$ with $\gamma > H^{-1} (1 - H)$.
  Yaskov \cite[Theorem~3.7]{yaskov18} proves that the sums converge in some $L_p(\P)$-space if $f$ is of bounded variation.
\end{remark}
\begin{proof}
  We will not write down dependence on $d$, $H$ and $m$.
  \tymchange{The filtration $(\mathcal{F}_t)_{t \in \mathbb{R}}$ is generated by the Brownian motion $W$ appearing in 
  the Mandelbrot--van Ness representation \eqref{eq:fbm_def}.}
  We will apply Theorem~\ref{thm:generalized_stochastic_sewing} with $A_{s, t} \defby f(B_s) \cdot (B_t - B_s)$.
  Let $m \geq 2$.
  We have
  \begin{equation*}
    \norm{A_{s, t}}_{L_m(\P)} \lesssim \norm{f}_{L_{\infty}} \abs{t - s}^H.
  \end{equation*}
  To estimate conditional expectations, let $0 \leq v < s < t$ be such that $t-s \leq s-v$ and set
  \begin{equation*}
    Y_s \defby \int_{-\infty}^v K(s, r) \dd W_r, \quad
    \tilde{B}_s \defby \int_v^s K(s, r) \dd W_r.
  \end{equation*}
  We write $y_s \defby Y_s$, if conditioned under $\cF_v$.
  Namely, we write for instance
  \begin{equation*}
     \expect[g(y_s, \tilde{B}_s)] \defby \expect[g(Y_s, \tilde{B}_s) \vert \cF_v ] = \expect[g(y,\tilde B_s)]|_{y=Y_s}.
  \end{equation*}
  We are going to compute
    $\expect[A_{s, t} \vert \cF_v]$.
  Conditionally on $\cF_v$, we have the Wiener chaos expansion \tymchange{\cite[Theorem~1.1.1]{nualart06}
  \begin{equation*}
    f(B_s) = f(y_s + \tilde{B}_s) = a_0(s) + \sum_{i=1}^d a_i(s) \tilde{B}^i_s + \tilde{B}^{\perp}_s,
  \end{equation*}
  where $\tilde{B}^{\perp}_s$ is orthogonal in $L_2(\mathbb{P})$ to the subspace spanned by the constant $1$ and 
  \begin{align*}
    (\tilde{B}^i_r)_{i=1, \ldots, d; r \geq v}.
  \end{align*}
  }
  Note that
  \begin{align*}
    a_0(s) &= \expect[f(y_s + \tilde{B}_s)],  \\
    a_i(s) &= \expect[(\tilde{B}_s^i)^2]^{-1} \expect[f(y_s + \tilde{B}_s) \tilde{B}_s^i]
    \overset{\text{Lem.~\ref{lem:kernel_correlation}}}{=} 2H (s-v)^{-2H}
    \expect[f(y_s + \tilde{B}_s) \tilde{B}_s^i].
  \end{align*}
  Then, by the orthogonality of the Wiener chaos decomposition,
  \begin{equation*}
    \expect[A_{s, t} \vert \cF_v]
    = a_0(s) \cdot Y_{s,t} + \sum_{i=1}^d a_i(s) \cdot \expect[\tilde{B_s^i} \tilde{B}_{s, t}].
  \end{equation*}
  Hence, for $u \in (s, t)$,
  \begin{equation*}
    \expect[\delta A_{s, u, t} \vert \cF_v] = A^0_{s, u, t} + \sum_{i=1}^d A^i_{s, u, t},
  \end{equation*}
  where
  \begin{align*}
    A^0_{s, u, t} \defby& a_0(s) \cdot Y_{s, t} - a_0(s) \cdot Y_{s, u} - a_0(u) \cdot Y_{u, t} = (a_0(s) - a_0(u)) \cdot Y_{u, t}, \\
    A^i_{s, u, t} \defby& a_i(s) \cdot \expect[\tilde{B}_s^i \tilde{B}_{s, t}] -  a_i(s) \cdot \expect[\tilde{B}_s^i \tilde{B}_{s, u}]
    - a_i(u) \cdot \expect[\tilde{B}_u^i \tilde{B}_{u, t}] \\
    =& [a_i(s) \cdot \bm{e}_i]  \expect[\tilde{B}_s^i \tilde{B}_{s, t}^i] -  [a_i(s) \cdot \bm{e}_i] \expect[\tilde{B}_s^i \tilde{B}_{s, u}^i]
    - [a_i(u) \cdot \bm{e}_i] \expect[\tilde{B}_u^i \tilde{B}_{u, t}^i].
  \end{align*}
  Here $\bm{e}_i$ is the $i$th unit vector of $\mathbb{R}^d$.
  We first estimate $A^0_{s, u, t}$, for which we begin with estimating $a_0(s) - a_0(u)$.  We set
  \begin{equation*}
    F(m, \sigma) \defby \expect[f(m + \sigma X)], \quad m \in \R^d, \,\, \sigma \in (0, \infty),
  \end{equation*}
  where $X$ has the standard normal distribution in $\R^d$. Note that
  \begin{equation*}
    a_0(s) = F(Y_s, (2H)^{-\frac{1}{2}}(s-v)^H)
  \end{equation*}
  and similarly for $a_0(u)$.
  we have
  \begin{align*}
    \partial_{m^i} F(m, \sigma) &= \frac{1}{(2 \pi)^{\frac{d}{2}} \sigma^{d+2}} \int_{\R^d} x^i e^{-\frac{\abs{x}^2}{2 \sigma^2}} f(x + m) \dd x, \\
    \partial_{\sigma} F(m, \sigma) &= \frac{-d}{(2 \pi)^{\frac{d}{2}} \sigma^{d+1}} \int_{\R^d} f(m + x) e^{-\frac{\abs{x}^2}{2 \sigma^2}} \dd x
    + \frac{1}{(2 \pi)^{\frac{d}{2}} \sigma^{d+3}} \int_{\R^d} \abs{x}^2 f(m + x) e^{-\frac{\abs{x}^2}{2 \sigma^2}} \dd x.
  \end{align*}
  Therefore,
  \begin{equation*}
    \abs{\partial_m F(m, \sigma)} + \abs{\partial_{\sigma} F(m, \sigma)} \lesssim \norm{f}_{L_{\infty}(\R^d)} \sigma^{-1}.
  \end{equation*}
  This yields
  \begin{align*}
    \abs{a_0(s) - a_0(u)} \leq& \abs{F(Y_s, (2H)^{-\frac{1}{2}}(s - v)^H) - F(Y_u, (2H)^{-\frac{1}{2}}(s - v)^H)} \\
    &+ \abs{F(Y_u, (2H)^{-\frac{1}{2}}(s - v)^H) - F(Y_u, (2H)^{-\frac{1}{2}}(u - v)^H)} \\
    \lesssim& \norm{f}_{L_{\infty}(\R^d)}(s-v)^{-H} \abs{Y_{s, u}} + \norm{f}_{L_{\infty}(\R^d)} (s-v)^{-H} (\abs{u-v}^H - \abs{s - v}^H) \\
    \lesssim& \norm{f}_{L_{\infty}(\R^d)}(s-v)^{-H} \abs{Y_{s, u}} +  \norm{f}_{L_{\infty}(\R^d)}(s-v)^{-1}(t - s).
  \end{align*}
  Therefore,
  \begin{equation}\label{eq:young_A_0_conditional}
    \abs{A^0_{s, u, t}} \lesssim \norm{f}_{L_{\infty}(\R^d)}(s-v)^{-H} \abs{Y_{s, u}} \abs{Y_{u, t}}
    + \norm{f}_{L_{\infty}(\R^d)} (s-v)^{-1}(t - s) \abs{Y_{u, t}}.
  \end{equation}
  The random variable $Y_{s, u}$ is Gaussian and
  \begin{align}
    \expect[\abs{Y_{s, u}}^2] =& d\int_{-\infty}^v (K(s, r) - K(u, r))^2 \dd r
    = d \int_{s - v}^{\infty} ((u - s + r)^{H- \frac{1}{2}} - r^{H-\frac{1}{2}})^2 \dd r \notag \\
    \lesssim& (u - s)^2 \int_{s - v}^{\infty} r^{2H - 3} \dd r
    \label{eq:estimate_of_Y_s_u}
    \lesssim (s - v)^{2H - 2} (u - s)^2.
  \end{align}
  We have a similar estimate for $Y_{u, t}$. Therefore,
  \begin{equation*}
    \norm{A^0_{s, u, t}}_{L_m(\P)} \lesssim \norm{f}_{L_{\infty}(\R^d)}(s - v)^{H - 2} (t - s)^2 \quad \text{if } t-s \leq v - s.
  \end{equation*}

  Now we move to estimate $A^i_{s, u, t}$. By Lemma~\ref{lem:kernel_correlation}, we have
  \begin{align*}
    \expect[\tilde{B}_s^i \tilde{B}_{s, t}^i]
    &=\int_v^s K(s, r) K(t, r) \dd r - \int_v^s K(s, r) K(s, r) \dd r \\
    &= \frac{1}{2} (s - v)^{2H - 1}(t-s) + O((t-s)^{2H}).
  \end{align*}
  Therefore, if we write $a_i^i(s) \defby a_i(s) \cdot \bm{e}_i$,
  \begin{multline*}
     A^i_{s, u, t} = \frac{1}{2} \big[a_i^i(s)(s-v)^{2H - 1} - a_i^i(u)(u-v)^{2H - 1} \big](t - u)  \\
    + O((\abs{a_i^i(s)} + \abs{a_i^i(u)}) \abs{t-s}^{2H}).
  \end{multline*}
  If we set
  \begin{equation*}
    G_i(m, \sigma) \defby \sigma^{-1}
    \expect[f^i(m + \sigma X) X^i], \quad m \in \R^d, \,\, \sigma \in (0, \infty),
  \end{equation*}
  then $a_i^i(s) =  G_i(Y_s, (2H)^{-\frac{1}{2}} (s - v)^H)$ and similarly for $a_i^i(u)$.
  Since
  \begin{equation*}
    G_i(m, \sigma)
    = (2 \pi)^{-\frac{d}{2}} \sigma^{-d-2}
    \int_{\R^d} f^i(y) (y^i - m^i) e^{-\frac{\abs{y-m}^2}{2 \sigma^2}} \dd y,
  \end{equation*}
  we have
  \begin{equation*}
    (2 \pi)^{\frac{d}{2}} \sigma^2 \partial_{m^j} G_i(m, \sigma)
    =  \int_{\R^d} f^i(m + \sigma x) [- \delta_{ij} + x^i x^j] e^{-\frac{\abs{x}^2}{2}} \dd x
  \end{equation*}
  \begin{equation*}
    (2 \pi)^{\frac{d}{2}} \sigma^2 \partial_{\sigma} G_i(m, \sigma)
    =  \int_{\R^d} f^i(m + \sigma x) x^i [- (d+2) + \abs{x}^2] e^{-\frac{\abs{x}^2}{2}} \dd x.
  \end{equation*}
  Therefore,
  \begin{equation*}
    \abs{G_i(m, \sigma)} \lesssim \norm{f}_{L_{\infty}(\R^d)} \sigma^{-1},
  \end{equation*}
  \begin{equation*}
    \abs{\partial_m G_i(m, \sigma)} \lesssim \norm{f}_{L_{\infty}(\R^d)} \sigma^{-2},
    \quad \abs{\partial_{\sigma} G_i(m, \sigma)} \lesssim \norm{f}_{L_{\infty}(\R^d)} \sigma^{-2}
  \end{equation*}
  and thus
  \begin{equation*}
    \abs{a^i_i(s)} \lesssim \norm{f}_{L_{\infty}(\R^d)} (s-v)^{-H},
  \end{equation*}
  \begin{align*}
    \abs{a^i_i(s) - a^i_i(u)}
    &\lesssim \norm{f}_{L_{\infty}(\R^d)} (s-v)^{-2H} \big( \abs{Y_{s,u}} + (u-v)^H - (s-v)^H \big) \\
    &\lesssim \norm{f}_{L_{\infty}(\R^d)} (s-v)^{-2H} \big( \abs{Y_{s,u}} + (s-v)^{H-1} (u-s) \big).
  \end{align*}
  This yields
  \begin{multline*}
    \abs{A^i_{s, u, t}} \lesssim \norm{f}_{L_{\infty}(\R^d)}
    \big[ (s -v)^{-1} (t-s) \abs{y_{s, u}} +  (s-v)^{H-2}(t-s)^{2} \\
    + (s-v)^{H-2}(t-s)^{2} + (s - v)^{-H}(t - s)^{2H} \big]
  \end{multline*}
  and
  \begin{align}
    \norm{A^i_{s, u, t}}_{L_m(\P)}
    &\lesssim \norm{f}_{L_{\infty}(\R^d)} \big[(s-v)^{H-2}(t-s)^{2}
    + (s - v)^{-H}(t - s)^{2H} \big] \notag \\
    \label{eq:young_A_i_conditional}
    &\lesssim \norm{f}_{L_{\infty}(\R^d)}(s - v)^{-H}(t - s)^{2H}
  \end{align}
  if  $t-s \leq s - v$.

  Therefore, by \eqref{eq:young_A_0_conditional} and \eqref{eq:young_A_i_conditional},
  \begin{equation*}\label{eq:conditional_expectation_young_integral}
    \norm{\expect[\delta A_{s, u, t} \vert \cF_v]}_{L_m(\P)} \lesssim \norm{f}_{L_{\infty}(\R^d)} (s - v)^{-H}(t - s)^{2H}
  \end{equation*}
  if $t-s \leq s - v$.  Hence, $(A_{s, t})$ satisfies the assumption of Theorem~\ref{thm:generalized_stochastic_sewing}
  with
  \begin{equation*}
    \alpha = H, \quad \beta_1 = 2H, \quad \beta_2 = H, \quad M = 1. \qedhere
  \end{equation*}
\end{proof}

Next, we consider the case $H \in (\frac{1}{6}, \frac{1}{2})$. The following result
reproduces \cite[Theorem~3.5]{nourdin12}, with a more elementary proof and with improvement of the regularity of $f$. More precisely, the cited result requires $f \in C^6$ while here $f \in C^\gamma$ with $\gamma > \frac{1}{2H}-1$ is sufficient and thus in particular $f \in C^2$ works for all $H \in (\frac{1}{6}, \frac{1}{2})$.
We denote by $C^{\gamma} (\R^d,\R^d)$ the space of $\gamma$-Hölder maps from $\R^d$ to $\R^d$, with the norm
\begin{equation*}
  \norm{f}_{C^{\gamma} } \defby \norm{f}_{L_{\infty}(\R^d)} +
  \sup_{x \neq y} \frac{\abs{f(x) - f(y)}}{\abs{x-y}^{\gamma}}
\end{equation*}
if $\gamma \in (0, 1)$ and
\begin{equation*}
  \norm{f}_{C^{\gamma} } \defby \norm{f}_{L_{\infty}(\R^d)} +
  \sum_{i=1}^d \norm{\partial_i f}_{C^{\gamma- 1}}
\end{equation*}
if $\gamma \in (1, 2)$.
\begin{proposition}\label{prop:stratonovich_integral}
  Let $H \in (\frac{1}{6}, \frac{1}{2})$, $\gamma > \frac{1}{2H} - 1$ and $f \in C^{\gamma}(\R^d, \R^d)$.
  If $H \leq \frac{1}{4}$ and $d > 1$, assume furthermore that
  \begin{equation}\label{eq:f_is_gradient}
    \partial_i f^j = \partial_j f^i, \quad \forall i, j \in \{1, \ldots, d\}.
  \end{equation}
  Then, for every $m \in [2, \infty)$ and $\tau \in [0, T]$, the family of Stratonovich approximations
  \begin{equation*}
    \sum_{[s, t] \in \pi} \frac{f(B_s) + f(B_t)}{2} \cdot B_{s, t}, \quad
    \text{where $\pi$ is a partition of $[0, \tau]$},
  \end{equation*}
  converges in $L_m(\P)$ as $\abs{\pi} \to 0$.
  Moreover, if we denote the limit by $\int_0^{\tau} f(B_r) \circ \dd B_r$ and if we write
  \begin{equation*}
    \int_s^t f(B_r) \circ \dd B_r
    \defby \int_0^{t} f(B_r) \circ \dd B_r - \int_0^{s} f(B_r) \circ \dd B_r,
  \end{equation*}
  then for every $0 \leq s < t \leq T$ we have
  \begin{equation*}
    \Big\lVert \int_s^t f(B_r) \circ \dd B_r - \frac{f(B_s) + f(B_t)}{2} \cdot B_{s, t}\Big\rVert_{L_m(\P)}
    \lesssim_{d, H, m, \gamma} \norm{f}_{C^{\gamma}} \abs{t-s}^{(\gamma+1)H}.
  \end{equation*}
\end{proposition}
\begin{proof}
  We will not write down dependence on $d$, $H$, $m$ and $\gamma$.
  \tymchange{The filtration $(\mathcal{F}_t)_{t \in \mathbb{R}}$ is generated by the Brownian motion $W$ appearing in 
  the Mandelbrot--van Ness representation \eqref{eq:fbm_def}.}
  We can assume
  \begin{equation*}
    \gamma < \indic_{\{H > \frac{1}{4} \}} + 2  \indic_{\{H \leq \frac{1}{4} \}}.
  \end{equation*}
  We will apply Theorem~\ref{thm:generalized_stochastic_sewing} with
  \begin{equation*}
    A_{s, t} \defby (f(B_s) + f(B_t)) \cdot B_{s, t}.
  \end{equation*}
  We first claim
  \begin{equation}\label{eq:Stratonovich_A_s_u_t_one_half}
    \norm{\delta A_{s, u, t}}_{L_m(\P)} \lesssim \norm{f}_{C^{\gamma} } \abs{t-s}^{(\gamma+1)H}.
  \end{equation}
  Observe
  \begin{equation*}
    \delta A_{s, u, t} = f(B)_{u, t} \cdot B_{s, u} + f(B)_{u, s} \cdot B_{u, t}.
  \end{equation*}
  If $H > \frac{1}{4}$, the claim \eqref{eq:Stratonovich_A_s_u_t_one_half} follows from the estimates
  \begin{equation*}
    \abs{f(B)_{u, t}} \leq \norm{f}_{C^{\gamma} } \abs{B_{u, t}}^{\gamma}, \quad
    \abs{f(B)_{u, s}} \leq \norm{f}_{C^{\gamma} } \abs{B_{u, s}}^{\gamma}.
  \end{equation*}
  If $H \leq \frac{1}{4}$, then $\gamma > 1$ and we have
  \begin{equation*}
    \delta A_{s, u, t} = \Big(f(B)_{u, t} - \sum_{j=1}^d \partial_j f(B_u) B_{u, t}^j\Big) \cdot B_{s, u}
    + \Big(f(B)_{u, s} - \sum_{j=1}^d \partial_j f(B_u) B_{u, s}^j \Big) \cdot B_{u, t},
  \end{equation*}
  where \eqref{eq:f_is_gradient} is used.
  Then, the claim \eqref{eq:Stratonovich_A_s_u_t_one_half} follows again from the Hölder estimate of $f$.
  Note that the condition $\gamma > \frac{1}{2H} - 1$ is equivalent to $(\gamma + 1) H > \frac{1}{2}$.

  The rest of the proof consists of estimating the conditional expectation
  $\expect[\delta A_{s, u, t} \vert \cF_v]$.
  Let
  $t -s \leq s - v$.
  We will use the same notation as in the proof of Proposition~\ref{prop:integral_H_greater_than_one_half}.
  We have
  \begin{equation*}
    \expect[\delta A_{s, u, t} \vert \cF_v]
    = D^0_{s, u, t} + \sum_{i=1}^d D^i_{s, u, t},
  \end{equation*}
  where
  \begin{align}
    D^0_{s, u, t} \defby&
    (a_0(s) + a_0(t)) \cdot Y_{s, t} - (a_0(s) + a_0(u)) \cdot Y_{s, u} - (a_0(u) + a_0(t)) \cdot Y_{u, t} \notag \\
    =& (a_0(t) - a_0(u)) \cdot Y_{s, u} + (a_0(s) - a_0(u)) \cdot Y_{u, t}
    \label{eq:D_0_simple}.
  \end{align}
  and
  \begin{multline*}
    D^i_{s, u, t} \defby \expect[(a_i^i(s) \tilde{B}_s^i + a_i^i(t) \tilde{B}_t^i) \tilde{B}_{s, t}^i \vert \cF_v] \\
    -\expect[(a_i^i(s) \tilde{B}_s^i + a_i^i(u) \tilde{B}_u^i) \tilde{B}_{s, u}^i \vert \cF_v]
    -\expect[(a_i^i(u) \tilde{B}_u^i + a_i^i(t) \tilde{B}_t^i) \tilde{B}_{u, t}^i \vert \cF_v]
  \end{multline*}

  We first estimate $D^0_{s, u, t}$.
  Suppose that $H > \frac{1}{4}$.
  Recall
  \begin{align*}
    \partial_{m^i} F(m, \sigma) &= \frac{1}{(2 \pi)^{\frac{d}{2}} \sigma^{d+2}} \int_{\R^d} x^i e^{-\frac{\abs{x}^2}{2 \sigma^2}} [f(x + m) - f(m)] \dd x, \\
    \partial_{\sigma} F(m, \sigma) &= \frac{-d}{(2 \pi)^{\frac{d}{2}} \sigma^{d+1}} \int_{\R^d} [f(m + x) - f(m)]
    e^{-\frac{\abs{x}^2}{2 \sigma^2}} \dd x \\
    &\hspace{2cm}+ \frac{1}{(2 \pi)^{\frac{d}{2}} \sigma^{d+3}} \int_{\R^d} \abs{x}^2 [f(m + x) - f(m)] e^{-\frac{\abs{x}^2}{2 \sigma^2}} \dd x.
  \end{align*}
  Therefore,
  \begin{equation*}
    \abs{\partial_{m^i} F(m, \sigma)} + \abs{\partial_{\sigma} F(m, \sigma)} \lesssim \norm{f}_{C^{\gamma} } \sigma^{\gamma-1}.
  \end{equation*}
  This yields
  \begin{equation}\label{eq:estimate_of_abs_D_0}
    \abs{D^0_{s, u, t}}
    \lesssim \norm{f}_{C^{\gamma} } \big[ (s-v)^{(\gamma-1)H} \abs{Y_{s,u}}\abs{Y_{u,t}}
    + (s-v)^{\gamma H - 1} (t-s) (\abs{Y_{s, u}} + \abs{Y_{u, t}}) \big].
  \end{equation}
  Therefore, by \eqref{eq:estimate_of_Y_s_u},
  \begin{equation}\label{eq:estimate_D_0_s_u_t}
    \norm{D^0_{s, u, t}}_{L_m(\P)} \lesssim \norm{f}_{C^{\gamma} } (s - v)^{(\gamma + 1)H - 2} (t - s)^2.
  \end{equation}
  Now suppose that $H \leq \frac{1}{4}$. To simplify notation, we write $I(m, \sigma) \defby F(m, (2H)^{-\frac{1}{2}} \sigma)$.
  Since \eqref{eq:f_is_gradient} gives $\partial_{m^i} I^j = \partial_{m^j} I^i$ for every $i, j$, we have
  \begin{align*}
    D^0_{s, u, t} =& [I(Y_s, (u - v)^H) - I(Y_u, (u - v)^H) - \sum_{i=1}^d \partial_{m^i} I(Y_u, (u-v)^H) Y_{u, s}^i] \cdot Y_{u, t} \\
    &+ [I(Y_t, (u - v)^H) - I(Y_u, (u - v)^H) - \sum_{i=1}^d \partial_{m^i} I(Y_u, (u-v)^H) Y_{u, t}^i] \cdot Y_{s, u} \\
    &+ [I(Y_s, (s-v)^H) - I(Y_s, (u - v)^H)] \cdot Y_{u, t} \\
    &+ [I(Y_t, (t-v)^H) - I(Y_t, (u - v)^H)] \cdot Y_{s, u}.
  \end{align*}
  Since
  \begin{equation*}
    \partial_{m^i} \partial_{m^j} F(m, \sigma) = \frac{1}{(2 \pi)^{\frac{d}{2}} \sigma^{d+2}}
    \int_{\R^d} x^i e^{-\frac{\abs{x}^2}{2 \sigma^2}} [\partial_jf(x + m) - \partial_jf(m)] \dd x,
  \end{equation*}
  we have
  \begin{multline*}
    \abs{I(Y_s, (u - v)^H) - I(Y_u, (u - v)^H) - \sum_{i=1}^d \partial_{m^i} I(Y_u, (u-v)^H) Y_{u, s}^i} \\
    \lesssim \norm{f}_{C^{\gamma} } (s - v)^{(\gamma - 2)H} \abs{Y_{s, u}}^2.
  \end{multline*}
  Notice
  \begin{multline*}
    \partial_{\sigma} F(m, \sigma) = \frac{-d}{(2 \pi)^{\frac{d}{2}} \sigma^{d+1}} \int_{\R^d} [f(m + x) - f(m) - \sum_{i=1}^d \partial_i f(m) x^i]
    e^{-\frac{\abs{x}^2}{2 \sigma^2}} \dd x \\
    + \frac{1}{(2 \pi)^{\frac{d}{2}} \sigma^{d+3}} \int_{\R^d} \abs{x}^2 [f(m + x) - f(m) - \sum_{i=1}^d \partial_i f(m) x^i] e^{-\frac{\abs{x}^2}{2 \sigma^2}} \dd x.
  \end{multline*}
  Therefore,
  \begin{equation*}
    \abs{\partial_{\sigma} F(m, \sigma)} \lesssim \norm{f}_{C^{\gamma} } \sigma^{\gamma - 1}.
  \end{equation*}
  This yields
  \begin{equation*}
    \abs{I(Y_s, (s-v)^H) - I(Y_s, (u - v)^H)} \lesssim \norm{f}_{C^{\gamma} } (s - v)^{\gamma H - 1} (t - s).
  \end{equation*}
  Hence, we obtain the estimate \eqref{eq:estimate_D_0_s_u_t} when $H \leq \frac{1}{4}$.

  We move to estimate $D^i_{s, u, t}$. By using the identity,
  \begin{equation*}
    \expect[(\tilde{B}_a^i + \tilde{B}_b^i) \tilde{B}_{a, b}^i] = \expect[(\tilde{B}_b^i)^2] - \expect[(\tilde{B}_a^i)^2],
  \end{equation*}
  we obtain
  \begin{multline}\label{eq:D_i_simple}
    D^i_{s, u, t} = (a_i^i(t) - a_i^i(u)) \expect[\tilde{B}_t^i \tilde{B}_{s, t}^i]
    + (a_i^i(s) - a_i^i(u)) \expect[\tilde{B}_s^i \tilde{B}_{s, t}^i] \\
    -(a_i^i(s) - a_i^i(u)) \expect[\tilde{B}_s^i \tilde{B}_{s, u}^i]
    -(a_i^i(t) - a_i^i(u)) \expect[\tilde{B}_t^i \tilde{B}_{u, t}^i].
  \end{multline}
  Since the other terms can be estimated similarly, we only estimate
  $(a_i^i(t) - a_i^i(u)) \expect[\tilde{B}_t^i \tilde{B}_{s, t}^i]$. By Lemma~\ref{lem:kernel_correlation},
  \begin{equation*}
    \abs{\expect[\tilde{B}_t \tilde{B}_{s, t}]} \lesssim \abs{t-s}^{2H}.
  \end{equation*}
  Now we estimate $\abs{a_i^i(t) - a_i^i(u)}$.
  Recall $a_i^i(s) =  G_i(Y_s, (2H)^{-\frac{1}{2}} (s - v)^H)$,
  \begin{multline*}
    (2 \pi)^{\frac{d}{2}} \sigma^2 \partial_{m^j} G_i(m, \sigma)
    = - \delta_{ij} \int_{\R^d} [f^i(m + \sigma x) - f^i(m)] e^{-\frac{\abs{x}^2}{2}} \dd x \\
    + \int_{\R^d} [f^i(m + \sigma x) - f^i(m)] x^i x^j e^{-\frac{\abs{x}^2}{2}} \dd x,
  \end{multline*}
  \begin{multline*}
    (2 \pi)^{\frac{d}{2}} \sigma^2 \partial_{\sigma} G(m, \sigma)
    = -(d+2) \int_{\R^d} [f^i(m + \sigma x) - f^i(m)] x^i e^{-\frac{\abs{x}^2}{2}} \dd x \\
    + \int_{\R^d} [f^i(m + \sigma x) - f^i(m)] x^i \abs{x}^2 e^{-\frac{\abs{x}^2}{2}} \dd x.
  \end{multline*}
  If $H \leq \frac{1}{4}$, we can replace $f^i(m + \sigma x) - f^i(m)$ by
  \begin{equation*}
    f^i(m + \sigma x) - f^i(m) - \sum_{k=1}^d \partial_k f^i(m) \sigma x^k.
  \end{equation*}
  Therefore,
  \begin{equation*}
    \abs{\partial_{m^j} G_i(m, \sigma)} + \abs{\partial_{\sigma} G_i(m, \sigma)}
    \lesssim \norm{f}_{C^{\gamma} } \sigma^{\gamma - 2}.
  \end{equation*}
  This yields
  \begin{equation*}
    \abs{a_i^i(t) - a_i^i(u)} \lesssim \norm{f}_{C^{\gamma} } (s - v)^{(\gamma - 2)H}
    (\abs{Y_{u, t}} + (s- v)^{H-1} (t-s))
  \end{equation*}
  and hence
  \begin{equation*}
    \norm{a_i^i(t) - a_i^i(u)}_{L_m(\P)} \lesssim \norm{f}_{C^{\gamma}}
    (s-v)^{(\gamma - 1) H - 1} (t - s).
  \end{equation*}
  Therefore, we obtain
  \begin{equation}\label{eq:estimate_D_i}
    \norm{D^i_{s, u, t}}_{L_m(\P)} \lesssim
    \norm{f}_{C^{\gamma}}
    (s - v)^{(\gamma - 1) H - 1} (t - s)^{1 + 2H}.
  \end{equation}

  By \eqref{eq:estimate_D_0_s_u_t} and \eqref{eq:estimate_D_i}, we conclude
  \begin{align*}
    \norm{\expect[\delta A_{s, u, t} \vert \cF_v]}_{L_m(\P)} &\lesssim
    \norm{f}_{C^{\gamma} } [(s - v)^{(1 + \gamma) H - 2} (t - s)^2 +
    (s - v)^{(\gamma - 1) H - 1} (t - s)^{1 + 2H}] \\
    &\lesssim \norm{f}_{C^{\gamma} } (s - v)^{(\gamma - 1) H - 1} (t - s)^{1 + 2H}
  \end{align*}
  if $t-s \leq s- v$.
  Therefore, we can apply Theorem~\ref{thm:generalized_stochastic_sewing} with
  \begin{equation*}
    \alpha = 1 - (\gamma - 1)H, \quad \beta_1 = 1 + 2H, \quad \beta_2 = (\gamma + 1) H, \quad M = 1. \qedhere
  \end{equation*}
\end{proof}

\section{Local times of fractional Brownian motions}\label{sec:local_time}
In this section, we set $d = 1$ and we are interested in local times of fractional Brownian motions.
In case of a Brownian motion $W$, or more generally semi-martingales
as discussed in {\L}ochowski, Ob{\l}oj, Prömel and Siorpaes \cite{Lochowski21},
there are three major methods to construct its local time.
\begin{enumerate}[leftmargin=*]
  \item \emph{Via occupation measure.} The local time $L_T^W(\cdot)$ of $W$
  is defined as the density with respect to the Lebesgue measure of
  \begin{equation*}
    A \mapsto \int_0^T \indic_A(W_s) \dd s.
  \end{equation*}
  Heuristically,
  \begin{equation*}
    L_T^W(a) = \int_0^T \delta(W_s - a) \dd s,
  \end{equation*}
  where $\delta$ is Dirac's delta function concentrated at $0$.
  \item \emph{Via discretization.}  The local time $L_T^W(a)$ is defined by
  \begin{equation*}
    L_T^W(a) \defby \lim_{\abs{\pi} \to 0} \sum_{[s, t] \in \pi}
    \abs{W_t - a} \indic_{(\min\{W_s, W_t\}, \max\{W_s, W_t\})}(a),
  \end{equation*}
  where $\pi$ is a partition of $[0, T]$ and the convergence is in probability.
  This representation of the local time is often used in the pathwise stochastic calculus,
  see Wuermli~\cite{wuermli80}, Perkowski and Prömel~\cite{perkowski15},
  Davis, Ob{\l}oj and Siorpaes~\cite{Davis18},
  Cont and Perkowski~\cite{cont19} and Kim~\cite{kim_2022}.
  \item \emph{Via numbers of interval crossing.}
  For $n \in \N$, we set $\tau^n_0 \defby 0$ and inductively
  \begin{equation*}
    \tau^n_k \defby \inf\set{t > \tau^n_{k-1} \given W_t \in 2^{-n} \Z \setminus \{W_{\tau^n_{k-1}}\} }.
  \end{equation*}
  Then, the local time $L_T^W(a)$ is defined by
  \begin{multline*}
    L^W_T(a) \defby\\ \lim_{n \to \infty} 2^{-n} \sum_{k \in \Z} \indic_{[\frac{k}{2^n}, \frac{k+1}{2^n})}(a)
    \# \set{\ell \in \N_0 \given \{W_{\tau_\ell^n}, W_{\tau^n_{\ell+1}}\} = \{k2^{-n}, (k+1)2^{-n}\}, \tau^n_{k+1} \leq T},
  \end{multline*}
  where the convergence holds almost surely.
  See the monograph \cite{mörters_peres_2010} for the Brownian motion. For general semi-martingales, see El
  Karoui~\cite{karoui78}, Lemieux~\cite{lemieux83} and \cite{Lochowski21}.
\end{enumerate}

In case of a fractional Brownian motion, the construction of its local time via the method (a) is well-known,
see the survey \cite{geman80} and the monograph \cite{Biagini2008}. In contrast,
there are few results in the literature in which the local time of
a fractional Brownian motion is constructed via the method (b) or (c).
Because of this, the construction of the local time via the method (c) was stated as a conjecture in \cite{cont19}.
We are aware of only two results in this direction.
One is the work \cite{azais90} of Azaïs, who proves Corollary \ref{cor:number_level_crossing} below.
The other is the work \cite{mukeru17} of Mukeru,
who proves that the local time $L_T(a)$ of a fractional Brownian motion with Hurst parameter less than $\frac{1}{2}$
is represented as
\begin{equation*}
  \lim_{n \to \infty} 2^{n(2H - 1)}
  \sum_{k \leq \lfloor T 2^n \rfloor } 2 \abs{B_{k2^{-n}} - a}
  \indic_{\{(B_{k2^{-n}} - a)(B_{(k-1)2^{-n}} - a) < 0 \}}
  \quad \text{a.s.}
\end{equation*}

Our goal in this section is to give new representations of the local times of fractional Brownian motions
in the spirit of the method (b) along deterministic partitions.
The representation in Corollary~\ref{cor:local_time_discretization} is compatible with \cite[Definition~3.1]{cont19}.
\begin{theorem}\label{thm:local_time_upcrossing}
  Let $B$ be an $(\cF_t)$-fractional Brownian motion with Hurst parameter $H \neq \frac{1}{2}$,
  in the sense of Definition~\ref{def:fbm}.
  Let $m \in [2, \infty)$,  $\gamma \in [0, \infty)$ and $a \in \R$.
  If $H > \frac{1}{2}$, assume that
   $m$ satisfies
  \begin{equation}\label{eq:local_time_m_condition}
    \frac{1}{m} > 1 - \frac{1}{2H}.
  \end{equation}
  Then, as $\abs{\pi} \to 0$, where $\pi$ is a partition of $[0, T]$,
  the family of
  \begin{equation*}
    \sum_{[s, t] \in \pi, B_s < a <  B_t} (t-s)^{1 - (1 + \gamma)H} \abs{B_t - B_s}^{\gamma}
  \end{equation*}
  converges in $L_m(\P)$ to $\fc_{H, \gamma} L_T(a)$, where $L_T(a)$ is the local time of $B$ at level $a$ and
  \begin{equation*}
    \fc_{H, \gamma} \defby c_H^{\frac{1 + \gamma}{2}} \int_0^{\infty} x^{\gamma + 1} \frac{e^{-\frac{x^2}{2}}}{\sqrt{2 \pi}} \dd x.
  \end{equation*}
  Furthermore, we have
  \begin{equation}\label{eq:local_time_convergence_L_m}
    \lim_{\abs{\pi} \to 0} \expect\Big[\int_{\R} \Big\lvert\fc_{H, \gamma} L_T(x)
    -  \sum_{[s, t] \in \pi, B_s < x <  B_t} (t-s)^{1 - (1 + \gamma)H} \abs{B_t - B_s}^{\gamma}\Big\rvert^m  \dd x \Big]= 0.
  \end{equation}
\end{theorem}
\begin{remark}
  A similar result holds for a Brownian motion ($H = \frac{1}{2}$). However, we omit a proof since
  it is easier but requires a special treatment.
\end{remark}
\begin{remark}
  We can similarly prove
  \begin{equation*}
    \lim_{\abs{\pi} \to 0} \sum_{[s, t] \in \pi, B_s > a >  B_t} (t-s)^{1 - (1 + \gamma)H} \abs{B_t - B_s}^{\gamma} = \fc_{H, \gamma} L_T(a).
  \end{equation*}
  \tymchange{Consequently,} 
  \begin{equation*}
    \lim_{\abs{\pi} \to 0} \sum_{[s, t] \in \pi, \min\{B_s, B_t\} < a  < \max\{B_s, B_t\}} (t-s)^{1 - (1 + \gamma)H} \abs{B_t - B_s}^{\gamma} = 2 \fc_{H, \gamma} L_T(a),
  \end{equation*}
  where the convergence is in $L_m(\P)$.
\end{remark}
\begin{proof}
  We will not write down dependence on $H$, $\gamma$ and $m$.
  Without loss of generality, we can assume $\expect[B(0)] = 0$.
  To apply Theorem~\ref{thm:generalized_stochastic_sewing} (for $H < \frac{1}{2}$) or Corollary~\ref{cor:singular_generalized_stochastic_sewing_lemma} (for $H > \frac{1}{2}$), respectively, we set
  \begin{equation*}
    A_{s, t} \defby A_{s, t}(a) \defby (t-s)^{1 - (1 + \gamma)H} \abs{B_t - B_s}^{\gamma} \indic_{\{B_s < a < B_t \}}.
  \end{equation*}
  If we set $\cA_t \defby \fc_{H, \gamma} L_t$, it suffices to show that
  the estimates \eqref{eq:curl_A_conditional} and \eqref{eq:curl_A_one_half} are satisfied for $H < \frac{1}{2}$,
  and that the estimates \eqref{eq:curl_A_conditional_singular}, \eqref{eq:curl_A_one_half_singular}
  and \eqref{eq:curl_A_continuity_at_zero} are satisfied for $H > \frac{1}{2}$.
  Since the proof is rather long, we split the main arguments into three lemmas.

  \begin{lemma}\label{lem:local_time_level_crossing_one_half}
    We have
    \begin{equation*}
      \norm{A_{s, t}}_{L_m(\P)}
      \lesssim_T
      \begin{cases}
        \abs{t-s}^{1 - H} f_H(a),  & \text{for all } H \in (0, 1), \\
        (\expect[B(0)^2] + s^{2H})^{-\frac{1}{2m}}
        \abs{t-s}^{1 - H + \frac{H}{m}} f_H(a), & \text{if } H > \frac{1}{2},
      \end{cases}
    \end{equation*}
    where in either case there exists a constant $c = c(H) > 0$ such that
    \begin{equation}\label{eq:estimate_f_a}
      \abs{f_H(a)} \leq \exp\Big(-\frac{c a^2}{m (\expect[B(0)^2] + T^{2H})} \Big), \quad \forall a \in \R.
    \end{equation}
  \end{lemma}
  \begin{remark}
    Due to \eqref{eq:local_time_m_condition}, the exponent $1 + \frac{H}{m} - H$ is greater than $\frac{1}{2}$.
  \end{remark}
  \begin{proof}
    We have
    \begin{align*}
      \norm{A_{s,t}}_{L_m(\P)} &\leq
      (t-s)^{1 - (1 + \gamma)H} \norm{\abs{B_t - B_s}^{2\gamma}}_{L_m(\P)}^{\frac{1}{2}}
      \P(B_s < a < B_t)^{\frac{1}{2m}} \\
      &\lesssim \abs{t-s}^{1-H} e^{- \frac{a^2}{4 m c_H (\expect[B(0)^2] + T^{2H})}}.
    \end{align*}
    Now we consider the case $H > \frac{1}{2}$.

    Set
    \begin{equation*}
      \chi_0 \defby (\expect[B(0)^2] + s^{2H})^{\frac{1}{2}}, \quad
      \chi_1 \defby \frac{1}{2\chi_0} (t^{2H} - s^{2H} - \abs{t-s}^{2H}),
      \quad \chi_2^2 \defby \abs{t-s}^{2H} - \chi_1^2.
    \end{equation*}
    Since $H > \frac{1}{2}$, $\chi_1 \geq 0$, and
    \begin{equation*}
      \abs{\chi_1} \lesssim \chi_0^{-1} T^{2H - 1} \abs{t-s}, \quad 0 \leq \chi_2 \leq \abs{t-s}^H.
    \end{equation*}
    Then, if $X$ and $Y$ are two independent standard normal distributions on $\R$, we have
    \begin{equation*}
      (B_s, B_{s, t}) = \sqrt{c_H} ( \chi_0 Y, \chi_1 Y + \chi_2 X) \quad \text{in law.}
    \end{equation*}
    Therefore, if we set
    \begin{equation}\label{eq:def_of_A_hat}
      \hat{A}_{s,t} \defby \abs{B_t - B_s}^{\gamma} \indic_{\{B_s < a < B_t\}}
    \end{equation}
    \begin{multline*}
      \norm{\hat{A}_{s, t}}_{L_m(\P)}^m
      \lesssim \chi_1^{m \gamma} \expect[\abs{Y}^{m \gamma} \indic_{\{\chi_0 Y < a, a - \chi_0 Y < \chi_1 Y + \chi_2 X \}}] \\
      + \chi_2^{m \gamma} \expect[\abs{X}^{m \gamma} \indic_{\{\chi_0 Y < a, a - \chi_0 Y < \chi_1 Y + \chi_2 X \}}].
    \end{multline*}
    We first estimate $\expect[\abs{Y}^{m \gamma} \indic_{\{\chi_0 Y < a, a - \chi_0 Y < \chi_1 Y + \chi_2 X \}}]$.
    Using the estimate
    \begin{equation*}
      \P(X > x) \lesssim  e^{-\frac{x^2}{2}} \quad \text{for } x > 0,
    \end{equation*}
    we have
    \begin{equation}\label{eq:X_and_Y}
      \P(X > \chi_2^{-1} (a - \chi_0 Y - \chi_1 Y) \vert Y)
      \lesssim  e^{-\frac{1}{8} [\chi_2^{-1} (a - \chi_0 Y)]^2}
      + \indic_{\{\chi_1 Y < \frac{1}{2} (a - \chi_0 Y) \}}.
    \end{equation}
    Then,
    \begin{align*}
      \MoveEqLeft[3]
      \expect[ e^{-\frac{1}{8} [\chi_2^{-1} (a - \chi_0 Y)]^2} \abs{Y}^{m \gamma} \indic_{\{\chi_0 Y < a\}}] \\
      &= \frac{1}{\sqrt{2 \pi}} \frac{\chi_2}{\chi_0}
      \int_0^{\infty} \abs*{\frac{\chi_2 y - a}{\chi_0}}^{m \gamma} e^{-\frac{1}{8} \abs{y}^2}
      e^{-\frac{1}{2} (\frac{\chi_2 y - a}{\chi_0})^2} \dd y \\
      &\lesssim \frac{\chi_2}{\chi_0} \int_0^{\infty}  e^{-\frac{1}{8} \abs{y}^2}
      e^{-\frac{1}{4} (\frac{\chi_2 y - a}{\chi_0})^2} \dd y \\
      &\lesssim \frac{\chi_2}{\chi_0} e^{-\frac{a^2}{4(\chi_0^2 + 2 \chi_2^2)} },
    \end{align*}
    where in the third line we applied
    \begin{equation*}
      \sup_{z \in \R} \abs{z}^{m \gamma} e^{- \frac{\abs{z}^2}{4}} < \infty.
    \end{equation*}
    And,
    \begin{align*}
      \expect[ \abs{Y}^{m \gamma}\indic_{\{\chi_1 Y < \frac{1}{2} (a - \chi_0 Y) \}}]
      =
      \int_{\frac{a}{\chi_0 + 2 \chi_1} < y < \frac{a}{\chi_0}} \abs{y}^{m \gamma} \frac{e^{-\frac{y^2}{2}}}{\sqrt{2 \pi}}
      \dd y
      \lesssim \frac{\chi_1}{\chi_0} e^{-\frac{a^2}{4 \chi_0^2}}.
    \end{align*}
    Therefore,
    \begin{equation*}
      \expect[\abs{Y}^{m \gamma} \indic_{\{\chi_0 Y < a, a - \chi_0 Y < \chi_1 Y + \chi_2 X \}}]
      \lesssim
      \frac{\chi_2}{\chi_0} e^{-\frac{a^2}{4(\chi_0^2 + 2 \chi_2^2)} }
      + \frac{\chi_1}{\chi_0} e^{-\frac{a^2}{4 \chi_0^2}}
    \end{equation*}

    We now estimate $\expect[\abs{X}^{m \gamma} \indic_{\{\chi_0 Y < a, a - \chi_0 Y < \chi_1 Y + \chi_2 X \}}]$.
    Similarly to \eqref{eq:X_and_Y}, we have
    \begin{equation*}
      \expect[\abs{X}^{m \gamma} \indic_{\{X > \chi_2^{-1} (a - \chi_0 Y - \chi_1 Y)\}} \vert Y]
      \lesssim  e^{-\frac{1}{16} [\chi_2^{-1} (a - \chi_0 Y)]^2}
      + \indic_{\{\chi_1 Y < \frac{1}{2} (a - \chi_0 Y) \}}
    \end{equation*}
    and similarly
    \begin{equation*}
      \expect[\abs{X}^{m \gamma} \indic_{\{\chi_0 Y < a, a - \chi_0 Y < \chi_1 Y + \chi_2 X \}}]
      \lesssim
      \frac{\chi_2}{\chi_0} e^{-\frac{a^2}{4(\chi_0^2 + 2 \chi_2^2)} }
      + \frac{\chi_1}{\chi_0} e^{-\frac{a^2}{4 \chi_0^2}}
    \end{equation*}

    Therefore, we conclude
    \begin{equation*}
      \norm{\hat{A}_{s, t}}_{L_m(\P)}
      \lesssim \Big(\frac{\chi_1 + \chi_2}{\chi_0} \Big)^{\frac{1}{m}}
      (\chi_1^{\gamma} + \chi_2^{\gamma}) e^{-\frac{a^2}{4 m (\chi_0^2 + \chi_2^2) }}
      \lesssim_T \chi_0^{-\frac{1}{m}} \abs{t-s}^{(\gamma + \frac{1}{m})H} e^{-\frac{a^2}{4 m (\expect[B(0)^2] + T^{2H}) }},
    \end{equation*}
    which completes the proof of the lemma.
  \end{proof}

  \tymchange{Recall the Mandelbrot--van Ness representation \eqref{eq:fbm_def}, and recall that $W$ is $(\mathcal{F}_t)_{t \in \mathbb{R}}$-Brownian motion.}

  \begin{lemma}
    Let $v < s < t$ and set
    \begin{equation*}
      Y_s \defby B(0) + \int_{-\infty}^v K(s, r) \dd W_r,
      \quad
      \sigma_s^{2} \defby \expect[ (\int_v^s K(s, r) \dd W_r )^2] = \frac{1}{2H} \abs{s - v}^{2H}.
    \end{equation*}
    If $\frac{t-s}{s-v}$ is sufficiently small,
    then
    \begin{equation*}
      \expect[A_{s, t} \vert \cF_v]
      = \fc_{H, \gamma} \frac{e^{-\frac{(Y_s - a)^2}{2 \sigma_s^2}}}{ \sqrt{2 \pi} \sigma_s}(t-s)
       + R,
    \end{equation*}
    where for some $c = c(H, m) > 0$,
    \begin{equation*}
      \norm{R}_{L_m(\P)} \lesssim
      (\expect[B(0)^2] + s^{2H})^{-\frac{1}{2m}} e^{- c (\expect[B(0)]^2 + T^{2H})^{-1} a^2}
    \Big(\frac{t-s}{s-v} \Big)^{\min\{1, 2H\} - \frac{H}{m} }  \abs{t-s}^{1 - H + \frac{H}{m}}.
    \end{equation*}
  \end{lemma}
  \begin{proof}
  As in the proof of Proposition~\ref{prop:integral_H_greater_than_one_half}, for $s > v$ we set
  \begin{equation*}
    Y_s \defby B(0) +  \int_{-\infty}^v K(s, r) \dd W_r, \quad \tilde{B}_s \defby \int_v^s K(s, r) \dd W_r
  \end{equation*}
  and we write $y_s \defby Y_s$ under the conditioning of $\cF_v$.
  Then, recalling $\hat{A}_{s,t}$ from \eqref{eq:def_of_A_hat}, we have
  \begin{equation}\label{eq:local_time_A_s_t_conditioning}
    \expect[\hat{A}_{s, t} \vert \cF_v]
    = \expect[\abs{y_{s,t} + \tilde{B}_{s, t}}^{\gamma} \indic_{\{y_s + \tilde{B}_s < a, \,\, y_{s, t} + \tilde{B}_{s, t} > a -y_s - \tilde{B}_s\}}]
  \end{equation}
  To compute, we set
  \begin{equation*}
    \sigma_{s, t}^2 \defby \expect[\tilde{B}_{s, t}^2], \quad \sigma_s^2 \defby \expect[\tilde{B}_s^2],
    \quad \rho_{s, t} \defby \expect[\tilde{B}_s \tilde{B}_{s, t}].
  \end{equation*}
  By Lemma~\ref{lem:kernel_correlation},
  \begin{equation*}
    \sigma_{s, t}^2 = c_H \abs{t-s}^{2H} + O(\abs{s-v}^{2H - 2} \abs{t-s}^2),
    \quad \sigma_{s}^2 = \frac{1}{2H} \abs{s-v}^{2H}
  \end{equation*}
  and
  \begin{equation*}
    \rho_{s, t} =
      \frac{1}{2} \abs{s-v}^{2H - 1} \abs{t-s} - \frac{c_H}{2} \abs{t-s}^{2H} + O(\abs{s-v}^{2H - 2} \abs{t-s}^2).
  \end{equation*}
  We have the decomposition
  \begin{equation*}
    \tilde{B}_{s, t} = \sigma^{-2}_s \rho_{s, t} \tilde{B}_s + (\tilde{B}_{s, t} - \sigma_s^{-2} \rho_{s, t} \tilde{B}_s ),
  \end{equation*}
  where the second term is independent of $\tilde{B}_{s}$. If we set
  \begin{equation}\label{eq:kappa_asymptotics}
    \kappa_{s, t}^2 \defby \sigma_{s, t}^2 - \sigma_s^{-2} \rho_{s, t}^2
    = c_H \abs{t-s}^{2H}
    + O(\abs{s-v}^{2H - 2} \abs{t-s}^2)
     + O(\abs{s-v}^{-2H} \abs{t-s}^{4H})
  \end{equation}
  and if we write $X$ and $Y$ for two independent standard normal distributions,
  then the quantity \eqref{eq:local_time_A_s_t_conditioning} equals to
  \begin{multline}\label{eq:local_time_A_s_t_conditioning_further}
     \expect[\abs{y_{s, t} + \sigma_s^{-1} \rho_{s, t} Y + \kappa_{s, t} X}^{\gamma}
    \indic_{\{y_s + \sigma_s Y < a \}} \indic_{\{ y_{s, t} + \sigma_s^{-1} \rho_{s, t} Y + \kappa_{s, t} X >
    a -y_s - \sigma_s Y \}}] \\
    = \kappa_{s, t}^{\gamma} \expect[\abs{X + p}^{\gamma} \indic_{\{y_s + \sigma_s Y < a \}}
    \indic_{\{ X > q\}}],
  \end{multline}
  where
  \begin{align*}
    p &\defby p(Y) \defby \kappa_{s, t}^{-1} (y_{s, t} + \sigma_s^{-1} \rho_{s, t} Y), \\
    q&\defby q(Y) \defby  \kappa_{s, t}^{-1} (a -y_s - \sigma_s Y- y_{s, t} - \sigma_s^{-1} \rho_{s, t} Y ).
  \end{align*}

  For a while, assume $\gamma > 0$. Using the estimate
  \begin{equation*}
    \abs{(1 + \epsilon)^{\gamma} - 1} \lesssim
      \epsilon, \quad \text{if } \abs{\epsilon} \leq 1,
  \end{equation*}
  we have
  \begin{multline*}
    \expect[\abs{X + p}^{\gamma} \indic_{\{X > q\}} \vert Y] \\
    = \int_q^{\infty} \abs{x}^{\gamma} \frac{e^{-\frac{x^2}{2}}}{\sqrt{2 \pi}} \dd x \indic_{\{q \geq 2\abs{p}\}}
    + O(\abs{p} \int_q^{\infty} \abs{x}^{\gamma - 1} e^{-\frac{x^2}{2}} \dd x ) \indic_{\{q \geq 2 \abs{p}\}}
    + O((1 + \abs{p})^{\gamma}) \indic_{\{q < 2 \abs{p} \}}.
  \end{multline*}
  We set
  \begin{equation*}
    I_{\gamma}(y) \defby \int_{\abs{y}}^{\infty} \abs{x}^{\gamma} \frac{e^{-\frac{x^2}{2}}}{\sqrt{2 \pi}} \dd x.
  \end{equation*}
  We have for $y > 0$
  \begin{equation*}
    I_{\gamma}'(y) = -\abs{y}^{\gamma} \frac{e^{-\frac{y^2}{2}}}{\sqrt{2 \pi}}
  \end{equation*}
  and if $q \geq 2 \abs{p}$ and $y \in [q - \abs{p}, q + \abs{p}]$, then
  \begin{equation*}
    \abs{I_{\gamma}'(y)} \lesssim e^{-\frac{1}{4}(q-\abs{p})^2} \leq e^{-\frac{1}{16}q^2}
  \end{equation*}
  Therefore, if $q \geq 2 \abs{p}$, we have
  \begin{equation*}
    \abs{I_{\gamma}(q) - I_{\gamma} (p + q)} \lesssim e^{-\frac{q^2}{16}} \abs{p}.
  \end{equation*}
  Therefore,
  \begin{align}
    \MoveEqLeft[3]
    \expect[\abs{X + p}^{\gamma} \indic_{\{X > q\}} \vert Y] \notag \\
    &= I_{\gamma}(p + q) + O(\abs{p} I_{\gamma - 1}(q))\indic_{\{q \geq 2 \abs{p}\}} + O(\abs{p} e^{-\frac{q^2}{16}})
     + O((1 + \abs{p})^{\gamma}) \indic_{\{q < 2 \abs{p} \}} \notag\\
     \label{eq:expect_X_a_three_terms}
    &= I_{\gamma}(p + q) + O(\abs{p} e^{-\frac{q^2}{16} })
     + O((1 + \abs{p})^{\gamma}) \indic_{\{q < 2 \abs{p} \}}.
  \end{align}
  When $\gamma = 0$, we have
  \begin{equation*}
    \expect[\indic_{\{X > q\}} \vert Y] = I_{0}(q)
    = I_{0}(p+q) + O(\abs{p} e^{-\frac{q^2}{16} })
     + O(1) \indic_{\{q < 2 \abs{p} \}}
  \end{equation*}
  and thus \eqref{eq:expect_X_a_three_terms} holds for $\gamma = 0$.
  We estimate the expectation (with respect to $Y$) of each term.

  We have
  \begin{align*}
    \expect[I_{\gamma}(p(Y) + q(Y)) \indic_{\{y_s + \sigma_s Y \leq a \}}]
    = \frac{\kappa_{s, t}}{\sqrt{2 \pi} \sigma_s} \int_{0}^{\infty} I_{\gamma}(z)
    e^{-\frac{(\kappa_{s, t} z + y_s - a)^2}{2 \sigma_s^2}} \dd z.
  \end{align*}
  By using the estimate
  \begin{equation}\label{eq:exponential_difference}
    \abs{e^{-( z - \eta)^2} - e^{-\eta^2}} \leq 3 \abs{\eta} e^{-\frac{\eta^2}{4}} \abs{z}
    + 2 \indic_{\{\abs{\eta} \leq 2  \abs{z}\}},
  \end{equation}
  we obtain
  \begin{align*}
    \MoveEqLeft[3]
    \expect[I_{\gamma}(p(Y) + q(Y)) \indic_{\{y_s + \sigma_s Y \leq 0 \}}] \\
    &= \frac{\kappa_{s, t} e^{-\frac{(y_s - a)^2}{2 \sigma_s^2}}}{\sqrt{2 \pi} \sigma_s} \int_{0}^{\infty} I_{\gamma}(z) \dd z
    + O(\sigma_s^{-2} \kappa_{s, t}^2e^{-\frac{(y_s - a)^2}{16 \sigma_s^2}})
    + O(\sigma_s^{-1} \kappa_{s,t} e^{-\frac{(y_s - a)^2}{8 \kappa_{s,t}^2}})\\
    &= \frac{\kappa_{s, t} e^{-\frac{(y_s - a)^2}{2 \sigma_s^2}}}{\sqrt{2 \pi} \sigma_s} \int_0^{\infty}
    \abs{x}^{\gamma + 1} \frac{e^{-\frac{x^2}{2}}}{\sqrt{2 \pi}} \dd x
    + O(\sigma_s^{-2} \kappa_{s, t}^2e^{-\frac{(y_s - a)^2}{16 \sigma_s^2}})
    + O(\sigma_s^{-1} \kappa_{s,t} e^{-\frac{(y_s - a)^2}{8 \kappa_{s,t}^2}}).
  \end{align*}

  Next, we estimate the second term of \eqref{eq:expect_X_a_three_terms}. We have for $n \in \{0, 1\}$,
  \begin{align*}
    \MoveEqLeft[3]
    \expect[\abs{Y}^n e^{-\frac{q(Y)^2}{16}}] \\
    &= \frac{\kappa_{s, t}}{\sigma_s + \sigma_s^{-1} \rho_{s,t}}
    \int_{\R} \abs*{\frac{\kappa_{s,t} z + a - y_{s,t} - y_s}{\sigma_s + \sigma_s^{-1} \rho_{s,t}} }^n
    e^{-\frac{z^2}{16}} \frac{1}{\sqrt{2 \pi}} e^{-\frac{1}{2}
    \Big(\frac{\kappa_{s,t} z + a - y_{s,t} - y_s}{\sigma_s + \sigma_s^{-1} \rho_{s,t}} \Big)^2} \dd z \\
    &\lesssim \frac{\kappa_{s, t}}{\sigma_s + \sigma_s^{-1} \rho_{s, t}}
    \Big(\frac{\kappa_{s,t}  + \abs{a - y_t}}{\sigma_s + \sigma_s^{-1} \rho_{s,t}} \Big)^n \\
    &\hspace{1cm}\times
    \Big[
      e^{-\frac{(y_t - a)^2}{2 (\sigma_s + \sigma_s^{-1} \rho_{s,t})^2}}
      + e^{-\frac{(y_t - a)^2}{16 (\sigma_s + \sigma_s^{-1} \rho_{s,t})^2}}
      \frac{\kappa_{s,t}}{\sigma_s + \sigma_s^{-1} \rho_{s,t}}
      + e^{-\frac{(y_t - a)^2}{32 \kappa_{s,t}^2}}
    \Big],
  \end{align*}
  where we applied \eqref{eq:exponential_difference} to get the last inequality.
  Therefore, we obtain
  \begin{equation*}
    \expect[\abs{p(Y)} e^{-\frac{q(Y)^2}{16}}]
    \lesssim \Big(\frac{\abs{y_{s, t}}}{\sigma_s + \sigma_s^{-1} \rho_{s, t}} + \sigma_s^{-1} \abs{\rho_{s, t}}
    \frac{\kappa_{s,t}  + \abs{y_t - a}}{(\sigma_s + \sigma_s^{-1} \rho_{s,t})^2} \Big)
      e^{-\frac{(y_t - a)^2}{32 (\sigma_s + \sigma_s^{-1} \rho_{s,t})^2}}
  \end{equation*}

  Finally, we estimate the third term of \eqref{eq:expect_X_a_three_terms}.
  Suppose that $\frac{t-s}{s-v}$ is so small that
  $\abs{\sigma_s^{-2} \rho_{s, t}} \leq \frac{1}{24}$, and then
  we have
  \begin{align*}
    \indic_{\{y_s + \sigma_s Y \leq a \}} \indic_{\{q(Y) \leq 2 \abs{p(Y)}\}}
    &\leq \indic_{\{\abs{a - y_s - \sigma_s Y} \leq 3 \abs{y_{s,t} + \sigma_s^{-1} \rho_{s,t} Y}\}} \\
    &\leq \indic_{\{\abs{\sigma_s^{-1} (a - y_s) - Y} \leq 6 \sigma_s^{-1}\abs{y_{s,t}}
    + 6 \sigma_s^{-3} \abs{\rho_{s,t}} \abs{a - y_s} \}}.
  \end{align*}
  Hence,
  \begin{multline*}
    \expect[(1 + p(Y))^{\gamma} \indic_{\{y_s + \sigma_s Y \leq a \}} \indic_{\{q(Y) \leq 2 \abs{p(Y)}\}}] \\
    \lesssim
    (1 + \kappa_{s,t}^{-1} (\abs{y_{s,t}} + \sigma_s^{-1} \rho_{s,t} ))^{\gamma} \\
    \times \Big(e^{-\frac{(y_s - a)^2}{5 \sigma_s^2}}
    + \indic_{\{\abs{\sigma_s^{-1} \abs{y_s - a}}
    \leq 12 \sigma_s^{-1} \abs{y_{s,t}} + 12 \sigma_s^{-3} \abs{\rho_{s,t}} \abs{y_s - a} \}} \Big)
    (\sigma_s^{-1} \abs{y_{s,t}} + \sigma_s^{-3} \abs{\rho_{s,t}} \abs{y_s -a }).
  \end{multline*}
  and
  \begin{equation*}
    \indic_{\{\abs{\sigma_s^{-1} (y_s - a)} \leq 12 \sigma_s^{-1} \abs{y_{s,t}}
    + 12 \sigma_s^{-3} \abs{\rho_{s,t}} \abs{y_s - a} \}}
    \leq \indic_{\{\abs{y_s - a} \leq 24  \abs{y_{s,t}} \}}.
  \end{equation*}
  This gives the estimate of the third term.

  In summary, recalling that $\expect[\hat{A}_{s,t} \vert \cF_v]$ equals to \eqref{eq:local_time_A_s_t_conditioning_further},
  we obtain
  \begin{equation*}
    \expect[\hat{A}_{s, t} \vert \cF_v]
    = \frac{\kappa_{s,t}^{\gamma+1} e^{-\frac{(Y_s - a)^2}{2 \sigma_s^2}}}{ \sqrt{2 \pi} \sigma_s} \int_0^{\infty}
    \abs{x}^{\gamma+1} \frac{e^{-\frac{x^2}{2}}}{\sqrt{2 \pi}} \dd x + R_1,
  \end{equation*}
  where
  \begin{align*}
    \abs{R_1} \lesssim& \sigma_s^{-2} \kappa_{s, t}^{\gamma + 2} e^{-\frac{(Y_s - a)^2}{16 \sigma_s^2}}
    + \sigma_s^{-1} \kappa_{s, t}^{\gamma + 1} e^{-\frac{(Y_s - a)^2}{8 \kappa_{s,t}^2}} \\
    &+ \kappa_{s,t}^{\gamma}
    \Big(\frac{\abs{Y_{s, t}}}{\sigma_s + \sigma_s^{-1} \rho_{s, t}} + \sigma_s^{-1} \abs{\rho_{s, t}}
    \frac{\kappa_{s,t}  + \abs{Y_t - a}}{(\sigma_s + \sigma_s^{-1} \rho_{s,t})^2} \Big)
    e^{-\frac{(Y_t - a)^2}{32 (\sigma_s + \sigma_s^{-1} \rho_{s,t})^2}}\\
    &+(\kappa_{s,t}^{\gamma} + (\abs{Y_{s,t}} + \sigma_s^{-1} \rho_{s,t})^{\gamma})
    \Big(e^{-\frac{(Y_s - a)^2}{5 \sigma_s^2}}
    + \indic_{\{\abs{Y_s - a} \leq 24  \abs{Y_{s,t}} \}} \Big) \\
    &\hspace{6cm} \times(\sigma_s^{-1} \abs{Y_{s,t}} + \sigma_s^{-3} \abs{\rho_{s,t}} \abs{Y_s - a}).
  \end{align*}
  Let us estimate $\norm{R_1}_{L_m(\P)}$.
  Recall that
  \begin{equation*}
    \kappa_{s, t} \lesssim \abs{t-s}^{H},
    \quad \sigma_s \lesssim \abs{s - v}^{H},
  \end{equation*}
  and
  \begin{equation*}
    \abs{\rho_{s, t}} \lesssim
    \begin{cases}
      \abs{s - v}^{2H -1} \abs{t-s}, \quad &H > \frac{1}{2}, \\
       \abs{t-s}^{2H}, \quad &H < \frac{1}{2}.
    \end{cases}
  \end{equation*}
  We have the estimate \eqref{eq:estimate_of_Y_s_u} of $Y_{s, t}$.
  Since
  \begin{equation*}
    \expect[Y_s^2] = \expect[B(0)^2] + c_H s^{2H} - \frac{1}{2H} \abs{s-v}^{2H} \gtrsim \expect[B(0)^2] + s^{2H}
    =: \chi_s^2,
  \end{equation*}
  there is a constant $c = c(H) > 0$ such that
  \begin{equation}\label{eq:Y_and_exp_Y}
    \expect[\abs{Y_s - a}^n e^{-\frac{(Y_s - a)^2}{\sigma^2}}]
    \lesssim_{n} \chi_s^{-1} \sigma^{n+1} e^{-\frac{c a^2}{\chi_s^2 + \sigma^2}},
  \end{equation}
  \begin{equation}\label{eq:Y_s_t_and_exp_Y}
    \expect[\abs{Y_{s,t}}^n e^{-\frac{(Y_t - a)^2}{\sigma^2}}]
    \lesssim_{n} \chi_s^{-1} \sigma
     \abs{v-s}^{n(H - 1)} \abs{t-s}^n e^{-\frac{c a^2}{\chi_s^2 + \sigma^2}}.
  \end{equation}
  Therefore, for some constant $c_1 = c(H, m)>0$,
  \begin{multline*}
    \norm{\sigma_s^{-2} \kappa_{s, t}^{\gamma + 2} e^{-\frac{(Y_s - a)^2}{2 \sigma_s^2}}}_{L_m(\P)} \\
    \lesssim \chi_s^{-\frac{1}{m}} \sigma_s^{-2 + \frac{1}{m}} \kappa_{s,t}^{\gamma+2}
    e^{-\frac{c_1 a^2}{\chi_s^2}}
    \lesssim \chi_s^{-\frac{1}{m}} \abs{s-v}^{-(2-\frac{1}{m})H} \abs{t-s}^{(\gamma+2)H} e^{-\frac{c_1 a^2}{\chi_s^2}},
  \end{multline*}
  \begin{align*}
    \MoveEqLeft[3]
    \norm{\kappa_{s,t}^{\gamma}
    \frac{\abs{Y_{s, t}}}{\sigma_s + \sigma_s^{-1} \rho_{s, t}}
      e^{-\frac{(Y_t - a)^2}{32 (\sigma_s + \sigma_s^{-1} \rho_{s,t})^2}}}_{L_m(\P)} \\
    &\lesssim \chi_s^{-\frac{1}{m}}
    \kappa_{s,t}^{\gamma} \sigma_s^{-1 + \frac{1}{m}} \abs{s-v}^{H-1} \abs{t-s}e^{-\frac{c_1 a^2}{\chi_s^2}}
    \lesssim \chi_s^{-\frac{1}{m}} \abs{s-v}^{\frac{H}{m} - 1} \abs{t-s}^{\gamma H + 1}e^{-\frac{c_1 a^2}{\chi_s^2}},
  \end{align*}
  \begin{multline*}
    \norm{
     \kappa_{s,t}^{\gamma}
    \sigma_s^{-1} \abs{\rho_{s, t}}
    \frac{\kappa_{s,t}  + \abs{Y_t - a}}{(\sigma_s + \sigma_s^{-1} \rho_{s,t})^2}
      e^{-\frac{(Y_t - a)^2}{32 (\sigma_s + \sigma_s^{-1} \rho_{s,t})^2}}}_{L_m(\P)}
    \lesssim \chi_s^{-\frac{1}{m}} \kappa_{s, t}^{\gamma} \sigma_s^{-2 + \frac{1}{m}} \abs{\rho_{s,t}} e^{-\frac{c_1 a^2}{\chi_s^2}} \\
    \lesssim
    \chi_s^{-\frac{1}{m}} e^{-\frac{c_1 a^2}{\chi_s^2}}
    \begin{cases}
       \abs{s-v}^{-1 + \frac{H}{m}} \abs{t-s}^{\gamma H + 1}, \quad &H > \frac{1}{2}, \\
       \abs{s-v}^{-2H + \frac{H}{m} } \abs{t-s}^{(\gamma+2)H}, & H < \frac{1}{2},
    \end{cases}
  \end{multline*}
  \begin{align*}
    \MoveEqLeft[3]
    \norm{
    \kappa_{s,t}^{\gamma}
    e^{-\frac{(Y_s - a)^2}{5 \sigma_s^2}}
    (\sigma_s^{-1}\abs{Y_{s,t}} + \sigma_s^{-3} \abs{\rho_{s,t}} \abs{Y_s - a})
    }_{L_m(\P)} \\
    &\lesssim \chi_s^{-1} \kappa_{s,t}^{\gamma}(\sigma_s^{\frac{1}{m} - 1} \abs{s-v}^{H-1} \abs{t-s}
    + \sigma_s^{-2+\frac{1}{m}} \abs{\rho_{s,t}}) e^{-\frac{c_1 a^2}{\chi_s^2}}\\
    &\lesssim
    \chi_s^{-\frac{1}{m}} e^{-\frac{c_1 a^2}{\chi_s^2}}
    \begin{cases}
       \abs{s-v}^{\frac{H}{m} - 1} \abs{t-s}^{2-H}, &H > \frac{1}{2}, \\
       ( \abs{s-v}^{\frac{H}{m} -1} \abs{t-s}^{2-H}
      + \abs{s-v}^{(-2+\frac{1}{m})H} \abs{t-s}^{1+H}), &H < \frac{1}{2},
    \end{cases}
  \end{align*}
  \begin{align*}
    \MoveEqLeft[3]
    \norm{
    \abs{Y_{s,t}}^{\gamma}
    e^{-\frac{(Y_s - a)^2}{5 \sigma_s^2}}
    (\sigma_s^{-1} \abs{Y_{s,t}} + \sigma_s^{-3} \abs{\rho_{s,t}} \abs{Y_s - a})
    }_{L_m(\P)} \\
    &\lesssim \sigma_s^{-1} \norm{
    \abs{Y_{s,t}}^{\gamma + 1}
    e^{-\frac{(Y_s - a)^2}{5 \sigma_s^2}}}_{L_m(\P)}
    + \sigma_s^{-2} \rho_{s, t} \norm{
    \abs{Y_{s,t}}^{\gamma}
    e^{-\frac{(Y_s - a)^2}{6 \sigma_s^2}}}_{L_m(\P)} \\
    &\lesssim
    \chi_s^{-\frac{1}{m}} e^{-\frac{c_1 a^2}{\chi_s^2}}
    \big[
    \sigma_s^{\frac{1}{m} - 1} \abs{s-v}^{(\gamma +1)(H - 1)} \abs{t-s}^{(\gamma + 1) } \\
    &\hspace{3cm}+
    \sigma_s^{-2+\frac{1}{m}} \abs{\rho_{s,t}} \abs{v-s}^{\gamma(H-1)} \abs{t-s}^{\gamma} \big]\\
    &\lesssim_{H, m}
    \chi_s^{-\frac{1}{m}}
    e^{-\frac{c_1 a^2}{\chi_s^2}} \\
    &\times
    \begin{cases}
       \abs{s-v}^{(\gamma+\frac{1}{m})H - (\gamma+1)} \abs{t-s}^{\gamma+1}, &H> \frac{1}{2}, \\
       (\abs{s-v}^{(\gamma+\frac{1}{m})H - (\gamma+1)} \abs{t-s}^{\gamma+1}
      + \abs{s-v}^{(-2 + \frac{1}{m} + \gamma)H - \gamma} \abs{t-s}^{\gamma+ 2H}), &H< \frac{1}{2},
    \end{cases}
  \end{align*}
  \begin{align*}
    \MoveEqLeft
    \norm{
    (\sigma_s^{-1} \rho_{s,t})^{\gamma}
    e^{-\frac{Y_s^2}{5 \sigma_s^2}}
    (\sigma_s^{-1} \abs{Y_{s,t}} + \sigma_s^{-3} \abs{\rho_{s,t}} \abs{Y_s})
    }_{L_m(\P)} \\
    &\lesssim \chi_s^{-\frac{1}{m}} e^{-\frac{c_1 a^2}{\chi_s^2}}
    \big[
    \sigma_s^{\frac{1}{m} - 1 - \gamma} \abs{\rho_{s,t}}^{\gamma} \abs{s-v}^{H-1} \abs{t-s}
    +  \sigma_s^{-2 + \frac{1}{m} - \gamma} \abs{\rho_{s,t}}^{\gamma+ 1} \big] \\
    &\lesssim \chi_{s}^{-\frac{1}{m}}e^{-\frac{c_1 a^2}{\chi_s^2}} \\
    &\times
    \begin{cases}
      \abs{s-v}^{(\gamma + \frac{1}{m})H - (\gamma + 1)} \abs{t-s}^{\gamma + 1}, & H > \frac{1}{2}, \\
      \abs{s-v}^{(\frac{1}{m} - \gamma)H - 1} \abs{t-s}^{1+ 2 H \gamma}
      + \abs{s-v}^{(-2 + \frac{1}{m} - \gamma)H} \abs{t-s}^{2H(\gamma+1)}, & H < \frac{1}{2}, \\
    \end{cases}
  \end{align*}
  and finally
  \begin{align*}
    \MoveEqLeft
    \norm{
      (\kappa_{s,t}^{\gamma} + (\abs{Y_{s,t}} + \sigma_s^{-1} \rho_{s,t} )^{\gamma})
     \indic_{\{\abs{Y_s - a} \leq 24  \abs{Y_{s,t}} \}} (\sigma_s^{-1}\abs{Y_{s,t}} + \sigma_s^{-3} \abs{\rho_{s,t}} \abs{Y_s})
    }_{L_m(\P)} \\
     \lesssim&
    \sigma_s^{-1} (\kappa_{s,t}^{\gamma} + \sigma_s^{-\gamma} \abs{\rho_{s,t}}^{\gamma})
    \norm{Y_{s,t} \indic_{\{\abs{Y_s - a} \leq 24  \abs{Y_{s,t}} \}}}_{L_m(\P)} \\
    &+ \sigma_s^{-1} \norm{\abs{Y_{s,t}}^{\gamma+1} \indic_{\{\abs{Y_s - a} \leq 24  \abs{Y_{s,t}} \}}}_{L_m(\P)} \\
    \lesssim& \sigma_s^{-1} (\kappa_{s,t}^{\gamma} + \sigma_s^{-\gamma} \abs{\rho_{s,t}}^{\gamma})
    \chi_s^{-\frac{1}{m}} e^{-\frac{c_1 a^2}{\chi_s^2}}
    \abs{s-v}^{(H-1)(1+\frac{1}{m})} \abs{t-s}^{1 + \frac{1}{m}} \\
    &+ \sigma_s^{-1} \chi_s^{-\frac{1}{m}} e^{-\frac{c_1 a^2}{\chi_s^2}}
    \abs{s-v}^{(H - 1)(1+ \gamma + \frac{1}{m})} \abs{t-s}^{1 + \gamma + \frac{1}{m}} \\
    \lesssim& \chi_s^{-\frac{1}{m}}e^{-\frac{c_1 a^2}{\chi_s^2}}
    \big[ \abs{s-v}^{\frac{H}{m} - (1 + \frac{1}{m})} \abs{t-s}^{1 + \frac{1}{m} + \gamma H}
    + \abs{s-v}^{(\gamma + \frac{1}{m})H - (1 + \gamma + \frac{1}{m})} \abs{t-s}^{1 + \gamma + \frac{1}{m}}  \\
    &+ \indic_{\{H < \frac{1}{2} \}} \abs{s-v}^{(\frac{1}{m} - \gamma)H - (1 + \frac{1}{m})}
    \abs{t-s}^{1 + 2H \gamma + \frac{1}{m}} \big].
  \end{align*}

  After this long calculation, we conclude
  \begin{equation}\label{eq:estimate_of_R_1}
    \norm{R_1}_{L_m(\P)} \lesssim  \chi_s^{-\frac{1}{m}} e^{-\frac{c_1 a^2}{\chi_s^2}}
    \Big(\frac{t-s}{s-v} \Big)^{\min\{1, 2H\} - \frac{H}{m}} \abs{t-s}^{(\gamma + \frac{1}{m}) H}
  \end{equation}
  if $\frac{t-s}{s-v}$ is sufficiently small.

  By \eqref{eq:kappa_asymptotics}, we have
  \begin{equation*}
    \frac{\kappa_{s,t}^{\gamma+1} e^{-\frac{(y_s - a)^2}{2 \sigma_s^2}}}{ \sqrt{2 \pi} \sigma_s} \int_0^{\infty}
    \abs{x}^{\gamma+1} \frac{e^{-\frac{x^2}{2}}}{\sqrt{2 \pi}} \dd x
    = \fc_{H, \gamma} \frac{e^{-\frac{(y_s - a)^2}{2 \sigma_s^2}} \abs{t-s}^{(\gamma+1)H}}{ \sqrt{2 \pi} \sigma_s} + R_2,
  \end{equation*}
  where
  \begin{equation}\label{eq:estimate_of_R_2}
    \abs{R_2} \lesssim_H \frac{e^{-\frac{(y_s - a)^2}{2 \sigma_s^2}} }{ \sigma_s}
    \Big(\frac{t-s}{s-v} \Big)^{2  \min\{H, 1-H\}} \abs{t-s}^{(\gamma+1)H} .
  \end{equation}
  Therefore,
  \begin{equation*}
    \norm{R_2}_{L_m(\P)}
    \lesssim_{H} \chi_s^{-\frac{1}{m}}
     e^{-\frac{c_1 a^2}{\chi_s^2}}
    \begin{cases}
      \abs{s-v}^{(\frac{1}{m} + 1)H - 2} \abs{t-s}^{2 + (\gamma - 1) H}, &H > \frac{1}{2},\\
      \abs{s-v}^{(\frac{1}{m} - 3)H} \abs{t-s}^{(\gamma + 3) H}, &H < \frac{1}{2}.
    \end{cases}
  \end{equation*}
  This completes the proof of the lemma.
  \end{proof}
  \begin{lemma}
    We have
    \begin{equation}\label{eq:local_time_one_half}
      \norm{L_t(a) - L_s(a)}_{L_m(\P)} \lesssim_T
      \begin{cases}
        \abs{t-s}^{1 - H} f_H(a),  & \hspace{-2cm}\text{for all } H \in (0, 1), \\
        (\expect[B(0)^2] + s^{2H})^{-\frac{1}{2m}}
        \abs{t-s}^{1 - H + \frac{H}{m}} f_H(a), & \text{if } H > \frac{1}{2},
      \end{cases}
    \end{equation}
    where $f_H(a)$ satisfies the estimate \eqref{eq:estimate_f_a}.
    Moreover, if $\frac{t-s}{s-v}$ is sufficiently small, then
    \begin{equation*}
      \expect[L_t(a) - L_s(a) \vert \cF_v] =
      \frac{e^{-\frac{\abs{Y_s - a}^2}{2 \sigma_s^2}}}{\sqrt{2 \pi} \sigma_s} (t-s)
      + \tilde{R},
    \end{equation*}
    where for some $c = c(H, m) > 0$,
    \begin{equation}\label{eq:tilde_R}
      \norm{\tilde{R}}_{L_m(\P)} \lesssim (\expect[B(0)^2] + s^{2H})^{-\frac{1}{2m}}
      e^{-c (\expect[B(0)^2] + T^{2H})^{-1} a^2} \Big(\frac{t-s}{s-v} \Big)^{\frac{1}{m} + (1 - \frac{1}{m})H} \abs{t-s}^{1 - H + \frac{H}{m}}.
    \end{equation}
  \end{lemma}
  \begin{proof}
    The estimate in~\eqref{eq:local_time_one_half} follows from \cite[(3.38)]{ayache08}.
    However, since this is not entirely obvious, we sketch here an alternative derivation, which is motivated by \cite{butkovsky2023stochastic}.
      In view of the formal expression $L_t(a) = \int_0^t \delta_a(B_r) \dd r$, we set
      \begin{equation*}
        \bar{A}_{s, t}(a) \defby \int_s^t \expect[\delta_a(B_r) \vert \cF_s] \dd r
        \defby \int_s^t \sqrt{\frac{H}{\pi c_H (r-s)^{2H}}} e^{-\frac{ H(Y_s - a)^2}{c_H (r-s)^{2H}}} \dd r.
      \end{equation*}
      We note $\expect[\delta \bar{A}_{s,u,t}(a) \vert \cF_s] =0$.
      By Lê's stochastic sewing lemma~\cite{khoa20}, to prove \eqref{eq:local_time_one_half}, it suffices to show
      \begin{itemize}
        \item the estimate
        \begin{equation*}
          \norm{e^{-\frac{ H(Y_s - a)^2}{c_H (r-s)^{2H}}}}_{L_m(\P)} \lesssim_{T}
          \begin{cases}
            f_H(a),  & H < \frac{1}{2}, \\
            (\expect[B(0)^2] + s^{2H})^{-\frac{1}{2m}}
            \abs{r-s}^{\frac{H}{m}} f_H(a), & H > \frac{1}{2},
          \end{cases}
        \end{equation*}
        \item and the identity $L_t(a) = \lim_{\abs{\pi} \to 0} \sum_{[u,v] \in \pi} \bar{A}_{u, v}(a)$, where $\pi$ is a partition of $[0, t]$.
      \end{itemize}
      The first point is essentially given in Lemma~\ref{lem:local_time_level_crossing_one_half}. The second point follows from the identity
      \begin{equation*}
        \int_{\R} \Big(\sum_{[u,v] \in \pi} \bar{A}_{u,v}(a) \Big) f(a) \dd a
        = \sum_{[u,v] \in \pi} \int_u^v \expect[f(B_r) \vert \cF_u] \dd r.
      \end{equation*}

    Thus, we now focus on the estimate \eqref{eq:tilde_R}.
    We have the identity
    \begin{equation*}
      \expect[L_t - L_s \vert \cF_v]
      = \int_s^t \frac{e^{-\frac{\abs{Y_r - a}^2}{2 \sigma_r^2}}}{\sqrt{2 \pi} \sigma_r} \dd r.
    \end{equation*}
    Indeed, we can convince ourselves of the validity of the identity from the formal expression
    \begin{equation*}
      L_t - L_s = \int_s^t \delta(B_r - a) \dd r.
    \end{equation*}
    We have the decomposition
    \begin{multline*}
      \frac{e^{-\frac{\abs{Y_r - a}^2}{2 \sigma_r^2}}}{\sigma_r}
      - \frac{e^{-\frac{\abs{Y_s - a}^2}{2 \sigma_s^2}}}{\sigma_s}
      = \Big[\frac{e^{-\frac{\abs{Y_r - a}^2}{2 \sigma_r^2}}}{\sigma_r}
      - \frac{e^{-\frac{\abs{Y_r - a}^2}{2 \sigma_r^2}}}{\sigma_s} \Big] \\
      + \Big[ \frac{e^{-\frac{\abs{Y_r - a}^2}{2 \sigma_r^2}}}{\sigma_s}
      - \frac{e^{-\frac{\abs{Y_r - a}^2}{2 \sigma_s^2}}}{\sigma_s} \Big]
      + \Big[\frac{e^{-\frac{\abs{Y_r - a}^2}{2 \sigma_s^2}}}{\sigma_s}
      - \frac{e^{-\frac{\abs{Y_s - a}^2}{2 \sigma_s^2}}}{\sigma_s}  \Big]
      =: R_3 + R_4 + R_5.
    \end{multline*}
    By \eqref{eq:Y_and_exp_Y}, we obtain
    \begin{equation*}
      \norm{R_3}_{L_m(\P)}
      \lesssim \chi_s^{-\frac{1}{m}} e^{-\frac{c a^2}{2 \chi_s^2}}
      \sigma_r^{\frac{1}{m}}\abs{\sigma_r^{-1} - \sigma_s^{-1}}
      \lesssim \chi_s^{-\frac{1}{m}} e^{-\frac{c a^2}{2 \chi_s^2}}
      \abs{s-v}^{-1-H} \abs{r-s}^{1+\frac{H}{m}}.
    \end{equation*}
    We have
    \begin{equation*}
      \sigma_s \abs{R_4} =  e^{-\frac{\abs{Y_r}^2}{2 \sigma_r^2}}
      \abs{1 - e^{-\frac{Y_r^2}{2} (\frac{1}{\sigma_s} - \frac{1}{\sigma_r})}}
      \lesssim \abs{\sigma_s^{-2} - \sigma_r^{-2}} Y_r^2 e^{-\frac{\abs{Y_r}^2}{2 \sigma_r}}.
    \end{equation*}
    Hence, by \eqref{eq:Y_and_exp_Y},
    \begin{equation*}
      \norm{R_4}_{L_m(\P)}
      \lesssim \chi_s^{-\frac{1}{m}} e^{-\frac{c a^2}{2 \chi_s^2}}\sigma_s^{-1 + \frac{1}{m}} (1 - \sigma_s^2 \sigma_r^{-2})
      \lesssim \chi_s^{-\frac{1}{m}} e^{-\frac{c a^2}{2 \chi_s^2}}
      \abs{s-v}^{-1 + (-1+\frac{1}{m})H} \abs{t-s}.
    \end{equation*}
    To estimate $R_5$, observe
    \begin{equation*}
      \sigma_s \abs{R_6} \lesssim e^{-\frac{\abs{Y_s - a}^2}{8 \sigma_s^2}} \sigma_s^{-1} \abs{Y_{s,r}}
      + \indic_{\{\abs{Y_s - a} \leq 2 \abs{Y_{s,r}} \}}.
    \end{equation*}
    By \eqref{eq:Y_s_t_and_exp_Y},
    \begin{align*}
      \norm{e^{-\frac{\abs{Y_s - a}^2}{8 \sigma_s^2}} \sigma_s^{-2} \abs{Y_{s,r}}}_{L_m(\P)}
      &\lesssim \chi_s^{-\frac{1}{m}} e^{-\frac{c a^2}{2 \chi_s^2}}\sigma_s^{-2 + \frac{1}{m}} \abs{s-v}^{H-1} \abs{t-s} \\
      &\lesssim \chi_s^{-\frac{1}{m}} e^{-\frac{c a^2}{2 \chi_s^2}}\abs{s-v}^{(-1 + \frac{1}{m})H - 1} \abs{t-s}.
    \end{align*}
    To estimate $\P(\abs{Y_s - a} \leq 2 \abs{Y_{s,r}})$, consider the decomposition
    \begin{equation*}
      Y_{s,r} = \expect[Y_s^2]^{-1} \expect[Y_s Y_{s,r}] Y_s + (Y_{s,r} - \expect[Y_s^2]^{-1} \expect[Y_s Y_{s,r}] Y_s).
    \end{equation*}
    If $\frac{t-s}{s-v}$ is sufficiently small, then $\abs{\expect[Y_s^2]^{-1} \expect[Y_s Y_{s,r}]} \leq \frac{1}{4}$.
    Therefore,
    \begin{align*}
      \P(\abs{Y_s - a} \leq 2 \abs{Y_{s,r}})
      &\leq \P( \chi_s \abs{Y - \chi_s^{-1} a} \lesssim_H \abs{s-v}^{H-1} \abs{t-s} \abs{X}) \\
      &\lesssim \chi_s^{-1} e^{-\frac{c a^2}{2 \chi_s^2}} \abs{s-v}^{H-1} \abs{t-s}.
    \end{align*}
    This gives an estimate of $R_5$. Thus, we conclude
    \begin{equation*}
      \expect[L_t - L_s \vert \cF_v]
      =  \frac{e^{-\frac{\abs{Y_s - a}^2}{2 \sigma_s^2}}}{\sqrt{2 \pi} \sigma_s} \abs{t-s} + R_6,
    \end{equation*}
    where
    \begin{equation*}
      \norm{R_6}_{L_m(\P)} \lesssim \chi_s^{-\frac{1}{m}} e^{-\frac{c a^2}{2 \chi_s^2}}
      \Big(\frac{t-s}{s-v} \Big)^{\frac{1}{m} + (1 - \frac{1}{m})H}
      \abs{t-s}^{1 - H + \frac{H}{m}}.
    \end{equation*}
    This completes the proof of the lemma.
  \end{proof}
  Now we can complete the proof of Theorem~\ref{thm:local_time_upcrossing}. The above lemmas show
  \begin{equation*}
    \norm{\fc_{H, \gamma}(L_t(a) - L_t(a)) - A_{s, t}(a)}_{L_m(\P)} \\
      \begin{cases}
        \lesssim_T \abs{t-s}^{1 - H},  & \text{for all } H \in (0, 1), \\
        \lesssim_{T} s^{-\frac{H}{m}}
        \abs{t-s}^{1 - H + \frac{H}{m}}, & \text{if } H > \frac{1}{2},
      \end{cases}
  \end{equation*}
  and if $\frac{t-s}{s-v}$ is sufficiently small, then
  \begin{multline*}
    \norm{\expect[\fc_{H, \gamma}(L_t(a) - L_s(a)) - A_{s, t}(a) \vert \cF_v]}_{L_m(\P)} \\
      \begin{cases}
        \lesssim_T  e^{- c (\expect[B(0)^2] + T^{2H})^{-1} a^2}
        \Big(\frac{t-s}{s-v} \Big)^{\min\{2H, \frac{1}{m} + H\} - \frac{H}{m}}
        \abs{t-s}^{1 - H},  & H < \frac{1}{2}, \\
        \lesssim_{T} s^{-\frac{H}{m}}
        e^{- c (\expect[B(0)^2] + T^{2H})^{-1} a^2}
        \Big(\frac{t-s}{s-v} \Big)^{\min\{1, \frac{1}{m} + H\} - \frac{H}{m}}\abs{t-s}^{1 - H + \frac{H}{m}},  & H > \frac{1}{2}.
      \end{cases}
  \end{multline*}
  Noting that the exponents satisfy the assumption of Theorem~\ref{thm:generalized_stochastic_sewing}
  or Corollary~\ref{cor:singular_generalized_stochastic_sewing_lemma},
  Remark~\ref{rem:convergence_rate} implies
  \begin{equation}\label{eq:local_time_convergence_rate}
    \norm{\fc_{H, \gamma} L_T(a) - \sum_{[s,t] \in \pi} A_{s, t}(a)}_{L_m(\P)}
    \lesssim_{T} e^{- c (\expect[B(0)^2] + T^{2H})^{-1} a^2}
     \abs{\pi}^{\epsilon}.
  \end{equation}
  Hence,
  we complete the proof of Theorem~\ref{thm:local_time_upcrossing}.
\end{proof}
\begin{corollary}\label{cor:number_level_crossing}
  We have
  \begin{equation*}
    \lim_{n \to \infty} \Big(\frac{T}{n} \Big)^{1-H} \#
    \set{k \in \{1, \ldots, n\} \given B_{\frac{(k-1)T}{n}} < a < B_{\frac{kT}{n}}}
    = \sqrt{\frac{c_H}{2 \pi}} L_T(a),
  \end{equation*}
  where the convergence is in $L_m(\P)$ with $m$ satisfying \eqref{eq:local_time_m_condition}.
\end{corollary}
\begin{proof}
  The claim is a special case of Theorem~\ref{thm:local_time_upcrossing} with $\gamma = 0$.
  When $m = 2$, it is proved in \cite[Theorem~5]{azais90}.
\end{proof}

{For applications to pathwise stochastic calculus, a representation of the local time as in (b) above is more useful. In \cite[Theorem~3.2]{cont19} a pathwise It\^o-Tanaka formula is derived under the assumption that
\begin{equation}\label{eq:tilde-L-def}
	\tilde{L}^{\pi}_T(a) \defby \sum_{[s,t] \in \pi: B_s < a < B_t} \abs{B_t - a}^{\frac{1}{H} - 1}
\end{equation}
converges weakly in $L_m(\R)$ for some $m > 1$. But as already suggested by \cite[Lemma~3.5]{cont19}, this weak convergence in $L_m(\R)$ follows from our convergence result in Theorem~\ref{thm:local_time_upcrossing}:

\begin{corollary}\label{cor:local_time_discretization}
	Let $B \in C([0,T], \R)$ and for any partition $\pi$ of $[0,T]$ let $\tilde{L}^{\pi}_T(a)$ be defined as in~\eqref{eq:tilde-L-def}.
  Assume that $m > 1$ and that 
  $(\pi_n)$ is a sequence of partitions of $[0,T]$ such that $\lim_{n \to \infty} \sup_{[s,t] \in \pi_n} |B_t - B_s| = 0$ and for a limit $L_T \in L_m(\R)$:
  \begin{equation}\label{eq:almost_sure_convergence_local_time}
    \lim_{n \to \infty} \int_{\R} \abs{\sum_{[s, t] \in \pi_n: B_s < a < B_t} \abs{B_t - B_s}^{\frac{1}{H} - 1}
    - \fc_{H, \frac{1}{H} - 1} L_T(a)}^m \dd a = 0.
  \end{equation}
  Then
  \begin{equation*}
    \lim_{n \to \infty} \tilde{L}^{\pi_n}_T(\cdot) =  H \fc_{H, \frac{1}{H} - 1} L_T(\cdot)
    \quad \text{weakly in } L_m(\R).
  \end{equation*}
\end{corollary}

\begin{remark}
  If $B$ is a sample path of the fractional Brownian motion with Hurst index $H \in (0,1)$, then by Theorem~\ref{thm:local_time_upcrossing} the convergence \eqref{eq:almost_sure_convergence_local_time} holds in probability for any sequence of partitions $(\pi_n)_{n \in \N}$,
  provided that $m$ satisfies \eqref{eq:local_time_m_condition}. Therefore, we can find a subsequence so that the convergence along the subsequence holds almost surely. In fact, by~\eqref{eq:local_time_convergence_rate} we even control the convergence rate in terms of the mesh size of the partition, and this easily gives us specific sequences of partitions along which the convergence holds almost surely and not only in probability. 
  For example, \tymchange{if $\pi_n$ is the $n$th dyadic partition of $[0, T]$, the estimate \eqref{eq:local_time_convergence_rate} gives 
  \begin{align*}
    \norm{\fc_{H, \gamma} L_T(a) - \sum_{[s,t] \in \pi_n} A_{s, t}(a)}_{L_m(\P)}
    \lesssim_{T} e^{- c (\expect[B(0)^2] + T^{2H})^{-1} a^2}
     2^{-n \epsilon}.
  \end{align*}
  Since the right-hand side is summable with respect to $n$, the Borel--Cantelli lemma implies the almost sure convergence.
  } Along any such sequence of partitions, we therefore obtain the a.s. weak convergence of $\tilde L^{\pi_n}_T$ in $L_m(\R)$.
\end{remark}

\begin{proof}
  Set
  \begin{equation*}
    \tilde{A}_{s, t}(a) \defby (\abs{B_t - a}^{\frac{1}{H} - 1} - H \abs{B_t - B_s}^{\frac{1}{H} - 1}) \indic_{\{B_s < a < B_t \}}.
  \end{equation*}
  It suffices to show that $\sum_{[s,t] \in \pi_n} \tilde{A}_{s, t}(\cdot)$ converge weakly to $0$ in $L_m(\R)$.
  Since
  \[
  	|\tilde{A}_{s, t}(a)| \le \min\{H, 1-H\} \abs{B_t - a}^{\frac{1}{H} - 1}\indic_{\{B_s < a < B_t \}}
  \]
  and since $\sum_{[s,t] \in \pi_n} \abs{B_t - a}^{\frac{1}{H} - 1}\indic_{\{B_s < a < B_t \}}$ is bounded in $L_m(\P)$ by assumption, it suffices to show that
  \begin{equation*}
    \lim_{n \to \infty} \int_{\R} \Big( \sum_{[s,t] \in \pi_n} \tilde{A}_{s, t}(a) \Big) g(a) \dd a = 0
  \end{equation*}
  for every compactly supported continuous function $g$.
  Since for $B_s < B_t$ we have
  \begin{equation*}
    B_{s,t}^{-1}\int_{B_s}^{B_t} \abs{B_t - a}^{\frac{1}{H} - 1} \dd a = H \abs{B_{s, t}}^{\frac{1}{H} - 1},
  \end{equation*}
  we obtain
  \begin{equation*}
    \int_{B_s}^{B_t} \tilde{A}_{s, t}(a) g(a) \dd a
    = \int_{B_s}^{B_t} \abs{B_t - a}^{\frac{1}{H} - 1} (g(a) - B_{s,t}^{-1} \int_{B_s}^{B_t} g(x) \dd x) \dd a.
  \end{equation*}
  Therefore,
  \begin{align*}
     \int_{\R} \Big( \sum_{[s,t] \in \pi_n} \tilde{A}_{s, t}(a) \Big) g(a) \dd a
    &= \sum_{[s, t] \in \pi_n} \int_{B_s}^{B_t} \abs{B_t - a}^{\frac{1}{H} - 1} (g(a) - B_{s,t}^{-1} \int_{B_s}^{B_t} g(x) \dd x) \dd a \\
    &\leq \int_{\R} \tilde{L}^{\pi_n}_T(a) \dd a
    \times \sup_{\abs{x-y} \leq \sup_{[s,t] \in \pi} \abs{B_t - B_s}} \abs{g(x) - g(y)}
  \end{align*}
  which converges to $0$.
\end{proof}
}

\begin{remark}
  As noted in \cite{mukeru17}, we can use Theorem~\ref{thm:local_time_upcrossing} to simulate the local time of a fractional Brownian motion,
  see Figure \ref{fig:fbm_local_time_0.1} ($H = 0.1$) and Figure \ref{fig:fbm_local_time_0.7} ($H=0.6$)\footnote{Fractional Brownian motions are
  simulated by the Python package \texttt{fbm}: https://pypi.org/project/fbm/}.
\end{remark}
\begin{figure}[p]
\centering
  \includegraphics[width=\textwidth]{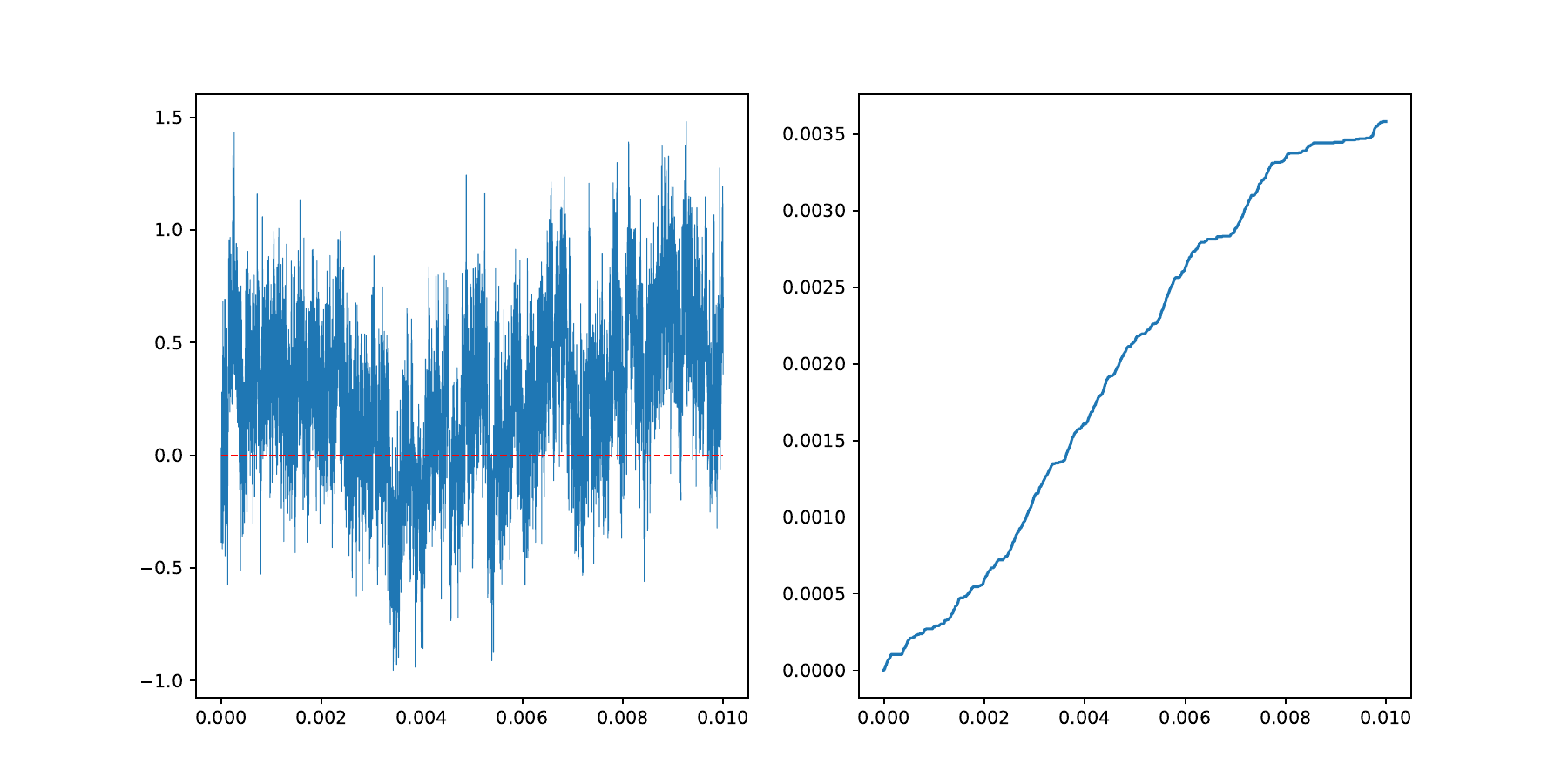}
  \caption{Left: a fractional Brownian motion with $H=0.1$, right: its local time at $0$  \label{fig:fbm_local_time_0.1}}
  \includegraphics[width=\textwidth]{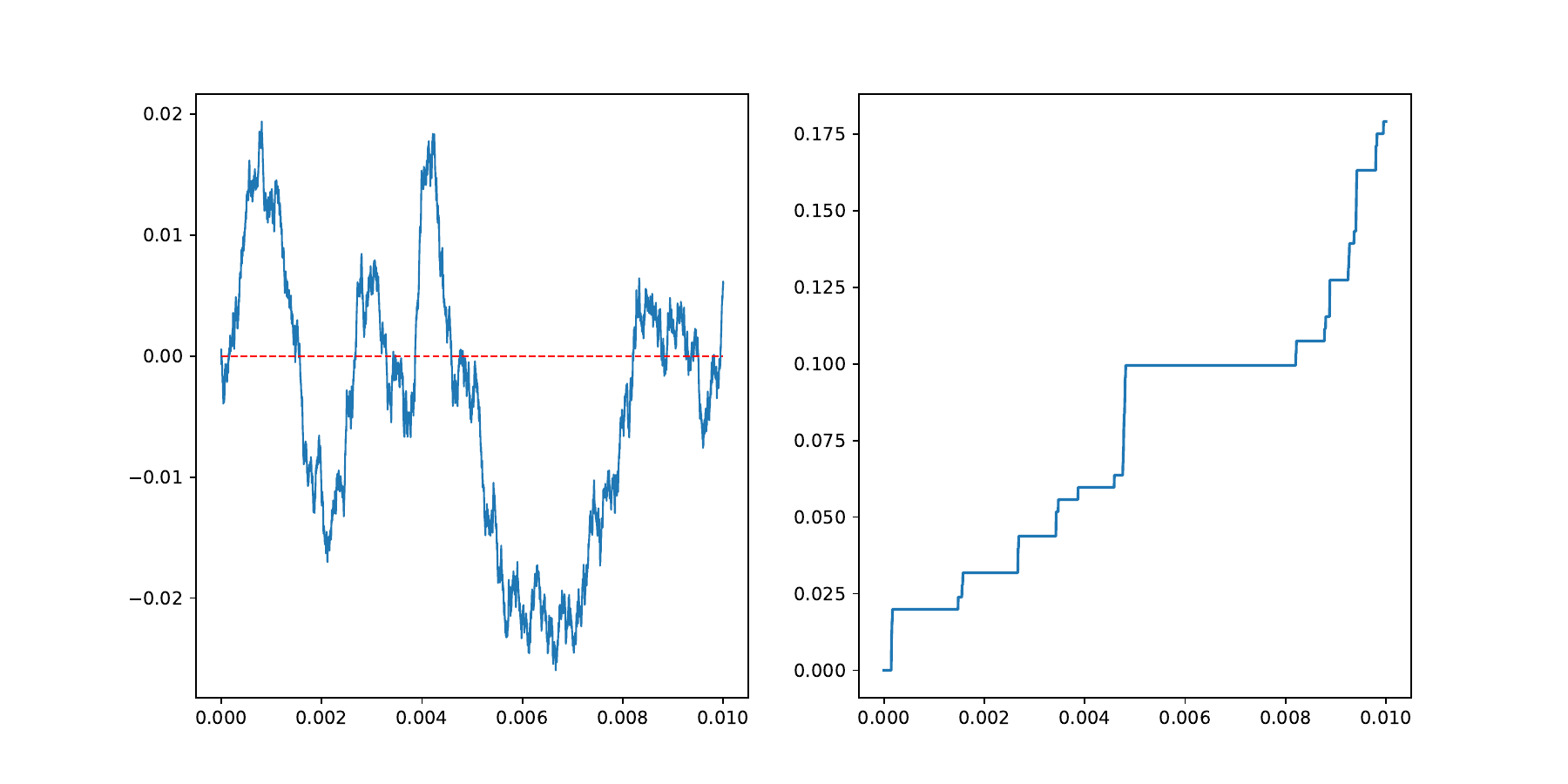}
  \caption{Left: a fractional Brownian motion with $H=0.6$, right: its local time at $0$ \label{fig:fbm_local_time_0.7}}
\end{figure}

\section{Regularization by noise for diffusion coefficients}\label{sec:regularization_by_noise}
Let $y \in C^{\alpha} ([0, T] , \mathbb{R}^d)$ with $\alpha \in (
\frac{1}{2}, 1 )$. We consider a Young differential equation
\begin{equation}
  \dd x_t = b (x_t) \dd t + \sigma (x_t) \dd y_t, \quad x_0 = x.
  \label{eq:young}
\end{equation}
We suppose that the drift coefficient $b$ belongs to $C^1_b (\mathbb{R}^d , \mathbb{R}^d)$, where $C^1_b(\mathbb{R}^d , \mathbb{R}^d)$ is the space of
continuously differentiable bounded functions between $\R^d$ with bounded derivatives. If the diffusion coefficient $\sigma$ belongs to  $C^1_b
(\mathbb{R}^d , \cM_d)$, where $\cM_d$ is the space of $d\times d$ matrices, then we can prove the existence of
a solution to {\eqref{eq:young}}. However, to prove the uniqueness of
solutions, the coefficient $\sigma$ needs to be more regular.
The following result is well-known (e.g. \cite{lyons94}), but we give a proof for the sake of later discussion.

\begin{proposition}\label{prop:uniqueness_young}
  Let $b \in C^1_b(\R^d, \R^d)$ and  $\sigma
  \in C^{1 + \delta}(\R^d, \cM_d)$ with $\delta > \frac{1-\alpha}{\alpha}$. Then the Young differential equation
  {\eqref{eq:young}} has a unique solution.
\end{proposition}

\begin{proof}
  The argument is very similar to that of \cite[Theorem~6.2]{khoa20}.
  Let $x^{(i)}$ ($i = 1, 2$) be two solutions to {\eqref{eq:young}}. Then,
  \begin{eqnarray*}
    x^{(1)}_t - x^{(2)}_t & = & \int_0^t \{ b (x^{(1)}_s) - b (x^{(2)}_s) \}
    \dd s + \int_0^t \{ \sigma (x^{(1)}_s) - \sigma (x^{(2)}_s) \} \dd
    y_s.\\
    & = & \int_0^t \{ x^{(1)}_s - x^{(2)}_s \} \dd v_s,
  \end{eqnarray*}
  where
  \[ v_t \defby \int_0^t \int_0^1 \nabla b (\theta x^{(1)}_s + (1 - \theta)
     x^{(2)}_s) \dd \theta \dd s + \int_0^t \int_0^1 \nabla \sigma
     (\theta x^{(1)}_s + (1 - \theta) x^{(2)}_s) \dd \theta \dd y_s . \]
  Note that the second term is well-defined as a Young integral since
  \[ s \mapsto \nabla \sigma (\theta x^{(1)}_s + (1 - \theta) x^{(2)}_s) \]
  is $\delta \alpha$-H{\"o}lder continuous and $\delta \alpha + \alpha > 1$ by
  our assumption of $\delta$. Therefore, $x^{(1)} - x^{(2)}$ is a solution
  of the Young differential equation
  \[ \dd z_t = z_t \dd v_t, \quad z_0 = 0. \]
  The uniqueness of this linear Young differential equation is known. Hence $x^{(1)}
  - x^{(2)} = 0.$
\end{proof}

Proposition~\ref{prop:uniqueness_young} is sharp in the sense that for any $\alpha \in (1,2)$ and any $\delta \in (0, \frac{1-\alpha}{\alpha})$, we can find
$\sigma \in C^{\gamma}(\R^2,\cM_2)$ and $y \in C^{\alpha}([0,T],\R^2)$ such that the Young differential equation
\begin{equation*}
  \dd x_t = \sigma(x_t) \dd y_t
\end{equation*}
has more than one solution, see Davie~\cite[Section~5]{davie_2010}.
However, if the driver $y$ is random,
we could hope to obtain the uniqueness of solutions in a probabilistic sense even when the regularity of $\sigma$ does not satisfy the assumption of
Proposition~\ref{prop:uniqueness_young}.
For instance, if the driver $y$ is a Brownian motion and the integral is understood
in It\^o's sense, the condition $\sigma \in C^1_b$ is sufficient to prove pathwise uniqueness.

The goal of this section is to prove the following.
\begin{theorem}\label{thm:regularization_by_noise_for_multiplicative_noise}
  Suppose that $B$ is an $(\cF_t)$-fractional Brownian motion with Hurst parameter $H \in (\frac{1}{2}, 1)$
  in the sense of Definition~\ref{def:fbm}.
  Let $b \in C^1_b
  (\mathbb{R}^d , \mathbb{R}^d)$ and $\sigma \in C^{1}_b (\mathbb{R}^d , \mathcal{M}_d)$.
  Assume one of the following.
  \begin{enumerate}
    \item We have $b \equiv 0$ and $\sigma \in C^{1+\delta} (\mathbb{R}^d , \mathcal{M}_d)$ with
      \begin{equation}\label{eq:delta_H_strong}
       \delta > \frac{(1 - H) (2 - H)}{H (3 - H)}.
      \end{equation}
    \item For all $x \in \R^d$ the matrix $\sigma(x)$ is symmetric and satisfies
    \begin{equation*}
      y \cdot \sigma(x) y > 0, \quad \forall y \in \R^d,
    \end{equation*}
    and $\sigma \in C^{1+\delta}(\mathbb{R}^d , \mathcal{M}_d)$ with $\delta$ satisfying \eqref{eq:delta_H_strong}.
    \item We have $\sigma \in C^{1+\delta}(\mathbb{R}^d ; \mathcal{M}_d)$ with
      \begin{equation}\label{eq:delta_H_weak}
       \delta > \frac{(1-H)(2-H)}{1+H-H^2}.
      \end{equation}
  \end{enumerate}
  Then, for every $x \in \R^d$, there exists
  a unique, up to modifications, process $(X_t)_{t \in [0, \infty)}$ with the following properties.
  \begin{itemize}
    \item The process $(X_t)$ is $(\cF_t)$-adapted and is $\alpha$-Hölder continuous for every $\alpha < H$.
    \item The process $(X_t)$ solves the Young differential equation
    \begin{equation}\label{eq:young_fbm}
   \dd X_t = b (X_t) \dd t + \sigma (X_t) \dd B_t, \quad X_0 = x .
    \end{equation}
  \end{itemize}
  Furthermore, in that case the process $(X_t)_{t \in [0, \infty)}$ is a strong solution, i.e. it is adapted to the natural filtration generated by \tymchange{the Brownian motion} $W$ \tymchange{appearing in the Mandelbrot--van Ness representation \eqref{eq:fbm_def}}.
\end{theorem}
\begin{figure}
    \centering
    \includegraphics[scale=0.5]{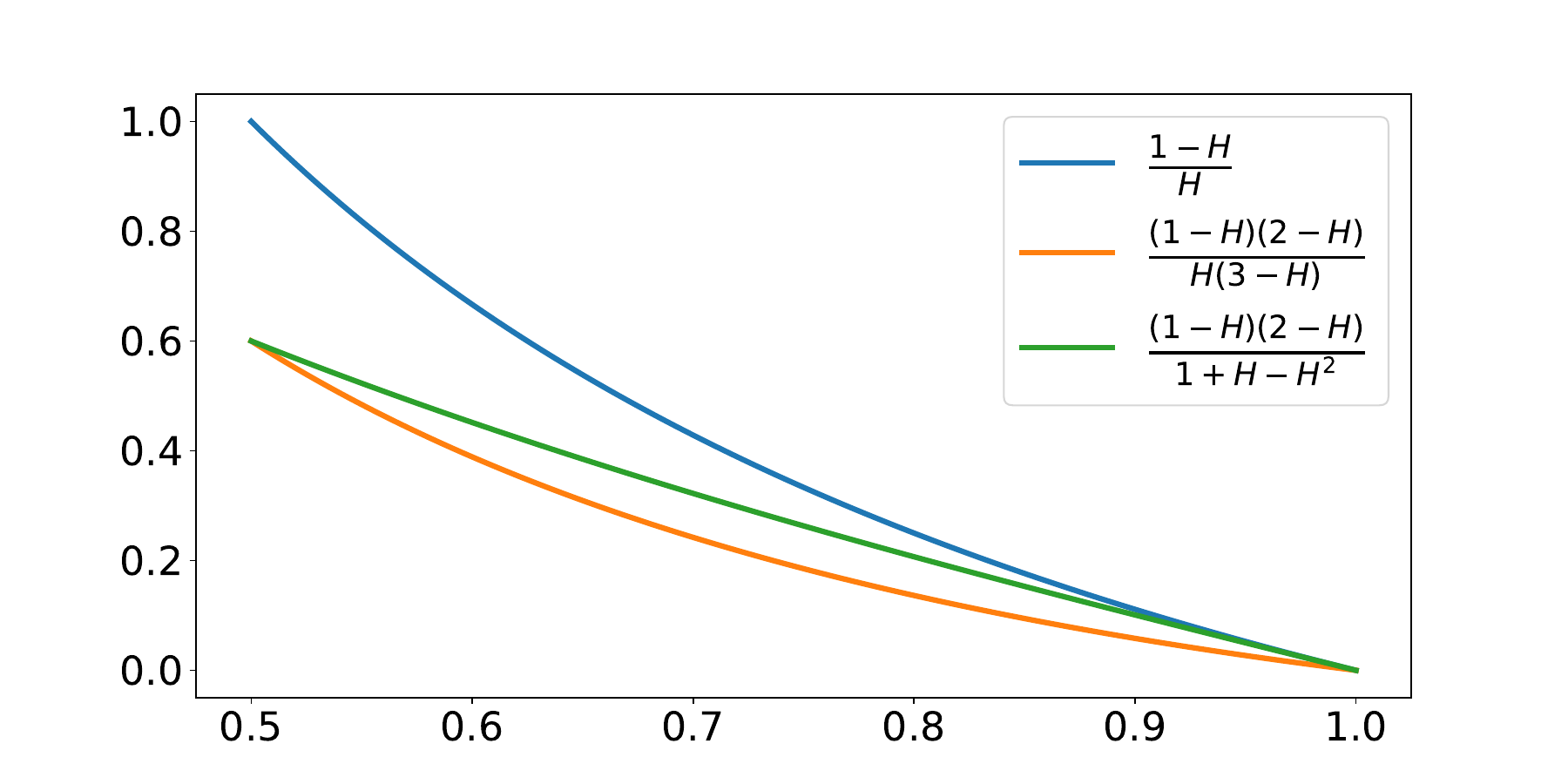}
    \caption{The functions of $H$ appearing in Theorem~\ref{thm:regularization_by_noise_for_multiplicative_noise}.}
\end{figure}
\begin{remark}
  In (b), we assume the positive-definiteness of $\sigma$ to ensure that for every $x, y \in \R^d$ and $\theta \in [0,1]$ the
  matrix
  \begin{equation*}
    \theta \sigma(x) + (1 - \theta) \sigma(y)
  \end{equation*}
  is invertible.
\end{remark}
\begin{remark}
Since the seminal work \cite{catellier16} of Catellier and Gubinelli, many works have appeared to establish
weak or strong existence or uniqueness to the SDE
\begin{equation*}
  \dd X_t = b(X_t) \dd t + \dd B_t
\end{equation*}
for an irregular drift $b$ and a fractional Brownian motion $B$.
In contrast, there are much fewer works that attempt to optimize the regularity of the diffusion coefficient $\sigma$.
The work \cite{hinz2020variability} by Hinz, Tölle and Viitasaari, where $b \equiv 0$, considers certain
existence and uniqueness for $\sigma$ that is merely of bounded variation,
at the cost of additional restrictive assumptions
(variability and \cite[Assumption~3.15]{hinz2020variability}).
It seems that Theorem~\ref{thm:regularization_by_noise_for_multiplicative_noise} is the first result
to improve the regularity of $\sigma$ without any additional assumption (except $\sigma$ being invertible
for the case (b)).
However, we believe that our assumption of $\delta$ is not optimal, see Remark~\ref{rem:optimal_regularity_of_sigma}.
\end{remark}

The proof of Proposition~\ref{prop:uniqueness_young} suggests that the pathwise uniqueness holds if,
for any two $(\cF_t)$-adapted solutions $X^{(1)}$ and $X^{(2)}$ to \eqref{eq:young_fbm}, we can construct the integral
\begin{equation}\label{eq:integral_X}
\int_0^t \int_0^1 \nabla \sigma
     (\theta X^{(1)}_s + (1 - \theta) X^{(2)}_s) \dd \theta \dd B_s.
\end{equation}
If $\theta X^{(1)}_s + (1 - \theta) X^{(2)}_s$ is replaced by $B_s$,
then the integral is constructed in Proposition~\ref{prop:integral_H_greater_than_one_half}.
The difficulty here is that $X^{(i)}$ is not Gaussian and the Wiener chaos decomposition crucially used in
the proof of Proposition~\ref{prop:integral_H_greater_than_one_half} cannot be applied.
Yet, the process $X^{(i)}$ is \emph{locally controlled by the Gaussian process $B$} (whose precise meaning will be clarified later), and by taking advantage of this fact, we can
still make sense of the integral \eqref{eq:integral_X}.

As a technical ingredient, we need a variant of Theorem~\ref{thm:generalized_stochastic_sewing}.
\begin{lemma}
  \label{lem:stochastic-sewing-tradeoff}Let $(A_{s, t})_{0 \leq s < t \leq T}$
  be a family of two-parameter random variables and let $(\mathcal{F}_t)_{t
  \in [0, T]}$ be a filtration such that $A_{s, t}$ is
  $\mathcal{F}_t$-measurable for every $0 \leq s \leq t \leq T$.
  Suppose that for some $m \geq 2$, $\Gamma_1, \Gamma_{2,} \Gamma_3 \in [0,
  \infty)$ and $\alpha, \gamma, \beta_1, \beta_2, \beta_3 \in [0, \infty)$,  we
  have for every $0 \leq v < s < u < t \leq T$
  \begin{align*}
    \| \mathbb{E} [\delta A_{s, u, t} | \mathcal{F}_v] \|_{L_m (\mathbb{P})}
     &\leq \Gamma_1 | s - v |^{- \alpha} | t - s |^{\beta_1} + \Gamma_2 | t - v
     |^{\gamma} | t - s |^{\beta_2}, \quad \text{if } t-s \leq s - v, \\
     \| \delta_{} A_{s, u, t} \|_{L_m (\mathbb{P})} &\leq \Gamma_3 | t - s
     |^{\beta_3}.
  \end{align*}
   Suppose that
  \begin{equation}\label{eq:stochastic_sewing_tradeoff_parameters}
   \min \{ \beta_1, 2 \beta_3 \} > 1, \quad \gamma + \beta_2 > 1, \quad 1 +
     \alpha - \beta_1 < \alpha \min \Big\{ \frac{\gamma + \beta_2 -
     1}{\gamma}, 2 \beta_3 - 1 \Big\} .
  \end{equation}
  Finally, suppose that there exists a stochastic process $(\cA_t)_{t \in [0,T]}$ such that
  \begin{equation*}
    \cA_t = \lim_{\abs{\pi} \to 0; \text{$\pi$ is a partition of $[0,t]$}} \sum_{[u,v] \in \pi} A_{u, v},
  \end{equation*}
  where the convergence is in $L_m(\P)$.
  Then, we have
  \[ \| \mathcal{A}_t - \mathcal{A}_s - A_{s, t} \|_{L_m (\mathbb{P})}
     \lesssim_{\alpha, \gamma, \beta_1, \beta_2, \beta_3} \Gamma_1 | t - s |^{\beta_1 - \alpha} + \Gamma_2 | t - s
     |^{\gamma + \beta_2} + \kappa_{m, d} \Gamma_3 | t - s |^{\beta_3}. \]
\end{lemma}
\begin{remark}
  It should be possible to formulate Lemma~\ref{lem:stochastic-sewing-tradeoff} at the generality of Theorem~\ref{thm:generalized_stochastic_sewing}.
  However, such generality is irrelevant to prove Theorem~\ref{thm:regularization_by_noise_for_multiplicative_noise},
  and we do not pursue the generality to simplify the presentation.
\end{remark}

\begin{proof}
  Here we consider dyadic partitions. Fix $s < t$ and set
  \[ A^n_{s, t} \defby \sum_{k = 0}^{2^n - 1} A_{s + \frac{k}{2^n} (t - s),
     s + \frac{k + 1}{2^n} (t - s)} . \]
  Since $\cA_{s,t} = \lim_{n \to \infty} A^n_{s, t}$, it suffices to show
  \begin{equation*}
    \norm{A^{n}_{s,t} - A^{n+1}_{s, t}}_{L_m(\P)}
    \lesssim 2^{-n \delta} (\Gamma_1 | t - s |^{\beta_1 - \alpha} + \Gamma_2 | t - s
     |^{\gamma + \beta_2} + \kappa_{m, d} \Gamma_3 | t - s |^{\beta_3} )
  \end{equation*}
  for some $\delta > 0$ and all sufficiently large $n$.
{As in the proof of Theorem~\ref{thm:generalized_stochastic_sewing} we decompose
	\begin{eqnarray*}
		A^n_{s, t} - A^{n + 1}_{s, t} & = & \sum_{l = 0}^L \sum_{j : l + j L \leq 2^n - 1}^{} \{ \delta A^n_{l
    + j L} -\mathbb{E} [\delta A^n_{l + j L} | \mathcal{F}_{l + (j - 1) L +
    1}^n] \}\\
    &  & + \sum_{l = 0}^L \sum_{j : l + j L \leq 2^n - 1} \mathbb{E} [\delta
    A^n_{l + j L} | \mathcal{F}_{l + (j - 1) L + 1}^n] .
	\end{eqnarray*}
}
  By the BDG inequality,
  \begin{multline*}
    \Big\| \sum_{j : l + j L \leq 2^n - 1}^{} \{ \delta A^n_{l + j L}
       -\mathbb{E} [\delta A^n_{l + j L} | \mathcal{F}_{l + (j - 1) L + 1}^n]
       \} \Big\|_{L_m (\mathbb{P})} \\
       \lesssim \kappa_{m, d} \Big( \sum_{j : l + j L \leq 2^n - 1} \| \delta
       A^n_{l + j L} \|^2_{L_m (\mathbb{P})} \Big)^{\frac{1}{2}}
       \leq \kappa_{m, d} \Gamma_3 \Big( \sum_{j : l + j L \leq 2^n - 1}
       (2^{- n} | t - s |)^{2 \beta_3} \Big)^{\frac{1}{2}}
  \end{multline*}
  Thus,
  \[ \Big\| \sum_{l = 0}^L \sum_{j : l + j L \leq 2^n - 1}^{} \{ \delta
     A^n_{l + j L} -\mathbb{E} [\delta A^n_{l + j L} | \mathcal{F}_{l + (j -
     1) L + 1}^n] \} \Big\|_{L_m(\P)} \lesssim \kappa_{m, d} \Gamma_3 L^{\frac{1}{2}}
     2^{- n ( \beta_3 - \frac{1}{2} )} | t - s |^{\beta_3} . \]
  Furthermore,
  \begin{align*}
    \MoveEqLeft[9]
    \Big\| \sum_{l = 0}^L \sum_{j : l + j L \leq 2^n - 1} \mathbb{E} [\delta
    A^n_{l + j L} | \mathcal{F}_{l + (j - 1) L + 1}^n] \Big\|_{L_m
    (\mathbb{P})} \\
    \leq & \sum_{l = 0}^L \sum_{j : l + j L \leq 2^n - 1} \|
    \mathbb{E} [\delta A^n_{l + j L} | \mathcal{F}_{l + (j - 1) L + 1}^n]
    \|_{L_m (\mathbb{P})}\\
     \leq & \Gamma_1 \sum_{l = 0}^L \sum_{j : l + j L \leq 2^n - 1} \Big(
    \frac{L - 1}{2^n} | t - s | \Big)^{- \alpha} (2^{- n} | t - s
    |)^{\beta_1}\\
      & + \Gamma_2 \sum_{l = 0}^L \sum_{j : l + j L \leq 2^n - 1} \Big(
    \frac{L }{2^n} | t - s | \Big)^{\gamma} (2^{- n} | t - s
    |)^{\beta_2}\\
     \lesssim & \Gamma_1 2^{- n (\beta_1 - \alpha - 1)} L^{- \alpha} | t - s
    |^{\beta_1 - \alpha} + \Gamma_2 2^{- n (\gamma + \beta_2 - 1)} L^{\gamma}
    | t - s |^{\gamma + \beta_2} .
  \end{align*}
  Therefore,
  \begin{multline*}
       \| A^n_{s, t} - A^{n + 1}_{s, t} \|_{L_m (\mathbb{P})}
       \lesssim \Gamma_1 2^{- n (\beta_1 - \alpha - 1)} L^{- \alpha} | t - s
       |^{\beta_1 - \alpha} \\ + \Gamma_2 2^{- n (\gamma + \beta_2 - 1)}
       L^{\gamma} | t - s |^{\gamma + \beta_2} + \kappa_{m, d} \Gamma_3
       L^{\frac{1}{2}} 2^{- n ( \beta_3 - \frac{1}{2} )} | t - s
       |^{\beta_3} .
  \end{multline*}
  We choose $L = \lfloor 2^{n \varepsilon} \rfloor$ with $\varepsilon \in (0, 1)$ so that
  \[ \min \Big\{ \beta_1 - \alpha - 1 + \alpha \varepsilon, \gamma + \beta_2
     - 1 - \gamma \varepsilon, \beta_3 - \frac{1 + \varepsilon}{2} \Big\} >
     0. \]
  Namely,
  \[ \frac{1 + \alpha - \beta_1}{\alpha} < \varepsilon < \min \Big\{
     \frac{\gamma + \beta_2 - 1}{\gamma}, 2 \beta_3 - 1 \Big\} . \]
  Such an $\epsilon$ exists exactly under our assumption \eqref{eq:stochastic_sewing_tradeoff_parameters}.
\end{proof}
We mentioned that a solution to \eqref{eq:young_fbm} is controlled by $B$.
Here comes a more precise statement.
We fix $\alpha \in ( \frac{1}{2}, H )$
and let $X$ be a solution to \eqref{eq:young_fbm}. We have the estimates
\[ \Big| \int_s^t b (X_r) \dd r - b (X_s) (t - s) \Big| \leq \| b
   \|_{C^1_b} \| X \|_{C^{\alpha} }^{} (t - s)^{1 + \alpha}, \]
\begin{equation}\label{eq:sigma_young_approximation}
 \Big| \int_s^t \sigma (X_r) \dd B_r - \sigma (X_s) B_{s, t} \Big|
   \lesssim_{\alpha} \| \sigma \|_{C^1_b} \| X \|_{C^{\alpha} } \| B
   \|_{C^{\alpha} } | t - s |^{2 \alpha} .
\end{equation}
Furthermore, the a priori estimate of the Young differential equation (\cite[Proposition~8.1]{friz2020}) gives
\begin{equation}\label{eq:X_a_priori}
  \norm{X}_{C^{\alpha}([0, T])} \lesssim_{T, \alpha} \abs{x} +
  (\norm{b}_{C^1_b} + \norm{\sigma}_{C^1_b} \norm{B}_{C^{\alpha}})
  (1 + \norm{b}_{C^1_b} + \norm{\sigma}_{C^1_b} \norm{B}_{C^{\alpha}}).
\end{equation}
Therefore, we have
\begin{equation*}
  X_t = X_s + \sigma(X_s) B_{s,t} + b(X_s) (t-s) + R_{s, t},
\end{equation*}
where
\begin{equation}\label{eq:a_priori_R}
  \abs{R_{s,t}} \lesssim_{T, \alpha}
  (\abs{x} + \norm{b}_{C^1_b} + \norm{\sigma}_{C^1_b} \norm{B}_{C^{\alpha}})
  (1 + \norm{b}_{C^1_b} + \norm{\sigma}_{C^1_b} \norm{B}_{C^{\alpha}})^2 \abs{t-s}^{2 \alpha}.
\end{equation}
This motivates the following definition.
\tymchange{Recall that $B$ is an $(\mathcal{F}_t)$-fractional Brownian motion in the sense of Definition~\ref{def:fbm}.}
\begin{definition}
  Let $Z$ be a random path in $C^{\alpha} ([0, T] , \mathbb{R}^d)$. For $\beta \in (\alpha, \infty)$, we write
  $Z \in \mathcal{D} (\alpha, \beta)$ if for every $s < t$, we have
  \begin{equation*}
   Z_t = z^{(1)}_{} (s) +  z^{(2)} (s) B_t + z^{
     (3)} (s) (t - s) + R_{s, t},
  \end{equation*}
  where
  \begin{itemize}
    \item the random variables $z^{(1)}(s), z^{(3)}(s) \in \R^d$ and $z^{(2)}(s) \in \cM_d$
    are $\mathcal{F}_s$-measurable and
    \item there exists a (random) constant $C \in [0, \infty)$ such that for all $s < t$
    \[ | R_{s, t} | \leq C  | t - s
       |^{\beta} . \]
  \end{itemize}
  We set
  \begin{equation*}
    \norm{Z}_{\cD(\alpha,\beta)}
    \defby \norm{Z}_{C^{\alpha}([0, T])} +  \norm{z^{(2)}}_{L_{\infty}([0, T])} + \norm{z^{(3)}}_{L_{\infty}([0, T])}
    + \sup_{0\leq s<t \leq T} \frac{\abs{R_{s,t}}}{(t-s)^{\beta}}.
  \end{equation*}
  Furthermore, we set
  \begin{equation*}
    \cD_0(\alpha, \beta) \defby \set{Z \in \cD(\alpha, \beta) \given z^{(3)} \equiv 0},
  \end{equation*}
  \begin{equation*}
    \cD_1(\alpha, \beta) \defby \set{Z \in \cD(\alpha, \beta) \given z^{(2)}(s) \text{ is invertible for all }s \in [0, T) }
  \end{equation*}
  and $\norm{Z}_{\cD_1(\alpha,\beta)} \defby \norm{Z}_{\cD_1(\alpha,\beta)} + \sup_{s \in [0, T]} \abs{(z^{(2)}(s))^{-1} z^{(3)}(s)}$.
\end{definition}
\begin{proposition}\label{prop:stochastic_integral_controlled_path}
  Let $f \in C^1_b (\mathbb{R}^d ; \mathbb{R}^d)$ and $Z \in
  \mathcal{D}_1(\alpha, \beta)$. If
  \begin{equation}\label{eq:delta_greater_than_three_guys}
    \max \Big\{ \frac{1 - H}{\beta}, \frac{2 - 3 H + H^2}{\beta +
     H - H^2}, \frac{3 - \sqrt{3}}{2 H} - 1 \Big\} < \delta < 1,
  \end{equation}
  and if $\alpha$ is sufficiently close to $H$, then for every $m \in [2, \infty)$,
  \begin{multline*}
    \Big\| \int_s^t f (Z_r) \dd B_r - \frac{f (Z_s) + f (Z_t)}{2} \cdot
     B_{s, t} \Big\|_{L_m (\mathbb{P})} \\ \lesssim_{T, \alpha, \delta, m}
     \norm{f}_{C^{\delta}} \norm{\norm{Z}_{\cD_1(\alpha,\beta)}^{\delta}}_{L_{3m}(\P)}
     (1 + \norm{\norm{Z}_{\cD_1(\alpha,\beta)}}_{L_{3m}(\P)})
     \abs{t-s}^{H + \alpha \delta}.
  \end{multline*}
  If $Z \in \cD_0(\alpha, \beta)$, a similar estimate holds with $\norm{Z}_{\cD_1(\alpha,\beta)}$ replaced by $\norm{Z}_{\cD(\alpha,\beta)}$.
\end{proposition}
\begin{proof}
  Our tool is Lemma~\ref{lem:stochastic-sewing-tradeoff}. Since arguments are similar, we only prove the claim for $Z \in \cD_1(\alpha, \beta)$.
  We set
  \[ A_{s, t} \defby \frac{f (Z_s) + f (Z_t)}{2} \cdot B_{s, t} . \]
  We have
  \[ \delta A_{s, u, t} = - \frac{f (Z_u) - f (Z_s)}{2} B_{u, t} + \frac{f
     (Z_t) - f (Z_u)}{2} B_{s, u} .
  \]
  Hence
  \begin{equation}\label{eq:stochastic_integral_controlled_path_pr1} \| \delta A_{s, u, t} \|_{L_m (\mathbb{P})} \lesssim \| f
     \|_{C^{\delta}} \norm{\norm{Z}_{C^{\alpha}}^{\delta}}_{L_{2m}(\P)}  | t -
     s |^{H + \alpha \delta}.
  \end{equation}
  Let $v < s$ with $t-s \leq s- v$. As $Z \in \mathcal{D}_1 (\alpha, \beta)$, we write
  \[ Z_r = z_{}^{(1)} (v) + z^{(2)}_{} (v) B_r + z^{(3)} (v) (r - v) + R_{v,
     r}, \quad r \in [s, t] . \]
  Then, if we write $\hat{Z}_r \defby z_{}^{(1)} (v) + z^{(2)}_{} (v) B_r +
  z^{(3)} (v) (r - v)$,
  \[ | f (Z_r) - f (\hat{Z}_r) | \leq \| f \|_{C^{\delta}} | R_{v, r}
     |^{\delta} \leq \| f \|_{C^{\delta}} \norm{Z}_{\cD(\alpha,\beta)}^{\delta}  | v - r |^{\delta \beta} . \]
  Hence, if we write $\hat{A}_{s, t} \defby \frac{f (\hat{Z}_s) + f
  (\hat{Z}_t)}{2} \cdot B_{s, t}$, then
  \begin{equation}\label{eq:stochastic_integral_controlled_path_pr2}
	\| \delta A_{s, u, t} - \delta \hat{A}_{s, u, t} \|_{L_m (\mathbb{P})}
     \lesssim_{\alpha, \delta, T}
     \| f \|_{C^{\delta}} \norm{\norm{Z}_{\cD(\alpha,\beta)}^{\delta}}_{L_{2m}(\P)}  | t - v |^{\delta \beta}
      | t - s |^H.
  \end{equation}
  Next, we will estimate $\norm{\mathbb{E} [\delta \hat{A}_{s, u, t} \vert \cF_v]}_{L_m(\P)}$.
  The rest of the calculation resembles the proof of Proposition~\ref{prop:stratonovich_integral}.
  We write $Y_r \defby
  \mathbb{E} [B_r | \mathcal{F}_v]$ and $\tilde{B}_r \defby B_r - Y_r$ as before. We
  can decompose
  \[ \hat{Z}_r = \hat{Y}_r + z^{(2)} (v) \tilde{B}_r, \quad \text{with } \hat{Y}_r \defby
     z_{}^{(1)} (v) + z^{(2)}_{} (v) Y_r + z^{(3)} (v) (r - v), \]
  where $\hat{Y}_r$ is $\mathcal{F}_v$-measurable and $\tilde{B}_r$ is
  independent of $\mathcal{F}_v$. We set
  \[ \hat{a}^{}_0 (r) \defby \mathbb{E} [f (\hat{Y}_r + z^{(2)}_{} (v)
     \tilde{B}_r) | \mathcal{F}_v] =\mathbb{E} [f (z^{(2)} (v) \{ (z^{(2)}
     (v))^{- 1} \hat{Y}^{}_r + \tilde{B}_r \}) | \mathcal{F}_v], \]
  \[ \hat{a}_i^j (r) \defby 2 H (s - v)^{- 2 H} \mathbb{E} [f^j (\hat{Y}_r +
     z^{(2)}_{} (v) \tilde{B}_r) \tilde{B}_r^i | \mathcal{F}_v] . \]
  Then, as in the proof of Proposition~\ref{prop:stratonovich_integral}, we have the decomposition
  \[ \mathbb{E} [\delta \hat{A}_{s, u, t} | \mathcal{F}_v] = \hat{D}^0_{s, u,
     t} + \sum_{i = 1}^d \hat{D}_{s, u, t}^i, \]
  where, as in \eqref{eq:D_0_simple} and \eqref{eq:D_i_simple},
  \[ \hat{D}^0_{s, u, t} \defby (\hat{a}_0 (t) - \hat{a}_0 (u)) \cdot Y_{s,
     u} + (\hat{a}_0 (s) - \hat{a}_0 (u)) \cdot Y_{u, t}, \]
  \begin{align*}
    D^i_{s, u, t} \defby& \expect[(\hat{a}_i^i(s) \tilde{B}_s^i + \hat{a}_i^i(t) \tilde{B}_t^i) \tilde{B}_{s, t}^i \vert \cF_v] \\
    &-\expect[(\hat{a}_i^i(s) \tilde{B}_s^i + \hat{a}_i^i(u) \tilde{B}_u^i) \tilde{B}_{s, u}^i \vert \cF_v]
    -\expect[(\hat{a}_i^i(u) \tilde{B}_u^i + \hat{a}_i^i(t) \tilde{B}_t^i) \tilde{B}_{u, t}^i \vert \cF_v] \\
    =& (\hat{a}_i^i(t) - \hat{a}_i^i(u)) \expect[\tilde{B}_t^i \tilde{B}_{s, t}^i]
    + (\hat{a}_i^i(s) - \hat{a}_i^i(u)) \expect[\tilde{B}_s^i \tilde{B}_{s, t}^i] \\
    &-(\hat{a}_i^i(s) - \hat{a}_i^i(u)) \expect[\tilde{B}_s^i \tilde{B}_{s, u}^i]
    -(\hat{a}_i^i(t) - \hat{a}_i^i(u)) \expect[\tilde{B}_t^i \tilde{B}_{u, t}^i]
  \end{align*}

  The map $\mathbb{R}^d \ni x \mapsto f (z^{(2)} (v)^{} x) \in \mathbb{R}^d$
  belongs to $C^{\delta} (\mathbb{R}^d)$ with its norm bounded by
  \[ | z^{(2)} (v) |^{\delta} \| f \|_{C^{\delta} (\mathbb{R}^d)} . \]
  Therefore, by repeating the argument used to obtain \eqref{eq:estimate_of_abs_D_0}, we obtain
  \begin{multline*}
   | \hat{D}^0_{s, u, t} | \leq | z^{(2)} (v) |^{\delta} \| f
     \|_{C^{\delta} (\mathbb{R}^d)} \Big\{  (s -
     v)^{\delta H - 1} (t - s) (| Y_{s, u} | + | Y_{u, t} |) \\
     + (s - v)^{(\delta - 1) H} \Big( |
     (z^{(2)} (v))^{- 1} \hat{Y}_{s, u} | | Y_{u, t} | + |  z^{(2)}
     (v)^{- 1} \hat{Y}_{u, t} | | Y_{s, u} | \Big)
     \Big\} .
  \end{multline*}
  Referring to \eqref{eq:estimate_of_Y_s_u}, we have
  \[ \| (z^{(2)} (v))^{- 1} \hat{Y}_{s, u} \|_{L_m (\mathbb{P})} \lesssim (s
     - v)^{H - 1} (t - s) +  \| z^{(2)}(v)^{-1} z^{(3)} (v) \|_{L_m
     (\mathbb{P})} (t - s). \]
  Therefore,
  \begin{multline*}
   \| \hat{D}^0_{s, u, t} \|_{L_m (\mathbb{P})} \lesssim \| f
     \|_{C^{\delta}} \| | z^{(2)} (v) |^{\delta} \|_{L_{3 m} (\mathbb{P})} \\
     \times \{
     (s - v)^{(\delta + 1) H - 2} (t - s)^2 +  \| z^{(2)}(v)^{-1} z^{(3)} (v)
     \|_{L_{3 m} (\mathbb{P})} (s - v)^{\delta H - 1} (t - s)^2 \} .
  \end{multline*}

  Similarly as before, we have
  \[ | a^i_i (t) - a^i_i (s) | \lesssim | z^{(2)} (v) |^{\delta} \| f
     \|_{C^{\delta} (\mathbb{R}^d)} (s - v)^{(\delta - 2) H} (| (z^{(2)}
     (v))^{- 1} \hat{Y}_{s, u} | + (s - v)^{H - 1} (t - s)) . \]
  As $H > \frac{1}{2}$, by Lemma~\ref{lem:kernel_correlation}, we have
  \[ | \mathbb{E} [\tilde{B}_s \tilde{B}_{s, t}] | \lesssim (s - v)^{2 H - 1}
     (t - s) . \]
  Therefore,
  \begin{multline*}
   \| \hat{D}^i_{s, u, t} \|_{L_m (\mathbb{P})} \lesssim \| f
     \|_{C^{\delta}} \| | z^{(2)} (v) |^{\delta} \|_{L_{3 m} (\mathbb{P})} \\
     \times \{
     (s - v)^{(\delta + 1) H - 2} (t - s)^2 +  \| z^{(2)}(v)^{-1} z^{(3)} (v)
     \|_{L_{3 m} (\mathbb{P})} (s - v)^{\delta H - 1} (t - s)^2 \} .
  \end{multline*}

  Combining our estimates, we have
  \begin{multline}\label{eq:stochastic_integral_controlled_path_pr3}
   \| \mathbb{E} [\delta \hat{A}_{s, u, t} | \mathcal{F}_v] \|_{L_m
     (\mathbb{P})} \lesssim_T \| f \|_{C^{\delta}} \| | z^{(2)} (v) |^{\delta}
     \|_{L_{3 m} (\mathbb{P})} \\
     \times (1 +  \| z^{(2)}(v)^{-1} z^{(3)} (v) \|_{L_{3
     m} (\mathbb{P})}) (s - v)^{(\delta + 1) H - 2} (t - s)^2.
  \end{multline}
  Hence, combining \eqref{eq:stochastic_integral_controlled_path_pr1}, \eqref{eq:stochastic_integral_controlled_path_pr2} and \eqref{eq:stochastic_integral_controlled_path_pr3}
  \[ \| \delta A_{s, u, t} \|_{L_m (\mathbb{P})} \lesssim
    \| f
     \|_{C^{\delta}} \norm{\norm{Z}_{C^{\alpha}}^{\delta}}_{L_{2m}(\P)}  | t -
     s |^{H + \alpha \delta},
    \]
    and
  \begin{multline*}
   \| \mathbb{E} [\delta A_{s, u, t} | \mathcal{F}_v] \|_{L_m (\mathbb{P})}
     \lesssim_{\alpha, \delta, T} \| f \|_{C^{\delta}} \| | z^{(2)} (v) |^{\delta} \|_{L_{3 m}
     (\mathbb{P})} \\
     \times (1 +  \| z^{(2)}(v)^{-1} z^{(3)} (v) \|_{L_{3 m}
     (\mathbb{P})}) (s - v)^{(\delta + 1) H - 2} (t - s)^2 \\
     +\| f \|_{C^{\delta}} \norm{\norm{Z}_{\cD(\alpha,\beta)}^{\delta}}_{L_{2m}(\P)}  | t - v |^{\delta \beta}
      | t - s |^H.
  \end{multline*}
  To apply Lemma \ref{lem:stochastic-sewing-tradeoff}, we need
  \[ \delta \beta + H > 1, \quad \frac{1 - (1 + \delta) H}{2 - (1 + \delta)
     H} < \min \Big\{ \frac{\delta \beta + H - 1}{\delta \beta}, 2 (H + \alpha \delta) - 1 \Big\} , \]
  which, if $\alpha$ is sufficiently close to $H$, are fulfilled under \eqref{eq:delta_greater_than_three_guys}.
\end{proof}

\begin{proof}[Proof of Theorem~\ref{thm:regularization_by_noise_for_multiplicative_noise}]
  We first prove the pathwise uniqueness. We suppose that
  the assumption (b) mentioned in Theorem~\ref{thm:regularization_by_noise_for_multiplicative_noise} holds,
  and the other cases will be discussed later.
  Let $X^{(1)}$ and $X^{(2)}$ be $(\cF_t)$-adapted solutions.
  Our strategy is similar to Proposition~\ref{prop:uniqueness_young}, but here we must construct the integral \eqref{eq:integral_X} stochastically.
  For each $k \in \N$, we set
  \begin{equation*}
    \lambda_k \defby \inf_{x: \abs{x} \leq k} \inf_{y: \abs{y} =1} y \cdot \sigma(x) y,
  \end{equation*}
  we let $\sigma^{k} \in C^{1 + \delta}(\R^d; \cM_d)$ be such that
  $\sigma^k = \sigma$ in $\set{x \given \abs{x} \leq k}$ and
  \begin{equation*}
    \inf_{x \in \R^d} \inf_{y: \abs{y} = 1} y \cdot \sigma^k(x) y \geq \frac{\lambda_k}{2},
  \end{equation*}
  and we set
  \begin{equation*}
    X^{(i), k}_t \defby x + \int_0^t b(X^{(i)}_r) \dd r + \int_{0}^t \sigma^k(X^{(i)}_r) \dd B_r, \quad i=1,2.
  \end{equation*}
  If we write
  \begin{equation*}
    \Omega_k \defby \set{\omega \in \Omega \given \sup_{t \in [0, T]} \max \{\abs{X^{(1)}_t(\omega)}, \abs{X^{(2)}_t(\omega)}\} \leq k},
  \end{equation*}
  then in the event $\Omega_k$ we have $X^{(i)}_t = X^{(i), k}_t$, $t \in [0, T]$.

  Let $\{\sigma^{k, n}\}_{n=1}^{\infty}$ be a smooth approximation of $\sigma^k$.
  In general we can only guarantee the convergence in $C^{1+\delta'}(\R^d, \cM_d)$ for any $\delta'<\delta$, which is still sufficient to make the following argument work. To simplify the notation we assume that we can take $\delta'=\delta$.

  We have
  \begin{equation*}
    \int_{0}^t \sigma^k(X^{(i)}_r) \dd B_r = \lim_{n \to \infty} \int_{0}^t \sigma^{k, n}(X^{(i)}_r) \dd B_r
  \end{equation*}
  and in $\Omega_k$
  \begin{equation*}
     \int_{0}^t \{ \sigma^{k, n}(X^{(1)}_r) - \sigma^{k, n}(X^{(2)}_r)\} \dd B_r
    = \int_0^t \{X^{(1), k}_r - X^{(2), k}_r \} \dd V^{k, n}_r,
  \end{equation*}
  where
  \begin{equation*}
    V^{k, n}_t \defby \int_0^t \int_0^1 \nabla \sigma^{k, n}(\theta X^{(1), k}_r + (1-\theta) X^{(2), k}_r) \dd \theta \dd B_r.
  \end{equation*}
  For a fixed $\theta \in (0, 1)$, we set
  \begin{equation*}
    Z_t^{\theta, k} \defby \theta X_t^{(1), k} + (1 - \theta) X_t^{(2), k}.
  \end{equation*}
  By the a priori estimate \eqref{eq:a_priori_R}, for $\alpha \in (\frac{1}{2}, H)$, we have
  \begin{multline*}
    Z_t^{\theta, k} -  Z_s^{\theta, k}
    = \{\theta \sigma^k(X_s^{(1), k}) + (1 - \theta) \sigma^k(X_s^{(1), k})\} B_{s,t} \\
    +  \{\theta b(X_s^{(1), k}) + (1 - \theta) b (X_s^{(1), k})\}(t-s) + R_{s, t}
  \end{multline*}
  with
  \begin{equation*}
    \abs{R_{s, t}} \lesssim_{T, \alpha}
    (1 + \abs{x} + \norm{b}_{C^1_b} + \norm{\sigma}_{C^1_b} \norm{B}_{C^{\alpha}})^3 \abs{t-s}^{2 \alpha}.
  \end{equation*}
  Note that we have
  \begin{equation*}
    \inf_{y: \abs{y} = 1} y \cdot \{\theta \sigma^k(X_s^{(1), k}) + (1 - \theta) \sigma^k(X_s^{(1), k})\} y
    \geq \frac{\lambda_k}{2}
  \end{equation*}
  and hence
  \begin{equation*}
    \abs{\{\theta \sigma^k(X_s^{(1), k}) + (1 - \theta) \sigma^k(X_s^{(1), k})\}^{-1} } \lesssim \lambda_k^{-1}.
  \end{equation*}
  Therefore, we have $Z^{\theta, k} \in \cD_1(\alpha, 2 \alpha)$ with
  \begin{equation*}
    \norm{Z^{\theta, k}}_{\cD_1(\alpha, 2 \alpha)}
    \lesssim_{T, \alpha}
    (1 + \abs{x} + \norm{b}_{C^1_b} + \norm{\sigma}_{C^1_b} \norm{B}_{C^{\alpha}})^3
    + (1 + \lambda_k^{-1}) (\norm{b}_{L_{\infty}} + \norm{\sigma}_{L_{\infty}}).
  \end{equation*}
  Since
  \begin{equation*}
    \max \Big\{ \frac{1 - H}{2 \alpha}, \frac{2 - 3 H + H^2}{2 \alpha +
     H - H^2}, \frac{3 - \sqrt{3}}{2 H} - 1 \Big\} =  \frac{2 - 3 H + H^2}{2 \alpha +
     H - H^2}< \delta
  \end{equation*}
  if $\alpha$ is sufficiently close to $H$,
  by Proposition \ref{prop:stochastic_integral_controlled_path},
  \begin{equation*}
    \norm{ V^{k, n_1}_{s, t} - V^{k, n_2}_{s, t}}_{L_m(\P)}
    \lesssim_{T, \alpha, \delta, b, \sigma, k, m}  \norm{\sigma^{k, n_1} - \sigma^{k, n_2}}_{C^{1 + \delta}(\R^d, \cM_d)}
    \abs{t-s}^{H}.
  \end{equation*}
  By Kolmogorov's continuity theorem, the sequence $(V^{k, n})_{n \in \N}$ converges to some $V^k$ in $C^{\alpha}([0, T], \R^d)$.

  Therefore, we conclude that almost surely in $\Omega_k$, the path $z = X^{(1)} - X^{(2)}$ solves the linear Young equation
  \begin{equation*}
    z_t = \int_0^t z_r \dd U_r^k, \quad U_t^k \defby V_t^k + \int_0^t \int_0^1 \nabla b(\theta X^{(1),k}_r + (1-\theta) X^{(2), k}_r) \dd \theta \dd r,
  \end{equation*}
  and hence $X^{(1)} = X^{(2)}$. Since $\P(\Omega_k) \to 1$, we conclude $X^{(1)} = X^{(2)}$ almost surely.
  Thus, we completed the proof of the uniqueness under (b). The other cases can be handled similarly.
  Indeed, under (a) we have $X^{(i)} \in \cD_0(\alpha, 2 \alpha)$ and under (c) we have $X^{(i)} \in \cD_0(\alpha, 1)$.

  Now, it remains to prove the existence of a strong solution. However, in view of the Yamada-Watanabe theorem (Proposition~\ref{prop:yamada_watanabe}),
  it suffices to show the existence of a weak solution, which will be proved in Lemma~\ref{lem:weak_solution} based on a standard compactness argument.
\end{proof}
\begin{remark}\label{rem:optimal_regularity_of_sigma}
  We believe that our assumption in Theorem~\ref{thm:regularization_by_noise_for_multiplicative_noise} is not optimal.
  One possible approach to relax the assumption is to consider a higher order approximation in \eqref{eq:sigma_young_approximation}.
  Yet, we believe that this will not lead to an optimal assumption, as long as we apply Lemma~\ref{lem:stochastic-sewing-tradeoff}.
  Thus, finding an optimal regularity of $\sigma$ for the pathwise uniqueness and the strong existence remains an interesting open question
  that is likely to require a new idea.
\end{remark}

\appendix
\section{Proofs of technical results}\label{app:technical}
\subsubsection*{Proofs of Lemma~\ref{lem:refine_A_t_0_t_1} and Lemma~\ref{lem:A_and_A_dash}}
\begin{proof}[Proof of Lemma~\ref{lem:refine_A_t_0_t_1} without \eqref{eq:alpha_and_beta_technical}]
  Let us first recall our previous strategy under \eqref{eq:alpha_and_beta_technical}. We used Lemma~\ref{lem:dyadic_allocation} to write
  \begin{equation}\label{eq:A_t_0_t_N_dyadic_expansion}
    A_{t_0, t_N} - \sum_{i=1}^N A_{t_{i-1}, t_i} = \sum_{n \in \N_0} \sum_{i=0}^{2^n - 1}
    R^n_i.
  \end{equation}
  Then, we decomposed
  \begin{multline}\label{eq:R_n_i_decomposition_appendix}
    \sum_{i=0}^{2^n - 1} R^n_i
    = \sum_{l = 0}^{L - 1} \sum_{j=0}^{2^n/L } \Big( R^n_{Lj + l} -
    \expect[R^n_{L j + l} \vert \cF_{L (j-1) + l + 1}^n] \indic_{\{j \geq 1\}} \Big)  \\
    + \sum_{l = 0}^{L - 1} \sum_{j=1}^{2^n/L }
    \expect[R^n_{L j + l} \vert \cF_{L (j-1) + l + 1}^n],
  \end{multline}
  where $\cF^n_k \defby \cF_{t_0 + \frac{k}{2^n}(t_N - t_0)}$.
  We estimated the first term of \eqref{eq:R_n_i_decomposition_appendix} by the BDG inequality and \eqref{eq:sewing_one_half}:
  \begin{multline}\label{eq:R_BDG}
    \norm{\sum_{l = 0}^{L - 1} \sum_{j=0}^{2^n/L } \Big( R^n_{Lj + l} -
    \expect[R^n_{L j + l} \vert \cF_{L (j-1) + l + 1}^n] \indic_{\{j \geq 1\}} \Big)}_{L_m(\P)}
    \lesssim \kappa_{m, d} \Gamma_2 L^{\frac{1}{2}} 2^{-n(\beta_2 - \frac{1}{2})} \abs{t_N - t_0}^{\beta_2}.
  \end{multline}
  In the proof under \eqref{eq:alpha_and_beta_technical}, we estimated the second term of \eqref{eq:R_n_i_decomposition_appendix} by the
  triangle inequality and \eqref{eq:sewing_regularization}:
  \begin{equation}\label{eq:R_triangle}
    \norm{\sum_{l = 0}^{L - 1} \sum_{j=1}^{2^n/L }
    \expect[R^n_{L j + l} \vert \cF_{L (j-1) + l + 1}^n]}_{L_m(\P)}
    \lesssim_{\alpha} \Gamma_1 L^{-\alpha} 2^{-(\beta_1 - \alpha - 1) n} \abs{t_N - t_0}^{\beta_1 - \alpha}.
  \end{equation}
  Then, we chose $L$ so that both \eqref{eq:R_BDG} and \eqref{eq:R_triangle} are summable with respect to $n$, for which to be possible,
  we had to assume \eqref{eq:alpha_and_beta_technical}.

  In order to remove the assumption \eqref{eq:alpha_and_beta_technical}, let us think again why we did the decomposition \eqref{eq:R_n_i_decomposition_appendix}.
  This is because we do not want to apply the simplest estimate, namely the triangle inequality, since the condition \eqref{eq:sewing_regularization}
  implies that $(A_{s,t})_{[s,t] \in \pi}$ are not so correlated.
  This point of view teaches us that, to estimate
  \begin{equation*}
    \sum_{l = 0}^{L - 1} \sum_{j=1}^{2^n/L }
    \expect[R^n_{L j + l} \vert \cF_{L (j-1) + l + 1}^n],
  \end{equation*}
  we should not simply apply the triangle inequality. That is, we should again apply the decomposition as in \eqref{eq:R_n_i_decomposition_appendix}.

  To carry out our new strategy, set
  \begin{equation*}
    S^{(1), l}_j \defby R^n_{L j + l}, \quad \cG^{(1), l}_j \defby \cF_{ L(j-1) + l +1 }^n, \qquad j \in \N.
  \end{equation*}
  \tymchange{For this new strategy, we can set $L \defby \max\{2, \lceil M \rceil\}$. 
  In particular, $L$ does not depend on $n$.
  }
  We use the convention $\expect[X \vert \cG^{(1), l}_j] = 0$ for $j \leq 0$.
  Then,
  \begin{align}
  \sum_{l = 0}^{L - 1} \sum_{j=1}^{2^n/L }
    \expect[R^n_{L j + l} \vert \cF_{L (j-1) + l + 1}^n]
    &=
    \sum_{l=0}^{L-1}  \sum_{j=1}^{L^{-1} 2^n} \expect[S^{(1), l}_j \vert \cG^{(1), l}_j] \notag \\
    \label{eq:S_decomposition_second_level}
    &= \sum_{l_1=0}^{L-1} \sum_{l_2=0}^{L-1} \sum_{j=0}^{L^{-2} 2^n}
    \expect[S^{(1), l_1}_{jL + l_2} \vert \cG^{(1), l_1}_{jL + l_2 }].
  \end{align}
  By setting
  \begin{equation*}
    S^{(2), l_1, l_2}_{j} \defby S^{(1), l_1}_{jL + l_2},
    \quad \cG^{(2), l_1, l_2}_j \defby \cG^{(1), l_1}_{(j-1)L + l_2},
  \end{equation*}
  the quantity \eqref{eq:S_decomposition_second_level} equals to
  \begin{multline*}
    \sum_{l_1=0}^{L} \sum_{l_2=0}^L \sum_{j=0}^{L^{-2} 2^n}
    \big(\expect[S^{(2), l_1, l_2}_j \vert \cG^{(2), l_1, l_2}_{j+1}]
    - \expect[S^{(2), l_1, l_2}_j \vert \cG^{(2), l_1, l_2}_{j}] \big) \\
    + \sum_{l_1=0}^{L} \sum_{l_2=0}^L \sum_{j=0}^{L^{-2} 2^n}
    \expect[S^{(2), l_1, l_2}_j \vert \cG^{(2), l_1, l_2}_{j }]
  \end{multline*}
  The $L_m(\P)$-norm of the first term can be estimated by the BDG inequality: it is bounded by
  \begin{equation}\label{eq:martingale_difference_second_level}
    2 \kappa_{m, d} \sum_{l_1, l_2 \leq L} \Big( \sum_{j \leq L^{-2} 2^n}
    \norm{\expect[S^{(2), l_1, l_2}_j \vert \cG^{(2), l_1, l_2}_{j+1}]}_{L_m(\P)}^2
    \Big)^{\frac{1}{2}}.
  \end{equation}
  By \eqref{eq:sewing_regularization}, we have
  \begin{equation*}
    \norm{\expect[S^{(2), l_1, l_2}_j \vert \cG^{(2), l_1, l_2 }_{j+1}]}_{L_m(\P)}
    \leq \Gamma_1 (L 2^{-n} \abs{t_N - t_0})^{-\alpha} (2^{-n} \abs{t_N - t_0})^{\beta_1}.
  \end{equation*}
  Therefore, the quantity \eqref{eq:martingale_difference_second_level} is bounded by
  \begin{equation*}
    2 \kappa_{m, d} \Gamma_1 L^{1 - \alpha} 2^{- n(\beta_1 - \alpha - \frac{1}{2})} \abs{t_N - t_0}^{\beta_1 - \alpha}.
  \end{equation*}

  As the reader may realize, we will repeat the same argument for
  \begin{equation*}
    \sum_{l_1=0}^{L} \sum_{l_2=0}^L \sum_{j=1}^{L^{-2} 2^n}
    \expect[S^{(2), l_1, l_2}_j \vert \cG^{(2), l_1, l_2}_{j - 1}]
  \end{equation*}
  and continue. To state more precisely,
  set inductively,
  \begin{equation}\label{eq:S_and_G}
    S^{(k), l_1, \ldots, l_k}_j \defby S^{(k-1), l_1, \ldots, l_{k-1}}_{Lj + l_k}, \quad
    \cG^{(k), l_1, \ldots, l_k}_j \defby \cG^{(k - 1), l_1, \ldots, l_{k-1}}_{L(j-1) + l_k}, \quad j \in [1, L^{-k} 2^n] \cap \N.
  \end{equation}
  We claim that, if $L^k  \leq 2^n$, we have
  \begin{multline*}
    \norm{\sum_{i=0}^{2^n - 1} R^n_i}_{L_m(\P)} \leq 2 \kappa_{m, d} \Gamma_2 L^{\frac{1}{2}}
    2^{-n(\beta_2 - \frac{1}{2})} \abs{t_N - t_0}^{\beta_2} \\
    +2 \kappa_{m, d} \Gamma_1 \Big(\sum_{j=1}^{k-1} L^{\frac{j}{2} - (j-1) \alpha} \Big)
    L^{\frac{1}{2}} 2^{-n(\beta_1 - \alpha - \frac{1}{2})} \abs{t_N - t_0}^{\beta_1 - \alpha} \\
    + \norm{\sum_{l_1, \ldots, l_k \leq L} \sum_{j \leq L^{-k} 2^{n}} \expect[S^{(k), l_1, \ldots, l_k}_j
    \vert \cG^{(k), l_1, \ldots,  l_k}_{j }]}_{L_m(\P)}.
  \end{multline*}
  The proof of the claim is based on induction. The case $k=1$ and $k=2$ is obtained.
  Suppose that the claim is correct for $k \geq 2$, and consider the case $k+1$.
  Again, decompose
  \begin{multline*}
    \sum_{l_1, \ldots, l_k \leq L} \sum_{j \leq L^{-k} 2^{n}} \expect[S^{(k), l_1, \ldots, l_k}_j
    \vert \cG^{(k), l_1, \ldots,  l_k}_{j }] \\
    =
    \sum_{l_1, \ldots, l_k, l_{k+1} \leq L} \sum_{j \leq L^{-(k+1)} 2^n}
    \big(\expect[S^{(k+1), l_1, \ldots, l_k}_j \vert \cG^{(k+1), l_1, \ldots, l_k, l_{k+1} }_{j+1}] \\
    \hspace{5cm}- \expect[S^{(k+1), l_1, \ldots, l_k}_j \vert \cG^{(k+1), l_1, \ldots,  l_k, l_{k+1}}_{j}] \big) \\
    + \sum_{l_1, \ldots, l_k, l_{k+1} \leq L} \sum_{j \leq L^{-(k+1)} 2^n}
     \expect[S^{(k+1), l_1, \ldots, l_k}_j \vert \cG^{(k+1), l_1, \ldots,  l_k, l_{k+1}}_{j}].
  \end{multline*}
  To prove the claim, it suffices to estimate the first sum in the right-hand side. By the BDG inequality, its
  $L_m(\P)$-norm is bounded by
  \begin{equation}\label{eq:martingale_difference_kth_level}
    2 \kappa_{m, d} \sum_{l_1, \ldots, l_k, l_{k+1} \leq L} \Big(\sum_{j \leq L^{-(k+1)} 2^n}
    \norm{\expect[S^{(k+1), l_1, \ldots, l_k, l_{k+1}}_j \vert \cG^{(k+1), l_1, \ldots, l_k, l_{k+1} }_{j+1}]}_{L_m(\P)}^2
     \Big)^{1/2}.
  \end{equation}
  By \eqref{eq:sewing_regularization},
  \begin{equation}\label{eq:sewing_martingale_difference_kth_level}
    \norm{\expect[S^{(k+1), l_1, \ldots, l_k, l_{k+1}}_j \vert \cG^{(k+1), l_1, \ldots, l_k, l_{k+1}}_{j+1}]}_{L_m(\P)}
    \leq \Gamma_1 (L^k 2^{-n} \abs{t_N - t_0})^{-\alpha} (2^{-n} \abs{t_N - t_0})^{\beta_1}.
  \end{equation}
  Therefore, the quantity \eqref{eq:martingale_difference_kth_level} is bounded by
  \begin{equation*}
    2 \kappa_{m, d} \Gamma_1 L^{\frac{1}{2}} L^{(\frac{1}{2} - \alpha)k}
    2^{-n(\beta_1 - \alpha - \frac{1}{2})} \abs{t_N - t_0}^{(\beta_1 - \alpha)}
  \end{equation*}
  and the claim follows.

  Now let us estimate
  \begin{equation}\label{eq:kth_level_conditional_expectation}
    \norm{\sum_{l_1, \ldots, l_k \leq L} \sum_{j \leq L^{-k} 2^{n}} \expect[S^{(k), l_1, \ldots, l_k}_j
    \vert \cG^{(k), l_1, \ldots,  l_k}_{j }]}_{L_m(\P)}
  \end{equation}
  by the triangle inequality:
  \begin{multline*}
    \norm{\sum_{l_1, \ldots, l_k \leq L} \sum_{j \leq L^{-k} 2^{n}} \expect[S^{(k), l_1, \ldots, l_k}_j
    \vert \cG^{(k), l_1, \ldots,  l_k}_{j }]}_{L_m(\P)} \\
    \leq \sum_{l_1, \ldots, l_k \leq L} \sum_{j \leq L^{-k} 2^{n}} \norm{\expect[S^{(k), l_1, \ldots, l_k}_j
    \vert \cG^{(k), l_1, \ldots,  l_k}_{j }]}_{L_m(\P)}.
  \end{multline*}
  By \eqref{eq:sewing_regularization} (or essentially the estimate \eqref{eq:sewing_martingale_difference_kth_level}),
  \begin{equation*}
    \norm{\expect[S^{(k), l_1, \ldots, l_k}_j
    \vert \cG^{(k), l_1, \ldots,  l_k}_{j }]}_{L_m(\P)}
    \leq \Gamma_1 (L^k 2^{-n} \abs{t_N - t_0})^{-\alpha} (2^{-n} \abs{t_N - t_0})^{\beta_1}
  \end{equation*}
  and hence the quantity \eqref{eq:kth_level_conditional_expectation} is bounded by
  \begin{equation*}
    \Gamma_1 L^{ - \alpha k}
    2^{-n(\beta_1 - \alpha - 1)} \abs{t_N - t_0}^{\beta_1 - \alpha}
  \end{equation*}

  In conclusion,   we obtained for $L^k \leq 2^n$,
  \begin{multline}\label{eq:estimate_of_R_n_i_improved}
    \norm{\sum_{i=0}^{2^n - 1} R^n_i}_{L_m(\P)} \lesssim
    \kappa_{m, d} \Gamma_2 L^{\frac{1}{2}} 2^{-n(\beta_2 - \frac{1}{2})} \abs{t_N - t_0}^{\beta_2} \\
    +  \kappa_{m, d} \Gamma_1
     f_k 2^{-n(\beta_1 - \alpha - \frac{1}{2})} \abs{t_N - t_0}^{\beta_1 - \alpha}
    + \Gamma_1 L^{-\alpha k}
    2^{-n(\beta_1 - \alpha - 1)} \abs{t_N - t_0}^{\beta_1 - \alpha },
  \end{multline}
  where
  \begin{equation}\label{eq:def_of_f_k_L}
    f_k =
    \begin{cases}
      L^{\frac{k}{2} - \alpha(k-1)} \indic_{\{k \geq 2\}}, &\text{if } \alpha < \frac{1}{2}, \\
      (k-1) L^{\frac{1}{2}}, &\text{if } \alpha \geq \frac{1}{2}.
    \end{cases}
  \end{equation}
  \tymchange{
    We recall that $L = \max\{2, \lceil M \rceil\}$. To respect $L^k \leq 2^n$, we set 
    $k \defby \lfloor \frac{n \log 2}{\log L} \rfloor$. We then have 
    \begin{align*}
      f_k 2^{-n(\beta_1 - \alpha - \frac{1}{2})} \lesssim_{M} 
      \begin{cases}
      2^{-n (\beta_1 - 1)}, &\text{if } \alpha < \frac{1}{2}, \\
      n 2^{-n(\beta_1 - \alpha - \frac{1}{2})}, &\text{if } \alpha \geq \frac{1}{2},
      \end{cases}
    \end{align*}
    and 
    \begin{align*}
      L^{-\alpha k} 2^{-n(\beta_1 - \alpha - 1)} \lesssim_{\alpha, M} 2^{-n (\beta_1 - 1)}.
    \end{align*}
  }
  Therefore, we note that the right-hand side of \eqref{eq:estimate_of_R_n_i_improved} is summable with respect to $n$ and
  \begin{equation*}
    \norm{A_{t_0, t_N} - \sum_{i=1}^N A_{t_{i-1}, t_i} }_{L_m(\P)}
    \lesssim_{\alpha, \beta_1, \beta_2, M}
    \kappa_{m, d} \Gamma_2 \abs{t_N - t_0}^{\beta_2} +  \kappa_{m, d} \Gamma_1 \abs{t_N - t_0}^{\beta_1 - \alpha_1}. \qedhere
  \end{equation*}
\end{proof}

\begin{proof}[Proof of Lemma~\ref{lem:A_and_A_dash}]
  The proof is similar to Lemma~\ref{lem:refine_A_t_0_t_1}.
  Write
  \begin{equation*}
    \pi' =: \{0 = t_0 < t_1 < \cdots < t_{N-1} < t_N = T\}
  \end{equation*}
   and
  \begin{equation*}
    \set{[s, t] \in \pi \given t_j \leq s < t \leq t_{j+1}}
    =: \{t_j = t^j_0  < t^j_1 < \cdots < t^j_{N_j - 1} < t^j_{N_j} = t_{j+1}\}.
  \end{equation*}
  By \eqref{eq:pi_almost_uniform}, we have $N \leq 3 \abs{\pi'}^{-1} T$.
  We fix a parameter $L$, which will be chosen later,  and set
  \begin{equation*}
    Z_{j}^{(1), l} \defby A_{t_{jL + l}, t_{jL + l +1}} - \sum_{k=1}^{N_{jL + l}} A_{t^{jL + l}_{k-1}, t^{jL + l}_{k}},
    \quad \cH_{j}^{(1), l} \defby \cF_{t_{(j-1)L + l + 1}}.
  \end{equation*}
  Inductively, we set
  \begin{equation*}
    Z^{(k), l_1, \ldots, l_k}_j \defby Z^{(k-1), l_1, \ldots, l_{k-1}}_{j L + l_k},
    \quad
    \cH^{(k), l_1, \ldots, l_k}_j \defby \cH^{(k-1), l_1, \ldots, l_{k-1}}_{(j-1) L + l_k}.
  \end{equation*}
  As in Lemma~\ref{lem:refine_A_t_0_t_1}, for each $k \in \N$, we consider the decomposition
  \begin{equation*}
  A^{\pi'}_T - A^{\pi}_T = A + B
  \end{equation*}
  where
  \begin{align*}
    A &\defby \sum_{p=1}^k \sum_{l_1, \ldots, l_p \leq L} \sum_{j \leq N L^{-p}}
    \Big\{ \expect[Z^{(p), l_1, \ldots, l_p}_j \vert \cH_{j+1}^{(p), l_1, \ldots, l_p}] -  \expect[Z^{(p), l_1, \ldots, l_p}_j \vert \cH_{j}^{(p), l_1, \ldots, l_p}]
    \Big\} \\
    B &\defby \sum_{l_1, \ldots, l_k \leq L} \sum_{j \leq N L^{-k}}
      \expect[Z^{(k), l_1, \ldots, l_k}_j \vert \cH_{j}^{(k), l_1, \ldots, l_k}].
  \end{align*}
  For this decomposition, we must have $L^k \leq N$.
  By the BDG inequality and the Cauchy-Schwarz inequality,
  \begin{align*}
    \norm{A}_{L_m(\P)} \lesssim \kappa_{m, d}
    \sum_{p=1}^k
    L^{\frac{p}{2}}
    \Big(
    \sum_{l_1, \ldots, l_p \leq L} \sum_{j \leq N L^{-p}}
    \norm{\expect[Z^{(p), l_1, \ldots, l_p}_j \vert \cH_{j+1}^{(p), l_1, \ldots, l_p}]
    }_{L_m(\P)}^2
    \Big)^{\frac{1}{2}}.
  \end{align*}
  By Lemma~\ref{lem:refine_A_t_0_t_1},
  \begin{equation*}
    \norm{Z_{j}^{(1), l}}_{L_m(\P)} \lesssim_{\alpha, \beta_1, \beta_2} \Gamma_1 \abs{t_{jL+l+1} - t_{jL + l}}^{\beta_1 - \alpha}
    + \kappa_{m, d} \Gamma_2 \abs{t_{jL+l+1} - t_{jL + l}}^{\beta_2}.
  \end{equation*}
  For $p \geq 2$, by Lemma~\ref{lem:refine_A_t_0_t_1_conditioning} and \eqref{eq:pi_almost_uniform},
  \begin{equation*}
    \norm{ \expect[Z^{(p), l_1, \ldots, l_p}_j \vert \cH_{j+1}^{(p), l_1, \ldots, l_p}]}_{L_m(\P)}
    \lesssim_{\beta_1} \Gamma_1 L^{-(p-1) \alpha} \abs{\pi'}^{-\alpha} \abs{\pi'}^{\beta_1}.
  \end{equation*}
  Therefore, we obtain
  \begin{equation}\label{eq:estimate_of_A}
    \norm{A}_{L_m(\P)} \lesssim_{\alpha, \beta_1, \beta_2, m, d, T}  L^{\frac{1}{2}}(\Gamma_1 \abs{\pi'}^{\beta_1 - \alpha - \frac{1}{2}} + \Gamma_2 \abs{\pi'}^{\beta_2 - \frac{1}{2}})
    + \Gamma_1 f_k \abs{\pi'}^{\beta_1 - \alpha - \frac{1}{2}},
  \end{equation}
  where $f_k$ is defined by \eqref{eq:def_of_f_k_L}.

  We move to estimate $B$. By Lemma~\ref{lem:refine_A_t_0_t_1_conditioning} and \eqref{eq:pi_almost_uniform},
  \begin{equation*}
    \norm{\expect[Z_{j}^{(k), l_1, \ldots, l_k} \vert \cH_{j}^{(k), l_1, \ldots, l_k}]}_{L_m(\P)}
    \lesssim_{\beta_1} \Gamma_1 L^{-\alpha k} \abs{\pi'}^{\beta_1 - \alpha}.
  \end{equation*}
  Therefore,
  \begin{equation}\label{eq:estimate_of_B}
    \norm{B}_{L_m(\P)} \lesssim_{\beta_1, T} \Gamma_1 L^{-\alpha k}
    \abs{\pi'}^{\beta_1 - \alpha - 1}.
  \end{equation}
  Combining \eqref{eq:estimate_of_A} and \eqref{eq:estimate_of_B}, we obtain
  \begin{multline*}
    \norm{A^{\pi'}_T - A^{\pi}_T}_{L_m(\P)} \\
    \lesssim_{\alpha, \beta_1, \beta_2, m, d, T}
    L^{\frac{1}{2}}(\Gamma_1 \abs{\pi'}^{\beta_1 - \alpha - \frac{1}{2}} + \Gamma_2 \abs{\pi'}^{\beta_2 - \frac{1}{2}})
    + \Gamma_1 f_{k} \abs{\pi'}^{\beta_1 - \alpha - \frac{1}{2}}
    + \Gamma_1 L^{-\alpha k}
    \abs{\pi'}^{\beta_1 - \alpha - 1}.
  \end{multline*}
  \tymchange{
  As in the proof of Lemma~\ref{lem:refine_A_t_0_t_1}, we set $L \defby \max\{2, \lceil M \rceil\}$ and $k \defby \lfloor \frac{\log N}{\log L} \rfloor$.
  We then obtain the claimed estimate.
  }
\end{proof}
\subsubsection*{Proof of Corollary~\ref{cor:singular_generalized_stochastic_sewing_lemma}}
The argument is similar to that of Theorem~\ref{thm:generalized_stochastic_sewing}. Therefore, we only prove an
analogue of Lemma~\ref{lem:refine_A_t_0_t_1}.

\begin{proof}[Analogue of Lemma~\ref{lem:refine_A_t_0_t_1}]
  Given a partition
  \begin{equation*}
    0 \leq t_0 < t_1 < \cdots < t_{N-1} < t_N \leq T,
  \end{equation*}
  we have
  \begin{equation}\label{eq:A_t_0_t_N_refinement_singular}
    \norm{A_{t_0, t_N} - \sum_{i=1}^N A_{t_{i-1}, t_i}}_{L_m(\P)}
    \lesssim_{\alpha, \beta_1, \beta_2, M}
    \kappa_{m, d} \Gamma_1 t_0^{-\gamma_1} \abs{t_N - t_0}^{\beta_1 - \alpha}
    + \kappa_{m, d} \Gamma_2 t_0^{-\gamma_2} \abs{t_N - t_0}^{\beta_2},
  \end{equation}
  \begin{multline}\label{eq:A_t_0_t_N_refinement_singular_continuity}
    \norm{A_{t_0, t_N} - \sum_{i=1}^N A_{t_{i-1}, t_i}}_{L_m(\P)} \\
    \lesssim_{\alpha, \beta_1, \beta_2, \beta_3, \gamma_1, \gamma_2, M}
    \kappa_{m,d}
    (\Gamma_1 \abs{t_N - t_0}^{\beta_1 - \alpha_1 - \gamma_1}
    + \Gamma_2 \abs{t_N - t_0}^{\beta_2 - \gamma_2} + \Gamma_3 \abs{t_N - t_0}^{\beta_3}),
  \end{multline}
  where $t_0 > 0$ is assumed for \eqref{eq:A_t_0_t_N_refinement_singular}.
  In fact, the proof of \eqref{eq:A_t_0_t_N_refinement_singular} is the same as that of Lemma~\ref{lem:refine_A_t_0_t_1},
  since we can simply replace $\Gamma_1$ and $\Gamma_2$ by $t_0^{-\gamma_1} \Gamma_1$ and $t_0^{-\gamma_2} \Gamma_2$.
  Therefore, we focus on proving \eqref{eq:A_t_0_t_N_refinement_singular_continuity}.

  Yet, the proof of \eqref{eq:A_t_0_t_N_refinement_singular_continuity} is similar to that of Lemma~\ref{lem:refine_A_t_0_t_1}.
  Recalling the notation therein, namely \eqref{eq:A_t_0_t_N_dyadic_expansion} and \eqref{eq:S_and_G}, we have
  \begin{equation*}
    A_{t_0, t_N} - \sum_{i=1}^N A_{t_{i-1}, t_i} = \sum_{n=0}^{\infty} \sum_{i=0}^{2^n} R^n_i,
  \end{equation*}
  \begin{multline}\label{eq:R_n_i_expansion}
     \sum_{i=0}^{2^n} R^n_i
    = \sum_{p=1}^k \sum_{l_1, \ldots, l_p \leq L} \sum_{j \leq 2^n L^{-p}}
    \Big\{ \expect[S_j^{(p), l_1, \ldots, l_p} \vert \cG_{j+1}^{(p), l_1, \ldots, l_p}] -
    \expect[S_j^{(p), l_1, \ldots, l_p} \vert \cG_{j}^{(p), l_1, \ldots, l_p}] \Big\}  \\
    + \sum_{l_1, \ldots, l_k \leq L} \sum_{j \leq 2^n L^{-k}} \expect[S_j^{(k), l_1, \ldots, l_k} \vert \cG_{j}^{(k), l_1, \ldots, l_k}].
  \end{multline}
  We fix a large $n$.
  To estimate the first term of \eqref{eq:R_n_i_expansion}, we apply the BDG inequality to obtain
  \begin{multline*}
    \norm{\sum_{j \leq 2^n L^{-p}}
    \Big\{ \expect[S_j^{(p), l_1, \ldots, l_p} \vert \cG_{j+1}^{(p), l_1, \ldots, l_p}] -
    \expect[S_j^{(p), l_1, \ldots, l_p} \vert \cG_{j}^{(p), l_1, \ldots, l_p}] \Big\}}_{L_m(\P)} \\
    \lesssim \kappa_{m, d} \Big( \sum_{j \leq 2^n L^{-p}} \norm{\expect[S_j^{(p), l_1, \ldots, l_p} \vert \cG_{j+1}^{(p), l_1, \ldots, l_p}]}_{L_m(\P)}^2 \Big)^{\frac{1}{2}}
  \end{multline*}

  For $p=1$,
   since $S^{(1), l_1}_j = R^n_{j L + l_1}$, by the Cauchy-Schwarz inequality,
  \begin{equation*}
    \sum_{l_1 \leq L}
    \Big( \sum_{j \leq 2^n L^{-1}} \norm{S_j^{(1), l_1}}_{L_m(\P)}^2 \Big)^{\frac{1}{2}}
    \leq L^{\frac{1}{2}} \Big( \sum_{i=0}^{2^n} \norm{R^n_i}_{L_m(\P)}^2 \Big)^{\frac{1}{2}}.
  \end{equation*}
  For $i \geq 1$ by \eqref{eq:sewing_one_half_singular} we have
  \begin{equation*}
    \norm{R^n_i}_{L_m(\P)} \leq 2 \Gamma_2 (t_0 + 2^{-n} i \abs{t_N - t_0})^{-\gamma_2}
    (2^{-n} \abs{t_N - t_0})^{\beta_2}
  \end{equation*}
  and by \eqref{eq:sewing_continuity_at_zero}
  \begin{equation*}
    \norm{R^n_0}_{L_m(\P)} \leq 2 \Gamma_3 (2^{-n} \abs{t_N - t_0})^{\beta_3}.
  \end{equation*}
  Therefore,
  \begin{equation*}
    \Big(\sum_{i=0}^{2^n} \norm{R^n_i}_{L_m(\P)}^2 \Big)^{\frac{1}{2}}
    \lesssim \Gamma_3 2^{-n \beta_3} \abs{t_N - t_0}^{\beta_3}
    + \Gamma_2 2^{-n \beta_2} \abs{t_N - t_0}^{\beta_2} \Big( \sum_{i=1}^{2^n}
    (t_0 + 2^{-n} i \abs{t_N - t_0})^{-2 \gamma_2} \Big)^{\frac{1}{2}}.
  \end{equation*}
  We observe
  \begin{equation}\label{eq:negative_power_integral}
    \sum_{i=1}^{2^n}
    (t_0 + 2^{-n} i \abs{t_N - t_0})^{-2 \gamma_2}
    \leq 2^n \abs{t_N - t_0}^{-1} \int_0^{\abs{t_N - t_0}} s^{-2 \gamma_2} \dd s
    = \frac{2^n \abs{t_N - t_0}^{- 2 \gamma_2}}{1 - 2 \gamma_2},
  \end{equation}
  where the condition $\gamma_2 < \frac{1}{2}$ is used.
  We conclude
  \begin{equation*}
    \sum_{l_1 \leq L}
    \Big( \sum_{j \leq 2^n L^{-1}} \norm{S_j^{(1), l_1}}_{L_m(\P)}^2 \Big)^{\frac{1}{2}}
    \lesssim_{\gamma_2}
    L^{\frac{1}{2}}
    (\Gamma_3 2^{-n \beta_3} \abs{t_N - t_0}^{\beta_3} +
    \Gamma_2 2^{-n (\beta_2 - \frac{1}{2})} \abs{t_N - t_0}^{\beta_2 - \gamma_2}
    ).
  \end{equation*}

  For $2 \leq p \leq k$, the argument is similar but now we use \eqref{eq:sewing_regularization_singular}.
  By the Cauchy-Schwarz inequality,
  \begin{multline*}
    \sum_{l_1, \ldots, l_p \leq L}
    \Big( \sum_{j \leq 2^n L^{-p}} \norm{\expect[S_j^{(p), l_1, \ldots, l_p} \vert \cG_{j+1}^{(p), l_1, \ldots, l_p}]}_{L_m(\P)}^2 \Big)^{\frac{1}{2}} \\
    \leq L^{\frac{p}{2}}
    \Big(
      \sum_{l_1, \ldots, l_p \leq L}
      \sum_{j \leq 2^n L^{-p}} \norm{\expect[S_j^{(p), l_1, \ldots, l_p} \vert \cG_{j+1}^{(p), l_1, \ldots, l_p}]}_{L_m(\P)}^2
    \Big)^{\frac{1}{2}}.
  \end{multline*}
  We note that for each index $l_1, \ldots, l_p$ and $j$ in the sum there exists a unique
  $i = i(l_1, \ldots, l_p; j)$ such that
  \begin{equation*}
    S_j^{(p), l_1, \ldots, l_p} = R^n_i.
  \end{equation*}
  As $p \geq 2$, we know $i \geq L$.
  By \eqref{eq:sewing_regularization_singular} (as in the estimate \eqref{eq:sewing_martingale_difference_kth_level}),
  \begin{multline*}
    \norm{\expect[S_j^{(p), l_1, \ldots, l_p} \vert \cG_{j+1}^{(p), l_1, \ldots, l_p}]}_{L_m(\P)} \\
    \leq 2 \Gamma_1 (L^{p-1} 2^{-n} \abs{t_N - t_0})^{-\alpha}
    (t_0 + 2^{-n} i(l_1, \ldots, l_p;j) \abs{t_N - t_0})^{- \gamma_1}
    (2^{-n } \abs{t_N - t_0})^{\beta_1}.
  \end{multline*}
  Therefore,
  \begin{align*}
    \MoveEqLeft[3]
    \Big(
      \sum_{l_1, \ldots, l_p \leq L}
      \sum_{j \leq 2^n L^{-p}} \norm{\expect[S_j^{(p), l_1, \ldots, l_p} \vert \cG_{j+1}^{(p), l_1, \ldots, l_p}]}_{L_m(\P)}^2
    \Big)^{\frac{1}{2}} \\
    &\lesssim
    \Gamma_1 (L^{p-1} 2^{-n} \abs{t_N - t_0})^{-\alpha}
    (2^{-n } \abs{t_N - t_0})^{\beta_1}
    \Big(
      \sum_{i=1}^{2^n}
      (t_0 + 2^{-n} i \abs{t_N - t_0})^{- 2\gamma_1}
    \Big)^{\frac{1}{2}} \\
    &\lesssim_{\gamma_1}
    \Gamma_1 L^{-\alpha(p-1)} 2^{- n (\beta_1 - \alpha - \frac{1}{2})}
    \abs{t_N - t_0}^{\beta_1 - \alpha - \gamma_1 },
  \end{align*}
  where to obtain the second inequality we applied the estimate \eqref{eq:negative_power_integral}.

  Now we consider the estimate of the second term of \eqref{eq:R_n_i_expansion}. By the triangle inequality,
  \begin{multline*}
    \norm{\sum_{l_1, \ldots, l_k \leq L} \sum_{j \leq 2^n L^{-k}}
    \expect[S_j^{(k), l_1, \ldots, l_k} \vert \cG_{j}^{(k), l_1, \ldots, l_k}]}_{L_m(\P)} \\
    \leq
    \sum_{l_1, \ldots, l_k \leq L} \sum_{j \leq 2^n L^{-k}}
    \norm{\expect[S_j^{(k), l_1, \ldots, l_k} \vert \cG_{j}^{(k), l_1, \ldots, l_k}]}_{L_m(\P)}.
  \end{multline*}
  But the estimate of the right hand side was just discussed. In fact, we have
  \begin{equation*}
    \sum_{l_1, \ldots, l_k \leq L} \sum_{j \leq 2^n L^{-k}}
    \norm{\expect[S_j^{(k), l_1, \ldots, l_k} \vert \cG_{j}^{(k), l_1, \ldots, l_k}]}_{L_m(\P)}
    \lesssim_{\gamma_1}
    \Gamma_1 L^{-\alpha k} 2^{- n (\beta_1 - \alpha - 1)}
    \abs{t_N - t_0}^{\beta_1 - \alpha - \gamma_1 }.
  \end{equation*}

  Hence, we obtain the estimate
  \begin{multline*}
    \norm{\sum_{i=0}^{2^n} R^n_i}_{L_m(\P)}
    \lesssim_{\gamma_1, \gamma_2}
      \kappa_{m, d} L^{\frac{1}{2}}
    (\Gamma_3 2^{-n \beta_3} \abs{t_N - t_0}^{\beta_3} +
    \Gamma_2 2^{-n (\beta_2 - \frac{1}{2})} \abs{t_N - t_0}^{\beta_2 - \gamma_2}
    ) \\
    + \kappa_{m, d} \Gamma_1 \sum_{p=2}^k  L^{\frac{p}{2} -\alpha(p-1)} 2^{- n (\beta_1 - \alpha - \frac{1}{2})}
    \abs{t_N - t_0}^{\beta_1 - \alpha - \gamma_1 } \\
    + \Gamma_1 L^{-\alpha k} 2^{- n (\beta_1 - \alpha - 1)}
    \abs{t_N - t_0}^{\beta_1 - \alpha - \gamma_1 }.
  \end{multline*}
  By choosing $L$ and $k$ exactly as in the proof of Lemma~\ref{lem:refine_A_t_0_t_1},
  we conclude that there exists an $\epsilon = \epsilon(\alpha, \beta_1, \beta_2, \beta_3) > 0$
  such that for all large $n$
  \begin{multline*}
    \norm{\sum_{i=0}^{2^n} R^n_i}_{L_m(\P)} \\
    \lesssim_{\alpha, \beta_1, \beta_2, \beta_3, \gamma_1, \gamma_2}
    \kappa_{m,d} 2^{-n \epsilon}
    (\Gamma_1 \abs{t_N - t_0}^{\beta_1 - \alpha_1 - \gamma_1}
    + \Gamma_2 \abs{t_N - t_0}^{\beta_2 - \gamma_2} + \Gamma_3 \abs{t_N - t_0}^{\beta_3}),
  \end{multline*}
  from which we obtain \eqref{eq:A_t_0_t_N_refinement_singular_continuity}.
\end{proof}
\subsubsection*{Proof of Lemma~\ref{lem:kernel_correlation}}
  Let $d = 1$.
  For $u > v$, we set
  \begin{equation*}
    B^{(1)}_u \defby \int_{-\infty}^v K(u, r) \dd W_r, \quad
    B^{(2)}_u \defby \int_{v}^u K(u, r) \dd W_r
  \end{equation*}
  so that $B_u - B(0) = B^{(1)}_u + B^{(2)}_u$ and $B^{(1)}_{\cdot}$ and $B^{(2)}_{\cdot}$ are independent.
  Then, we have
  \begin{equation*}
    \expect[B^{(2)}_s B^{(2)}_t] = \int_v^s K(s, r) K(t, r) \dd r.
  \end{equation*}
  and by \eqref{eq:correlation_fbm}, we have
  \begin{equation*}
    \frac{c_H}{2}(s^{2H} + t^{2H} - \abs{t-s}^{2H})
    = \expect[B^{(1)}_s B^{(1)}_t] + \expect[B^{(2)}_s B^{(2)}_t],
  \end{equation*}
  and thus, we will estimate $\expect[B^{(1)}_s B^{(1)}_t]$.
  We have
  \begin{multline}\label{eq:correlation_B_1}
    \expect[B^{(1)}_s B^{(1)}_t]
    = \int_0^{\infty} \big[(s+r)^{H-1/2} - r^{H-1/2} \big] \big[(t+r)^{H-1/2} - r^{H-1/2} \big] \dd r \\
    + \int_0^v (s-r)^{H-1/2} (t-r)^{H-1/2} \dd r.
  \end{multline}
  By \cite[Theorem~33]{picard11}, the first term of \eqref{eq:correlation_B_1} equals to
  \begin{equation}\label{eq:correlation_B_1_first_term}
    (c_H - (2H)^{-1}) s^{2H} + \int_0^{\infty} \big[(s+r)^{H-1/2} - r^{H-1/2} \big]
    \big[(t+r)^{H-1/2} - (s+r)^{H-1/2} \big] \dd r.
  \end{equation}
  Since
  \begin{equation*}
    (t+r)^{H-1/2} - (s+r)^{H-1/2}
    = (H - 1/2) (s+r)^{H-3/2} (t-s) + O((s+r)^{H-5/2}(t-s)^2),
  \end{equation*}
  the second term of \eqref{eq:correlation_B_1_first_term} equals to
  \begin{equation*}
    s^{2H-1} (t-s) (H - 1/2) \int_0^{\infty} \big[(1+r)^{H-1/2} - r^{H-1/2} \big] (1+r)^{H-3/2} \dd r
    + O(s^{2H - 2} (t-s)^2).
  \end{equation*}
  By \cite[Theorem~33]{picard11},
  \begin{equation*}
    (H - 1/2) \int_0^{\infty} \big[(1+r)^{H-1/2} - r^{H-1/2} \big] (1+r)^{H-3/2} \dd r
    = -\frac{1}{2} + H c_H.
  \end{equation*}
  Similarly, the second term of \eqref{eq:correlation_B_1} equals to
  \begin{equation*}
    \frac{1}{2H} (s^{2H} - (s-v)^{2H})
    + \frac{t-s}{2} (s^{2H-1} - (s-v)^{2H-1}) + O((s-v)^{2H-2} (t-s)^2).
  \end{equation*}
  Therefore, $\expect[B^{(1)}_s B^{(1)}_t]$ equals to
  \begin{equation*}
    c_H s^{2H} + H c_H s^{2H-1} (t-s) -\frac{1}{2H} (v-s)^{2H}
    -\frac{1}{2} (s-v)^{2H-1} (t-s) +
    O((s-v)^{2H-2} (t - s)^2).
  \end{equation*}
  Since
  \begin{multline*}
    \frac{c_H}{2} (s^{2H} + t^{2H} - \abs{t-s}^{2H}) -
    c_H s^{2H} + H c_H s^{2H-1} (t-s) -\frac{1}{2H} (v-s)^{2H} \\
    = - \frac{c_H}{2} (t-s)^{2H} + O((s-v)^{2H-2} (t-s)^2),
  \end{multline*}
  the proof is complete.
\section{Yamada-Watanabe theorem for fractional SDEs}
We consider a Young differential equation
\begin{equation}\label{eq:young_fbm_appendix}
  \dd X_t = b(X_t) \dd t + \sigma(X_t) \dd B_t, \quad X_0 = x,
\end{equation}
where $b \in L_{\infty}(\R^d,\R^d)$ and
$B$ is an $(\cF_t)_{t \in \R}$ fractional Brownian motion with Hurst parameter $H \in (\frac{1}{2}, 1)$.
We fix $\alpha \in (\frac{1}{2}, H)$ and we assume that $\sigma \in C^{\frac{1-\alpha}{\alpha}}(\R^d; \cM_d)$ so that the integral
\begin{equation*}
  \int_s^t \sigma(X_r) \dd B_r
\end{equation*}
is interpreted as a Young integral.
\begin{definition}\label{def:}
  We say that a quintuple $(\Omega, (\cF_t)_{t \in \R}, \P, B, X)$ is a \emph{weak solution}
  to \eqref{eq:young_fbm_appendix} if $(B, X)$ are random variables defined on the filtered probability space $(\Omega, (\cF_t), \R)$,
  if $B$ is an $(\cF_t)$-fractional
  Brownian motion, if $X \in C^{\alpha}([0, T])$ is adapted to $(\cF_t)$ and if $X$ solves the Young differential equation \eqref{eq:young_fbm_appendix}.
  Given a filtered probability space $(\Omega, (\cF_t)_{t \in \R}, \P)$ and an $(\cF_t)$-fractional
  Brownian motion $B$, we say that a $C^{\alpha}([0,T])$-valued random variable $X$ defined on $(\Omega, (\cF_t)_{t \in \R}, \P)$ is a
  \emph{strong solution} if it solves \eqref{eq:young_fbm_appendix} and if it is adapted to the natural filtration generated by $B$.
  We say that the \emph{pathwise uniqueness} holds for \eqref{eq:young_fbm_appendix} if, for any two adapted
  $C^{\alpha}([0, T])$-valued random process $X$ and $Y$
  defined on a common filtered probability space that solve \eqref{eq:young_fbm_appendix} driven by a common $(\cF_t)$-Brownian motion, we have
  $X = Y$ almost surely.
\end{definition}
We will prove a Yamada-Watanabe type theorem for \eqref{eq:young_fbm_appendix} based on Kurtz~\cite{kurtz07}.
To this end, we recall that an $(\cF_t)$-fractional Brownian motion has the representation \eqref{eq:fbm_def}, and we view
\eqref{eq:young_fbm_appendix} as an equation of $X$ and the Brownian motion $W$.
\begin{proposition}\label{prop:yamada_watanabe}
  Suppose that a weak solution to \eqref{eq:young_fbm_appendix} exists and that the pathwise uniqueness holds for \eqref{eq:young_fbm_appendix}.
  Then, there exists a strong solution to \eqref{eq:young_fbm_appendix}.
\end{proposition}
\begin{proof}
  We would like to apply \cite[Theorem~3.14]{kurtz07}. For this purpose, we need a setup. We follow the notation in \cite{kurtz07}.
  We fix $\beta > 0$ that is less than but sufficiently close to $\frac{1}{2}$.
  As before we set $K_H(t, r) \defby (t-r)_+^{H - \frac{1}{2}} - (-r)_+^{H - \frac{1}{2}}$ and we set $S_1 \defby C^{\alpha}([0, T])$ and define
  $S_2$ as a subspace of
  \begin{equation*}
    \set{w \in C^{\beta}(\R) \given \lim_{r \to -\infty} \abs{w_r} (-r)^{H - \frac{3}{2}} = 0, \,\, \int_{-\infty}^{-1} \abs{w(r)} (-r)^{H - \frac{3}{2}} \dd r < \infty}
  \end{equation*}
  that is Polish and the Brownian motion lives in $S_2$.
  We note that for $w \in S_2$, the improper integral
  \begin{equation*}
    \int_{-\infty}^t K_H(t, r) \dd w_r = \lim_{M \to \infty} \int_{-M}^t K_H(t, r) \dd w_r
  \end{equation*}
  is well-defined.
  For $t \in [0, T]$, we denote by $(\cB^{S_1}_t)_{t \in [0, T]}$ and $(\cB^{S_2}_t)_{t \in [0, T]}$ the filtration generated by
  the coordinate maps in $S_1$ and $S_2$ respectively. We set
  \begin{equation*}
    \cC \defby \set{(\cB^{S_1}_t, \cB^{S_2}_t) \given t \in [0, T] }
  \end{equation*}
  as our compatibility structure in the sense of \cite[Definition~3.4]{kurtz07}.
  We denote by $\cS_{\Gamma, \cC, W}$ the set of probability measures $\mu$ on $S_1 \times S_2$ such that
  \begin{itemize}
    \item we have
    \begin{equation*}
      \mu(\set{(x,y) \in S_1 \times S_2 \given x_t = x + \int_0^t b(x_r) \dd r + \int_0^t \sigma(x_r) \dd Iy_r\,\, \text{for all }t \in [0,T]} ) = 1,
    \end{equation*}
    where $(Iy)_t \defby \int_{-\infty}^t K_H(t, r) \dd y_r$ ;
    \item $\mu$ is $\cC$-compatible in the sense of \cite[Definition~3.6]{kurtz07};
    \item $\mu(S_1 \times \cdot)$ has the law of the Brownian motion.
  \end{itemize}
  By \cite[Lemma~3.8]{kurtz07}, $\cS_{\Gamma, \cC, W}$ is convex. In view of \cite[Lemma~3.2]{kurtz07},
  the existence of weak solutions implies $\cS_{\Gamma, \cC, W}
  \neq \emptyset$.

  Therefore, to apply \cite[Theorem~3.14]{kurtz07}, it remains to prove the pointwise uniqueness in the sense of \cite[Definition~3.12]{kurtz07}.
  Suppose that $(X_1, X_2, W)$ are defined on a common probability space, that the laws of $(X_1, Y)$ and $(X_2, Y)$ belong to $\cS_{\Gamma, \cC, W}$
  and that $(X_1, X_2)$ are jointly compatible with $W$ in the sense of \cite[Definition~3.12]{kurtz07}.
  But then, if we denote by $(\cF_t)$ the filtration generated by $(X_1, X_2, W)$,
  by \cite[Lemma~3.2]{kurtz07}, the joint compatibility implies that $W$ is an $(\cF_t)$-Brownian motion,
  and therefore the pathwise uniqueness implies $X_1 = X_2$ almost surely.

  Hence, by \cite[Theorem~3.14]{kurtz07}, there exists a measurable map $F: S_2 \to S_1$ such that for a Brownian motion $W$,
  the law of $(F(W), W)$ belongs to $\cS_{\Gamma, \cC, W}$. Then, \cite[Lemma~3.11]{kurtz07} implies that $F(W)$ is a strong solution.
\end{proof}
\begin{lemma}\label{lem:weak_solution}
  Let $b \in C_b^1(\R^d)$ and $\sigma \in C^1_b(\R^d)$. Then, there exists a weak solution to \eqref{eq:young_fbm_appendix}.
\end{lemma}
\begin{proof}
  Let $(\sigma^n)_{n \in \N}$ be a smooth approximation to $\sigma$ and let $X^n$ be the solution to
  \begin{equation*}
    X^n_t = x + \int_0^t b(X^n_r) \dd r + \int_0^t \sigma^n(X^n_r) \dd B_r.
  \end{equation*}
  Let $W$ be the Brownian motion such that $B_t = \int_{-\infty}^t K_H(t, r) \dd W_r$.
  Let $\epsilon$ be greater than but sufficiently close to $0$ and let $S$ be
  a subspace of
  \begin{equation*}
    \set{w \in C^{\frac{1}{2} - \epsilon}(\R) \given \lim_{r \to -\infty} \abs{w_r} (-r)^{H - \frac{3}{2}} = 0, \,\, \int_{-\infty}^{-1} \abs{w(r)} (-r)^{H - \frac{3}{2}} \dd r < \infty}
  \end{equation*}
  that is Polish and where the Brownian motion lives.
  By the a priori estimate \eqref{eq:X_a_priori}, we see that a sequence of the laws of $(X^n, W)$ is tight in $C^{H - \epsilon}([0, T]) \times S$.
  Thus, replacing it with a subsequence, we suppose that the sequence $(X^n, W)$ converge to some limit $(\tilde{X}, \tilde{W})$ in law.

  To see that $(\tilde{X}, \tilde{W})$ solves \eqref{eq:young_fbm},
  we write $I w_t \defby \int_{-\infty}^t K_H(t, r) \dd w_r$ and for $\delta>0$ we set
  \begin{equation*}
    A_{\delta} \defby
    \set{(y, w) \in C^{H-\epsilon}([0, T]) \times S \given \sup_{t \in [0, T]} \abs{y_t - x - \int_0^t b(y_r) \dd r - \int_0^t \sigma(y_r) \dd (Iw)_r} > \delta}.
  \end{equation*}
  Then, we have
  \begin{align*}
    \P((\tilde{X}, \tilde{W}) \in A_{\delta})
    &\leq \liminf_{n \to \infty} \P((X^n, W) \in A_{\delta}) \\
    &\leq \liminf_{n \to \infty} \P(\sup_{t \in [0, T]} \abs{\int_0^t \{\sigma - \sigma^n\} (X^n_r) \dd B_r} > \delta).
  \end{align*}
  However, by the estimate of Young's integral,
  \begin{equation*}
    \abs{\int_0^t \{\sigma - \sigma^n\} (X^n_r) \dd B_r}
    \lesssim_{H, \epsilon} \norm{\sigma - \sigma^n}_{C^1_b} \norm{X^n}_{C^1_b} \norm{B}_{C^1_b}.
  \end{equation*}
  Thus, combined with the a priori estimate \eqref{eq:X_a_priori}, we observe
  \begin{equation*}
    \liminf_{n \to \infty} \P(\sup_{t \in [0, T]} \abs{\int_0^t \{\sigma - \sigma^n\} (X^n_r) \dd B_r} > \delta)
    = 0
  \end{equation*}
  and hence $\P((\tilde{X}, \tilde{W}) \in A_{\delta}) = 0$.
  Since $\delta$ is arbitrary, this implies
  \begin{equation*}
    \P(\tilde{X}_t = x + \int_0^t b(\tilde{X}_r) \dd r + \int_0^t \sigma(\tilde{X}_r) \dd (I \tilde{W})_r
    \,\, \forall t \in [0, T]) = 1.
  \end{equation*}

  Finally, since $W_t - W_s$ is independent of the $\sigma$-algebra generated by
  $(X^n_r)_{r \leq s}$ and $(W_r)_{r \leq s}$, we know that $\tilde{W}_t - \tilde{W}_s$ is independent of
  \begin{equation*}
    \tilde{\cF}_s \defby \sigma(\tilde{X}_r, \tilde{W}_r \vert r \leq s),
  \end{equation*}
  or equivalently $\tilde{W}$ is $(\tilde{\cF}_t)$-Brownian motion.
\end{proof}
\printbibliography[heading=bibintoc]

@misc{anzeletti21,
  doi = {10.48550/ARXIV.2112.05685},
  url = {https://arxiv.org/abs/2112.05685},
  author = {Anzeletti, Lukas and Richard, Alexandre and Tanré, Etienne},
  title = {Regularisation by fractional noise for one-dimensional differential equations with nonnegative distributional drift},
  publisher = {arXiv},
  primaryClass={math.PR},
  year = {2021},
  eprint={2112.05685},
}

@misc{athreya2021wellposedness,
      title={Well-posedness of stochastic heat equation with distributional drift and skew stochastic heat equation},
      author={Siva Athreya and Oleg Butkovsky and Khoa Lê and Leonid Mytnik},
      year={2021},
      eprint={2011.13498},
      archivePrefix={arXiv},
      primaryClass={math.PR}
}

@article {ayache08,
    AUTHOR = {Ayache, Antoine and Wu, Dongsheng and Xiao, Yimin},
     TITLE = {Joint continuity of the local times of fractional {B}rownian
              sheets},
   JOURNAL = {Ann. Inst. Henri Poincar\'{e} Probab. Stat.},
  FJOURNAL = {Annales de l'Institut Henri Poincar\'{e} Probabilit\'{e}s et
              Statistiques},
    VOLUME = {44},
      YEAR = {2008},
    NUMBER = {4},
     PAGES = {727--748},
      ISSN = {0246-0203},
   MRCLASS = {60G22 (60J55)},
  MRNUMBER = {2446295},
MRREVIEWER = {Anne Estrade},
       DOI = {10.1214/07-AIHP131},
       URL = {https://doi.org/10.1214/07-AIHP131},
}

@article{azais90,
  TITLE = {{Conditions for convergence of number of crossings to the local time. Application to stable processes with independent increments and to gaussian processes}},
  AUTHOR = {Azaïs, Jean-Marc},
  URL = {https://hal.inrae.fr/hal-02717652},
  JOURNAL = {{Probability and Mathematical Statistics (Pol)}},
  VOLUME = {11},
  PAGES = {19-36},
  YEAR = {1990},
}

@book{Biagini2008,
  title="Stochastic Calculus for Fractional Brownian Motion and Applications",
  author={Francesca Biagini and Yaozhong Hu and Bernt {\O}ksendal and Tusheng Zhang},
  year="2008",
  publisher="Springer London",
  address="London",
  isbn="978-1-84628-797-8",
  doi="10.1007/978-1-84628-797-8_10",
  url="https://doi.org/10.1007/978-1-84628-797-8_10"
}

@inproceedings{BDG72,
    AUTHOR = {Burkholder, D. L. and Davis, B. J. and Gundy, R. F.},
     TITLE = {Integral inequalities for convex functions of operators on
              martingales},
 BOOKTITLE = {Proceedings of the {S}ixth {B}erkeley {S}ymposium on
              {M}athematical {S}tatistics and {P}robability ({U}niv.
              {C}alifornia, {B}erkeley, {C}alif., 1970/1971), {V}ol. {II}:
              {P}robability theory},
     PAGES = {223--240},
      YEAR = {1972},
   MRCLASS = {60G15},
  MRNUMBER = {0400380},
MRREVIEWER = {Maurizio Pratelli},
}

@misc{butkovsky2023stochastic,
      title={Stochastic equations with singular drift driven by fractional Brownian motion}, 
      author={Oleg Butkovsky and Khoa Lê and Leonid Mytnik},
      year={2023},
      eprint={2302.11937},
      archivePrefix={arXiv},
      primaryClass={math.PR}
}

@article{catellier16,
    AUTHOR = {Catellier, R. and Gubinelli, M.},
     TITLE = {Averaging along irregular curves and regularisation of {ODE}s},
   JOURNAL = {Stochastic Process. Appl.},
  FJOURNAL = {Stochastic Processes and their Applications},
    VOLUME = {126},
      YEAR = {2016},
    NUMBER = {8},
     PAGES = {2323--2366},
      ISSN = {0304-4149},
   MRCLASS = {34C29 (34F05 60G17 60G22 60H10)},
  MRNUMBER = {3505229},
MRREVIEWER = {Fuke Wu},
       DOI = {10.1016/j.spa.2016.02.002},
       URL = {https://doi.org/10.1016/j.spa.2016.02.002},
}

@article {cont19,
    AUTHOR = {Cont, Rama and Perkowski, Nicolas},
     TITLE = {Pathwise integration and change of variable formulas for
              continuous paths with arbitrary regularity},
   JOURNAL = {Trans. Amer. Math. Soc. Ser. B},
  FJOURNAL = {Transactions of the American Mathematical Society. Series B},
    VOLUME = {6},
      YEAR = {2019},
     PAGES = {161--186},
   MRCLASS = {60H05 (60H99)},
  MRNUMBER = {3937343},
MRREVIEWER = {Torstein K. Nilssen},
       DOI = {10.1090/btran/34},
       URL = {https://doi.org/10.1090/btran/34},
}

@article{coutin02,
	author = {Coutin, Laure and Qian, Zhongmin},
	date = {2002/01/01},
	doi = {10.1007/s004400100158},
	id = {Coutin2002},
	isbn = {1432-2064},
	journal = {Probability Theory and Related Fields},
	number = {1},
	pages = {108--140},
	title = {Stochastic analysis, rough path analysis and fractional Brownian motions},
	url = {https://doi.org/10.1007/s004400100158},
	volume = {122},
	year = {2002}}

@article{davie_2010,
  title={Differential Equations Driven by Rough Paths: An Approach via Discrete Approximation},
  ISSN={1687-1200},
  url={https://dx.doi.org/10.1093/amrx/abm009},
  DOI={10.1093/amrx/abm009},
  journal={Applied Mathematics Research eXpress},
  publisher={Applied Mathematics Research eXpress},
  author={Davie, A. M.},
  year={2010}
}

@article {Davis18,
    AUTHOR = {Davis, Mark and Ob\l \'{o}j, Jan and Siorpaes, Pietro},
     TITLE = {Pathwise stochastic calculus with local times},
   JOURNAL = {Ann. Inst. Henri Poincar\'{e} Probab. Stat.},
  FJOURNAL = {Annales de l'Institut Henri Poincar\'{e} Probabilit\'{e}s et
              Statistiques},
    VOLUME = {54},
      YEAR = {2018},
    NUMBER = {1},
     PAGES = {1--21},
      ISSN = {0246-0203},
   MRCLASS = {60G17 (60H05 60J55)},
  MRNUMBER = {3765878},
MRREVIEWER = {Peter Imkeller},
       DOI = {10.1214/16-AIHP792},
       URL = {https://doi.org/10.1214/16-AIHP792},
}

@article {feyel06,
    AUTHOR = {Feyel, Denis and de La Pradelle, Arnaud},
     TITLE = {Curvilinear integrals along enriched paths},
   JOURNAL = {Electron. J. Probab.},
  FJOURNAL = {Electronic Journal of Probability},
    VOLUME = {11},
      YEAR = {2006},
     PAGES = {no. 34, 860--892},
      ISSN = {1083-6489},
   MRCLASS = {60G17 (60G15 60H10)},
  MRNUMBER = {2261056},
MRREVIEWER = {Hong Xue},
       DOI = {10.1214/EJP.v11-356},
       URL = {https://doi.org/10.1214/EJP.v11-356},
}

@book{friz2020,
author="Friz, Peter K.
and Hairer, Martin",
title="A Course on Rough Paths: With an Introduction to Regularity Structures",
year="2020",
publisher="Springer International Publishing",
address="Cham",
isbn="978-3-030-41556-3",
doi="10.1007/978-3-030-41556-3_1",
url="https://doi.org/10.1007/978-3-030-41556-3_1"
}

@article {geman80,
    AUTHOR = {Geman, Donald and Horowitz, Joseph},
     TITLE = {Occupation densities},
   JOURNAL = {Ann. Probab.},
  FJOURNAL = {The Annals of Probability},
    VOLUME = {8},
      YEAR = {1980},
    NUMBER = {1},
     PAGES = {1--67},
      ISSN = {0091-1798},
   MRCLASS = {60J55 (26A27 60G15 60G17)},
  MRNUMBER = {556414},
MRREVIEWER = {Simeon M. Berman},
}

@article{Gerencser2022,
      title={Regularisation by regular noise},
      author={Máté Gerencsér},
      year={2022},
      journal={Stoch. PDE: Anal. Comp.},
      fjournal={Stochastics and Partial Differential Equations: Analysis and Computations},
}

@article {gubinelli04,
    AUTHOR = {Gubinelli, M.},
     TITLE = {Controlling rough paths},
   JOURNAL = {J. Funct. Anal.},
  FJOURNAL = {Journal of Functional Analysis},
    VOLUME = {216},
      YEAR = {2004},
    NUMBER = {1},
     PAGES = {86--140},
      ISSN = {0022-1236},
   MRCLASS = {60H05 (26A42 60H10)},
  MRNUMBER = {2091358},
MRREVIEWER = {Vigirdas Mackevi\v{c}ius},
       DOI = {10.1016/j.jfa.2004.01.002},
       URL = {https://doi.org/10.1016/j.jfa.2004.01.002},
}

@article{hairer_theory_2014,
	title = {A theory of regularity structures},
	volume = {198},
	issn = {0020-9910, 1432-1297},
	url = {http://arxiv.org/abs/1303.5113},
	doi = {10.1007/s00222-014-0505-4},
	number = {2},
	journal = {Invent. Math.},
	author = {Hairer, Martin},
	year = {2014},
	pages = {269--504},
}

@article{hinz2020variability,
  title={Variability of paths and differential equations with $ BV $-coefficients},
  author={Hinz, Michael and T{\"o}lle, Jonas M and Viitasaari, Lauri},
  journal={arXiv preprint arXiv:2003.11698},
  year={2020}
}

@misc{kern2021stochastic,
      title={A stochastic reconstruction theorem},
      author={Hannes Kern},
      year={2021},
      eprint={2107.03867},
      archivePrefix={arXiv},
      primaryClass={math.PR}
}

@incollection{karoui78,
     author = {El Karoui, Nicole},
     title = {Sur les mont\'ees des semi-martingales},
     booktitle = {Temps locaux},
     series = {Ast\'erisque},
     publisher = {Soci\'et\'e math\'ematique de France},
     number = {52-53},
     year = {1978},
     language = {fr},
     url = {http://www.numdam.org/item/AST_1978__52-53__63_0/}
}

@article{kim_2022,
	author = {Kim, Donghan},
	date = {2022/12/01},
	doi = {10.1007/s10959-022-01159-z},
	id = {Kim2022},
	isbn = {1572-9230},
	journal = {Journal of Theoretical Probability},
	number = {4},
	pages = {2540--2568},
	title = {Local Times for Continuous Paths of Arbitrary Regularity},
	url = {https://doi.org/10.1007/s10959-022-01159-z},
	volume = {35},
	year = {2022}}

@article{khoa20,
    AUTHOR = {L\^{e}, Khoa},
     TITLE = {A stochastic sewing lemma and applications},
   JOURNAL = {Electron. J. Probab.},
  FJOURNAL = {Electronic Journal of Probability},
    VOLUME = {25},
      YEAR = {2020},
     PAGES = {Paper No. 38, 55},
   MRCLASS = {60H10 (60H05 60L20)},
  MRNUMBER = {4089788},
MRREVIEWER = {Torstein K. Nilssen},
       DOI = {10.1214/20-ejp442},
       URL = {https://doi.org/10.1214/20-ejp442},
}

@article{kurtz07,
   author = {Kurtz, Thomas G.},
   title = {The Yamada-Watanabe-Engelbert theorem for general stochastic equations and inequalities},
   journal = {Electron. J. Probab},
   volume = {12},
   pages = {951-965},
   DOI = {10.1214/ejp.v12-431.full},
   url = {https://dx.doi.org/10.1214/ejp.v12-431.full},
   year = {2007},
   type = {Journal Article}
}

@article{le21,
	author = {Khoa Lê},
	doi = {10.1214/23-EJP918},
	journal = {Electronic Journal of Probability},
	pages = {1 -- 22},
	publisher = {Institute of Mathematical Statistics and Bernoulli Society},
	title = {{Stochastic sewing in Banach spaces}},
	url = {https://doi.org/10.1214/23-EJP918},
	volume = {28},
	year = {2023}}

@mastersthesis{lemieux83,
  author  = "Marc Lemieux",
  title   = "On the quadratic variation of semi-martingales",
  school  = "University of British Columbia",
  year    = "1983",
  note = "supervised by Edwin A. Perkins",
  URL = {https://circle.ubc.ca/handle/2429/23964}
}

@article{Lochowski21,
  author = {Rafał M. {\L}ochowski and Jan Obłój and David J. Prömel and Pietro Siorpaes},
  title = {{Local times and Tanaka–Meyer formulae for càdlàg paths}},
  volume = {26},
  journal = {Electron. J. Probab.},
  journal = {Electronic Journal of Probability},
  number = {none},
  publisher = {Institute of Mathematical Statistics and Bernoulli Society},
  pages = {1 -- 29},
  year = {2021},
  doi = {10.1214/21-EJP638},
  URL = {https://doi.org/10.1214/21-EJP638}
}

@article{lyons94,
    AUTHOR = {Lyons, Terry},
     TITLE = {Differential equations driven by rough signals. {I}. {A}n
              extension of an inequality of {L}. {C}. {Y}oung},
   JOURNAL = {Math. Res. Lett.},
  FJOURNAL = {Mathematical Research Letters},
    VOLUME = {1},
      YEAR = {1994},
    NUMBER = {4},
     PAGES = {451--464},
      ISSN = {1073-2780},
   MRCLASS = {60H10},
  MRNUMBER = {1302388},
MRREVIEWER = {Eckhard Platen},
       DOI = {10.4310/MRL.1994.v1.n4.a5},
       URL = {https://doi.org/10.4310/MRL.1994.v1.n4.a5},
}

@article {lyons1998,
    AUTHOR = {Lyons, Terry J.},
     TITLE = {Differential equations driven by rough signals},
   JOURNAL = {Rev. Mat. Iberoamericana},
  FJOURNAL = {Revista Matem\'{a}tica Iberoamericana},
    VOLUME = {14},
      YEAR = {1998},
    NUMBER = {2},
     PAGES = {215--310},
      ISSN = {0213-2230},
   MRCLASS = {60H10 (34A99 34F05 37H99)},
  MRNUMBER = {1654527},
MRREVIEWER = {Ben Hambly},
       DOI = {10.4171/RMI/240},
       URL = {https://doi.org/10.4171/RMI/240},
}

@article{mandelbrot68,
	author = {Mandelbrot, Benoit B. and Van Ness, John W.},
	doi = {10.1137/1010093},
	eprint = {https://doi.org/10.1137/1010093},
	journal = {SIAM Review},
	number = {4},
	pages = {422-437},
	title = {Fractional Brownian Motions, Fractional Noises and Applications},
	url = {https://doi.org/10.1137/1010093},
	volume = {10},
	year = {1968}}

@book{mörters_peres_2010,
  place={Cambridge},
  series={Cambridge Series in Statistical and Probabilistic Mathematics},
  title={Brownian Motion},
  DOI={10.1017/CBO9780511750489},
  publisher={Cambridge University Press},
  author={Mörters, Peter and Peres, Yuval},
  year={2010},
  collection={Cambridge Series in Statistical and Probabilistic Mathematics}
}

@article{mukeru17,
   author = {Mukeru, Safari},
   title = {Representation of local times of fractional Brownian motion},
   journal = {Stat. Probab. Lett.},
   fjournal = {Statistics \& Probability Letters},
   volume = {131},
   pages = {1-12},
   DOI = {https://doi.org/10.1016/j.spl.2017.07.018},
   url = {https://www.sciencedirect.com/science/article/abs/pii/S0167715217302596?via%3Dihub},
   year = {2017},
   type = {Journal Article}
}

@book {nualart06,
    AUTHOR = {Nualart, David},
     TITLE = {The {M}alliavin calculus and related topics},
    SERIES = {Probability and its Applications (New York)},
   EDITION = {Second},
 PUBLISHER = {Springer-Verlag, Berlin},
      YEAR = {2006},
     PAGES = {xiv+382},
      ISBN = {978-3-540-28328-7; 3-540-28328-5},
   MRCLASS = {60-02 (60H07 60H30)},
  MRNUMBER = {2200233},
MRREVIEWER = {Daniel\ Ocone},
}

@book {nourdin12,
    AUTHOR = {Nourdin, Ivan},
     TITLE = {Selected aspects of fractional {B}rownian motion},
    SERIES = {Bocconi \& Springer Series},
    VOLUME = {4},
 PUBLISHER = {Springer, Milan; Bocconi University Press, Milan},
      YEAR = {2012},
     PAGES = {x+122},
      ISBN = {978-88-470-2822-7; 978-88-470-2823-4},
   MRCLASS = {60G15 (46L54 60F05 60G22 60H07)},
  MRNUMBER = {3076266},
MRREVIEWER = {Laurent Decreusefond},
       DOI = {10.1007/978-88-470-2823-4},
       URL = {https://doi.org/10.1007/978-88-470-2823-4},
}

@article{perkowski15,
author = {Nicolas Perkowski and David Prömel},
title = {{Local times for typical price paths and pathwise Tanaka formulas}},
volume = {20},
journal = {Electron. J. Probab.},
fjournal = {Electronic Journal of Probability},
number = {none},
publisher = {Institute of Mathematical Statistics and Bernoulli Society},
pages = {1 -- 15},
year = {2015},
doi = {10.1214/EJP.v20-3534},
URL = {https://doi.org/10.1214/EJP.v20-3534}
}

@article{picard08,
   author = {Picard, Jean},
   title = {A tree approach to p-variation and to integration},
   journal = {Ann. Probab.},
   fjournal = {The Annals of Probability},
   volume = {36},
   number = {6},
   pages = {2235-2279},
   ISSN = {0091-1798},
   DOI = {10.1214/07-aop388},
   url = {https://dx.doi.org/10.1214/07-aop388},
   year = {2008},
   type = {Journal Article}
}

@Inbook{picard11,
author="Picard, Jean",
editor="Donati-Martin, Catherine
and Lejay, Antoine
and Rouault, Alain",
title="Representation Formulae for the Fractional Brownian Motion",
bookTitle="S{\'e}minaire de Probabilit{\'e}s XLIII",
year="2011",
publisher="Springer Berlin Heidelberg",
address="Berlin, Heidelberg",
pages="3--70",
isbn="978-3-642-15217-7",
doi="10.1007/978-3-642-15217-7_1",
url="https://doi.org/10.1007/978-3-642-15217-7_1"
}

@mastersthesis{wuermli80,
  author  = "M. Wuermli",
  title   = "Lokalzeiten für Martingale",
  school  = "Universität Bonn",
  year    = "1980",
  note = "supervisedby Hans Föllmer"
}

@article{yaskov18,
    author= {Yaskov, Pavel},
     title= {Extensions of the sewing lemma with applications},
   journal= {Stochastic Process. Appl.},
  fjournal= {Stochastic Processes and their Applications},
    volume= {128},
      year= {2018},
    NUMBER = {11},
     PAGES = {3940--3965},
      ISSN = {0304-4149},
   MRCLASS = {28A25 (60G17 60G22 60H05)},
  MRNUMBER = {3860015},
       DOI = {10.1016/j.spa.2017.09.023},
       URL = {https://doi.org/10.1016/j.spa.2017.09.023},
}
\end{document}